\documentclass[a4paper,12pt,reqno]{amsart}

\usepackage{amsmath,amssymb,amsthm}
\usepackage{latexsym}
\usepackage{color}
\usepackage{graphicx}
\usepackage{mathrsfs}
\usepackage{enumerate}
\usepackage[abbrev]{amsrefs}
\usepackage[T1]{fontenc}
\usepackage{mathtools}
\mathtoolsset{showonlyrefs=true}

\allowdisplaybreaks

\allowdisplaybreaks[4]

\theoremstyle{plain}
\newtheorem{thm}{Theorem}[section]
\newtheorem{lemm}[thm]{Lemma}

\theoremstyle{definition}
\newtheorem{df}[thm]{Definition}
\newtheorem{rem}[thm]{Remark}

\makeatletter

\@addtoreset{equation}{section}
\makeatother

\renewcommand{\div}{\operatorname{div}}
\newcommand{\dB}{\dot{B}}

\newcommand{\supp}{\operatorname{supp}}

\begin{document}
\title[Low Mach number limit for the compressible flow]
{Low Mach number limit of the global solution to the compressible Navier--Stokes system for large data in the critical Besov space}
\author{Mikihiro Fujii}
\address{Institute of Mathematics for Industry Kyushu University, Fukuoka 819--0275, Japan}
\email{fujii.mikihiro.096@m.kyushu-u.ac.jp}
\begin{abstract}
In this paper, we consider the compressible Navier--Stokes system around the constant equilibrium states and prove the existence of a unique global solution for {\it arbitrarily large initial data} in the scaling critical Besov space provided that the Mach number is sufficiently small and the incompressible part of the initial velocity generates the global solution of the incompressible Navier--Stokes equation.
Moreover, we consider the low Mach number limit and show that the compressible solution converges to the solution of the incompressible Navier--Stokes equation in some space time norms.
\end{abstract}
\maketitle

\section{Introduction}\label{sec:intro}
We consider the initial value problem for the compressible Navier--Stokes system,
describing the motion of the barotropic fluid in the whole space $\mathbb{R}^d$ with $d \geqslant 2$:
\begin{align}\label{eq:comp-1}
    \begin{dcases}
    \partial_t \rho^{\varepsilon} + \div (\rho^{\varepsilon} u^{\varepsilon}) = 0, & \qquad t>0,x \in\mathbb{R}^d,\\
    \rho^{\varepsilon}(\partial_t u^{\varepsilon} +(u^{\varepsilon} \cdot \nabla)u^{\varepsilon} ) + \frac{1}{\varepsilon^2}\nabla P(\rho^{\varepsilon}) = \mathcal{L}u^{\varepsilon}, & \qquad t>0,x \in\mathbb{R}^d,\\
    \rho^{\varepsilon}(0,x)=\rho^{\varepsilon}_0(x),\quad u^{\varepsilon}(0,x)=u_0(x), & \qquad x \in \mathbb{R}^d,
    \end{dcases}
\end{align}
where two unknown functions $\rho^{\varepsilon}=\rho^{\varepsilon}(t,x):(0,\infty) \times \mathbb{R}^d \to (0,\infty)$ and $u^{\varepsilon}=u^{\varepsilon}(t,x):(0,\infty) \times \mathbb{R}^d \to \mathbb{R}^d$ represent the density of the fluid and the velocity field of the fluid.
The given real valued smooth function $P:(0,\infty) \to \mathbb{R}$ represents the pressure of the fluid
and satisfies $c_{\infty}:=P'(\rho_{\infty})>0$ with some positive constant $\rho_{\infty}$.
For the viscosity term $\mathcal{L}u$ in the second equation of \eqref{eq:comp-1},  $\mathcal{L}=\mu \Delta + (\mu+\lambda)\nabla \div$ is called the Lam\'e operator, where the constant coefficients $\mu$ and $\lambda$ satisfy the ellipticity condition $\mu>0$ and $\nu:=2\mu+\lambda>0$.
The constant $\varepsilon>0$ is called the Mach number, which plays a significant role in this article.
We assume that the given initial density $\rho^{\varepsilon}_0=\rho^{\varepsilon}_0(x):\mathbb{R}^d \to (0,\infty)$ satisfies $\rho^{\varepsilon}_0(x)=\rho_{\infty} + \varepsilon \rho_{\infty} a_0(x)$ with some bounded function $a_0=a_0(x):\mathbb{R}^d \to \mathbb{R}$ which is independent of $\varepsilon$, and the given initial velocity $u_0=u_0(x):\mathbb{R}^d \to \mathbb{R}^d$ is also independent of $\varepsilon$.
If $(\rho^{\varepsilon},u^{\varepsilon})$ is a solution to the system \eqref{eq:comp-1},
then the scaled functions 
\begin{align}\label{scaling}
    \rho^{\varepsilon}_{\alpha}(t,x) = \rho^{\varepsilon}( \alpha^2 t , \alpha x),\qquad
    u^{\varepsilon}_{\alpha}(t,x) = \alpha u^{\varepsilon}( \alpha^2 t , \alpha x)
\end{align}
also solve to \eqref{eq:comp-1} with $P$ replaced by $\alpha^2P$ for all $\alpha > 0$. 
Then, the critical Besov spaces corresponding to the invariant scaling \eqref{scaling} are given by 
$\dB_{p,\sigma}^{\frac{d}{p}}(\mathbb{R}^d) \times \dB_{p,\sigma}^{\frac{d}{p}-1}(\mathbb{R}^d)^d$.

The aim of this paper is to construct the global solution around the constant equilibrium state 
for given {\it large} initial perturbation in the critical Besov space $(\dB_{2,1}^{\frac{d}{2}-1}(\mathbb{R}^d) \cap \dB_{2,1}^{\frac{d}{2}}(\mathbb{R}^d)) \times  \dB_{2,1}^{\frac{d}{2}-1}(\mathbb{R}^d)^d$ 
if 
the Mach number $\varepsilon$ is sufficiently small and 
there exists a solution $w$ of the incompressible Navier--Stokes equation \eqref{main:eq-incomp} below with the initial data $\mathbb{P}u_0$.
Here, 
\begin{align}
    \mathbb{P}:=I+\nabla\div(-\Delta)^{-1}
\end{align}
denotes the Helmholtz projection on to the divergence free vector fields.
We note that our result implies that the global solution of the weakly compressible Navier--Stokes equation uniquely exists for arbitrarily {\it large} data in the critical regularity setting for the two dimensional case.
(See Remark \ref{rem:thm} (1) below.)
Moreover, we also consider the low Mach number singular limit $\varepsilon \downarrow 0$ and show that the solution $(\rho^{\varepsilon},u^{\varepsilon})$ converges to the incompressible flow $(\rho_{\infty},w)$ in  some space time norms.

Before we state our main result precisely, we first reformulate the equation \eqref{eq:comp-1} and recall previous studies related to our topic.
By virtue of the scaling transforms
\begin{gather}
	\rho^{\varepsilon} (t, x) \mapsto \frac1{\rho_\infty} \rho^{\varepsilon} \bigg(\frac{\nu t}{\rho_\infty c_{\infty} } , \frac{\nu x}{\rho_\infty \sqrt{c_{\infty}}}\bigg), \quad
	u^{\varepsilon} (t, x) \mapsto \frac{1}{\sqrt{c_{\infty}}} u^{\varepsilon} \bigg(\frac{\nu t}{\rho_\infty c_{\infty} } , \frac{\nu x}{\rho_\infty \sqrt{c_{\infty}}}\bigg),\\
    \mu \mapsto \frac{\mu}{\nu},\quad
    \lambda \mapsto \frac{\lambda}{\nu},\quad 
    P(\rho^{\varepsilon}) \mapsto \frac{1}{\rho_{\infty}c_{\infty}}P(\rho_{\infty}\rho^{\varepsilon}),
\end{gather}
we may assume that 
\begin{align}
    \rho_{\infty} = 1,\quad
    c_{\infty} = P'(1) = 1,\quad 
    \nu = 1.
\end{align}
Let $a^{\varepsilon}$ satisfy $\rho^{\varepsilon}=1+\varepsilon a^{\varepsilon}$ that is $a^{\varepsilon}=(\rho^{\varepsilon}-1)/\varepsilon$.
Then, $(a^{\varepsilon},u^{\varepsilon})$ should solve
\begin{align}\label{eq:re_comp-1}
    \begin{dcases}
    \partial_t a^{\varepsilon} + \frac{1}{\varepsilon} \div u^{\varepsilon} = -\div (a^{\varepsilon}u^{\varepsilon}), & \qquad t>0,x \in\mathbb{R}^d,\\
    \partial_t u^{\varepsilon} - \mathcal{L}u^{\varepsilon}  + \frac{1}{\varepsilon}\nabla a^{\varepsilon} = -\mathcal{N}^{\varepsilon}[a^{\varepsilon},u^{\varepsilon}], & \qquad t>0,x \in\mathbb{R}^d,\\
    a^{\varepsilon}(0,x)=a_0(x),\quad u^{\varepsilon}(0,x)=u_0(x), & \qquad x \in \mathbb{R}^d.
    \end{dcases}
\end{align}
Here, $\mathcal{N}^{\varepsilon}[a^{\varepsilon},u^{\varepsilon}]$ represents the nonlinear terms defined as
\begin{align}
    \mathcal{N}^{\varepsilon}[a,u]
    :=
    (u \cdot \nabla)u
    +
    \mathcal{J}(\varepsilon a) \mathcal{L}u
    +
    \frac{1}{\varepsilon}
    \mathcal{K}(\varepsilon a) \nabla a,
\end{align}
where
\begin{align}
    \mathcal{J}(a):= \frac{a}{1+a},\qquad
    \mathcal{K}(a):= \frac{P'(1+a)}{1+a}-1.
\end{align}
It was Danchin \cites{Dan-01-L,Dan-00-G} who first proved the local and global well-posedness of the compressible Navier--Stokes system \eqref{eq:re_comp-1} with $\varepsilon=1$ in the critical regularity setting.
These results are improved by many researchers.
For the local well-posedness, Danchin \cites{Dan-01-L, Dan-05-U, Dan-14} and Chen--Miao--Zhang \cite{Che-Mia-Zha-10} showed it for large initial data $(a_0,u_0)$ in the critical Besov spaces $\dB_{p,1}^{\frac{d}{p}}(\mathbb{R}^d) \times  \dB_{p,1}^{\frac{d}{p}-1}(\mathbb{R}^d)^d$ with $1 \leqslant p < 2d$.
Here, Chen--Miao--Zhang \cite{Che-Mia-Zha-15} and Iwabuchi--Ogawa \cite{Iw-Og-22} proved the range $1 \leqslant p < 2d$ of the integrable exponent is optimal by virtue of the ill-posedness argument.
In the case of the global well-posedness, Danchin \cite{Dan-00-G} proved it for small data in  $(\dB_{2,1}^{\frac{d}{2}-1}(\mathbb{R}^d) \cap \dB_{2,1}^{\frac{d}{2}}(\mathbb{R}^d)) \times  \dB_{2,1}^{\frac{d}{2}-1}(\mathbb{R}^d)^d$.
Charve--Danchin \cite{Cha-Dan-10} and Haspot \cite{Has-11-2} considered the $L^p$ setting and
showed the global well-posedness of \eqref{eq:re_comp-1} with $\varepsilon=1$ for initial data satisfying
\begin{align}
    \| ( a_0,  \mathbb{Q}u_0 ) \|_{\dB_{2,1}^{\frac{d}{2}-1}}^{\ell;\beta_0}
    +
    \| a_0 \|_{\dB_{p,1}^{\frac{d}{p}}}^{h;\beta_0}
    +
    \| \mathbb{Q}u_0 \|_{\dB_{p,1}^{\frac{d}{p}-1}}^{h;\beta_0}
    +
    \| \mathbb{P}u_0 \|_{\dB_{p,1}^{\frac{d}{p}-1}}
    \ll
    1
\end{align}
for some $\beta_0 >0$,
where $p$ satisfies $2 \leqslant p \leqslant \min \{ 4, 2d/(d-2)\}$ for $d \geqslant 3$ and $2\leqslant p<4$ for $d=2$. Here,
\begin{align}
    \mathbb{Q} 
    := I - \mathbb{P}
    = -\nabla \div (- \Delta)^{-1}
\end{align}
is the projection onto the curl free vector fields.
We give the definition of the truncated Besov norms $\| \cdot \|_{\dB_{p,\sigma}^{s}}^{\ell;\beta_0}$ and $\| \cdot \|_{\dB_{p,\sigma}^{s}}^{h;\beta_0}$ in Section \ref{sec:pre} below.
Next, we review the results on the incompressible limit of the solution to \eqref{eq:comp-1} in the critical regularity setting.
Danchin \cite{Dan-02-R} and Bahouri--Chemin--Danchin \cite{Bah-Che-Dan-11} proved that the global solution converges to the incompressible flow in the low Mach number limit $\varepsilon \downarrow 0$ in the critical setting $(\dB_{2,1}^{\frac{d}{2}-1}(\mathbb{R}^d) \cap \dB_{2,1}^{\frac{d}{2}}(\mathbb{R}^d)) \times  \dB_{2,1}^{\frac{d}{2}-1}(\mathbb{R}^d)^d$ and the subcritical setting.
The corresponding result on the periodic boundary condition case is shown in \cite{Dan-02-T}.
Danchin--He \cite{Dan-He-16} extended the results of \cites{Dan-02-R,Bah-Che-Dan-11} to the critical $L^p$ framework.
See \cites{Dan-Muc-17,Dan-Muc-19,Che-Zha-19} for the incompressible limits of $\lambda \to \infty$ type.
See \cites{Dan-05-L,Dan-Ducom-16,Des-Gre-Lio-Mas-99,Des-Gre-99,Fei-09,Has-20,Ha-Lo-98,Hoff,Kre-Lor-Nau-91} also for other results on the incompressible limits. 
Here, we focus on the results of \cites{Dan-02-R,Bah-Che-Dan-11}.
In \cite{Bah-Che-Dan-11}*{Theorem 10.29} and \cite{Dan-02-R}, they proved that for the initial data $(a_0,u_0)$ satisfying 
\begin{align}
    \left\| \left(a_0,u_0\right) \right\|_{\dB_{2,1}^{\frac{d}{2}-1}}^{\ell;\frac{\beta_0}{\varepsilon}}
    +
    \varepsilon\| a_0 \|_{\dB_{2,1}^{\frac{d}{2}}}^{h,\frac{\beta_0}{\varepsilon}}
    \ll 
    1,
\end{align}
the global solution $(a,u)$ converges to the incompressible flow of \eqref{main:incomp} below as $\varepsilon \downarrow 0$
in the following sense:
\begin{align}
    \left\| \left(a^{\varepsilon}, \mathbb{Q}u^{\varepsilon} \right) \right\|_{L^2(0,\infty;\dB_{p,1}^{\frac{d}{p}-\frac{1}{2}})}
    +
    \left\| \mathbb{P}u^{\varepsilon}-w \right\|_{L^{\infty}(0,\infty;\dB_{p,1}^{\frac{d}{p}-\frac{3}{2}}) \cap L^1(0,\infty;\dB_{p,1}^{\frac{d}{p}+\frac{1}{2}})}
    \leqslant
    C
    \varepsilon^{\frac{1}{2}}
\end{align}
for  $d \geqslant 4$ and $2(d-1)/(d-3) \leqslant p \leqslant \infty$,
\begin{align}
    \left\| \left(a^{\varepsilon}, \mathbb{Q}u^{\varepsilon} \right) \right\|_{L^{\frac{2p}{p-2}}(0,\infty;\dB_{p,1}^{\frac{2}{p}-\frac{1}{2}})}
    +
    \left\| \mathbb{P}u^{\varepsilon}-w \right\|_{L^{\infty}(0,\infty;\dB_{p,1}^{\frac{4}{p}-\frac{3}{2}}) \cap L^1(0,\infty;\dB_{p,1}^{\frac{4}{p}+\frac{1}{2}})}
    \leqslant
    C
    \varepsilon^{\frac{1}{2}-\frac{1}{p}}
\end{align}
for $d=3$ and $2 \leqslant p < \infty$  
and
\begin{align}
    \left\| \left(a^{\varepsilon}, \mathbb{Q}u^{\varepsilon} \right) \right\|_{L^{\frac{4p}{p-2}}(0,\infty;\dB_{p,1}^{\frac{3}{2p}-\frac{3}{4}})}
    +
    \left\| \mathbb{P}u^{\varepsilon} - w \right\|_{L^{\infty}(0,\infty;\dB_{p,1}^{\frac{5}{2p}-\frac{5}{4}}) \cap L^1(0,\infty;\dB_{p,1}^{\frac{5}{2p}+\frac{3}{4}})}
    \leqslant
    C
    \varepsilon^{\frac{1}{4}-\frac{1}{2p}}
\end{align}
for $d=2$ and $2 \leqslant p \leqslant 6$.
Moreover, it is also considered the scaling subcritical case and shown in \cite{Bah-Che-Dan-11}*{Theorem 10.31} and \cite{Dan-02-R} that for {\it large} initial data 
$a_0 \in \dB_{2,1}^{\frac{d}{2}-1}(\mathbb{R}^d) \cap \dB_{2,1}^{\frac{d}{2}+\alpha}(\mathbb{R}^d)$
and 
$u_0 \in \dB_{2,1}^{\frac{d}{2}-1}(\mathbb{R}^d)^d \cap \dB_{2,1}^{\frac{d}{2}-1+\alpha}(\mathbb{R}^d)^d$
with some $0<\alpha \leqslant 1/6$,
the unique global large solution $(\rho^{\varepsilon},u^{\varepsilon})$ exists provided that the Mach number is sufficiently small $0 < \varepsilon \ll 1$ and the incompressible part of the initial data admits the global incompressible flow $w$ of \eqref{main:eq-incomp} below. 
Moreover, the large solution converges to the incompressible flow $w$ in the following sense:
\begin{align}
    \| \mathbb{P}u^{\varepsilon} - w \|_{
    L^{\infty}(0,\infty ;\dB_{2,1}^{\frac{d}{2}-1} \cap \dB_{2,1}^{\frac{d}{2}-1+\alpha})
    \cap 
    L^1(0,\infty ;\dB_{2,1}^{\frac{d}{2}+1} \cap \dB_{2,1}^{\frac{d}{2}+1+\alpha})}
    \leqslant
    C 
    \varepsilon^{\frac{2\alpha}{2 + d + 2\alpha}}
\end{align}
and 
\begin{align}
    \left\| \left(a^{\varepsilon}, \mathbb{Q}u^{\varepsilon} \right) \right\|_{L^{p}(0,\infty;\dB_{\infty,1}^{\alpha-1+\frac{1}{p}})}
    \leqslant
    C \varepsilon^{\frac{1}{p}},
\end{align}
where $p=2$ if $d\geqslant 4$, $2<p<\infty$ if $d=3$ and $p=4$ if $d=2$.
Note that the proof of the result for large data fails in the critical case $\alpha=0$.

The aim of the present paper is to improve the above result on the incompressible limit for large data and show the existence of the global {\it large} solution $(a^{\varepsilon},u^{\varepsilon})$ of \eqref{eq:re_comp-1} in the scaling {\it critical} regularity setting $(\dB_{2,1}^{\frac{d}{2}}(\mathbb{R}^d) \cap \dB_{2,1}^{\frac{d}{2}-1}(\mathbb{R}^d)) \times \dB_{2,1}^{\frac{d}{2}-1}(\mathbb{R}^d)^d$ 
if the Mach number is sufficiently small and the incompressible part $\mathbb{P}u_0$ of the initial velocity generates the global solution $w$ of the incompressible Navier--Stokes equation in the weaker regularity class $\dB_{q,1}^{\frac{d}{q}-1}(\mathbb{R}^d)^d$ for some $q>2$.
Moreover, we also show that the large compressible solution $(a^{\varepsilon},u^{\varepsilon})$ converges to the incompressible solution $(0,w)$ in the sense of the low regularity setting with the convergence rate:
\begin{align}
    &\left\| \left(a^{\varepsilon},\mathbb{Q}u^{\varepsilon}\right) \right\|_{{L^r}(0,\infty;\dB_{q,1}^{\frac{d}{q}-1+\frac{1}{r}})}
    +
    \| \mathbb{P}u^{\varepsilon}-w \|_{{L^{\infty}}(0,\infty;\dB_{q,1}^{\frac{d}{q}-1-\frac{1}{r}}) \cap L^1(
    0,\infty;\dB_{q,1}^{\frac{d}{q}+1-\frac{1}{r}})}
    \leqslant
    C 
    \varepsilon^{\frac{1}{r}}
\end{align}
and the scaling critical regularity setting:
\begin{align}
    \lim_{\varepsilon \to +0}
    \left(
    \| (a^{\varepsilon},\mathbb{Q}u^{\varepsilon}) \|_{L^r(0,\infty;\dB_{q,1}^{\frac{d}{q}-1+\frac{2}{r}})}
    +
    \| \mathbb{P}u^{\varepsilon}-w \|_{L^{\infty}(0,\infty;\dB_{q,1}^{\frac{d}{q}-1}) \cap L^1(
    0,\infty;\dB_{q,1}^{\frac{d}{q}+1})}
    \right)
    =0
\end{align}
for some $2 < q,r < \infty$.


Now, we state our main theorem.  
Note that we provide the definitions of the function spaces, appearing in the following theorem, in Section \ref{sec:pre}.
The main result of this paper reads as follows.
\begin{thm}\label{thm:2}
Let $q$ and $r$ satisfy
\begin{align}\label{main:q}
    2 < q < \min \left\{ 4, \frac{2d}{d-2}\right\}
\end{align}
and
\begin{gather}\label{main:r}
    0 < \frac{1}{r} \leqslant{} \min \left\{ \frac{1}{2} - \frac{d}{2} \left( \frac{1}{2} - \frac{1}{q} \right), \frac{d-1}{2}\left( \frac{1}{2} - \frac{1}{q} \right) \right\} ,\quad
    \frac{1}{r} < \frac{2d}{q} - 1 .
\end{gather}
Let the intial data $(a_0,u_0)$ satisfy the following two assumptions:
\begin{itemize}
    \item [{\bf (A1)}] $(a_0,u_0) \in (\dB_{2,1}^{\frac{d}{2}-1}(\mathbb{R}^d) \cap \dB_{2,1}^{\frac{d}{2}}(\mathbb{R}^d)) \times  \dB_{2,1}^{\frac{d}{2}-1}(\mathbb{R}^d)^d$.
    \item [{\bf (A2)}] There exists a global solution 
    \begin{align}\label{main:incomp-reg}
        w \in C([0,\infty);\dB_{q,1}^{\frac{d}{q}-1}(\mathbb{R}^d))^d \cap L^1(0,\infty;\dB_{q,1}^{\frac{d}{q}+1}(\mathbb{R}^d))^d
    \end{align}
    to the following incompressible Navier--Stokes equation:
    \begin{align}\label{main:eq-incomp}
    \begin{dcases}
    \partial_t w - \mu \Delta w + \mathbb{P}(w \cdot \nabla)w = 0, & \qquad t > 0, x \in\mathbb{R}^d,\\
    \div w = 0, & \qquad t \geqslant 0, x \in\mathbb{R}^d,\\
    w(0,x)= \mathbb{P}u_0(x), & \qquad x \in \mathbb{R}^d.
    \end{dcases}
\end{align}
\end{itemize}
Then, there exists
a positive constant
$\varepsilon_0=\varepsilon_0(d,q,r,\mu,P,a_0,u_0)$
such that 
for any
$0 < \varepsilon \leqslant \varepsilon_0$,
\eqref{eq:re_comp-1} possesses a unique global solution $(a^{\varepsilon},u^{\varepsilon})$ in the class
\begin{align}\label{main:class}
    \begin{split}
    &a^{\varepsilon} \in C([0,\infty); \dB_{2,1}^{\frac{d}{2}-1}(\mathbb{R}^d) \cap \dB_{2,1}^{\frac{d}{2}}(\mathbb{R}^d)),\\
    &u^{\varepsilon} \in C([0,\infty); \dB_{2,1}^{\frac{d}{2}-1}(\mathbb{R}^d))^d \cap L^1(0,\infty; \dB_{2,1}^{\frac{d}{2}+1}(\mathbb{R}^d))^d
    \end{split}
\end{align}
with $\rho^{\varepsilon}(t,x) = 1 + \varepsilon a^{\varepsilon}(t,x) > 0$.
Moreover, there exists a positive constant $C=C(d,q,r,\mu,P,a_0,u_0)$ such that
\begin{align}
    &\| (a^{\varepsilon},\mathbb{Q}u^{\varepsilon}) \|_{L^r(0,\infty;\dB_{q,1}^{\frac{d}{q}-1+\frac{1}{r}})}
    +
    \| \mathbb{P}u^{\varepsilon}-w \|_{L^{\infty}(0,\infty;\dB_{q,1}^{\frac{d}{q}-1-\frac{1}{r}}) \cap L^1(
    0,\infty;\dB_{q,1}^{\frac{d}{q}+1-\frac{1}{r}})}
    \leqslant
    C 
    \varepsilon^{\frac{1}{r}},\quad \label{main:incomp}
\end{align}
for all $0<\varepsilon \leqslant \varepsilon_0$
and
it also holds
\begin{align}
    \lim_{\varepsilon \to +0}
    \left(
    \| (a^{\varepsilon},\mathbb{Q}u^{\varepsilon}) \|_{L^r(0,\infty;\dB_{q,1}^{\frac{d}{q}-1+\frac{2}{r}})}
    +
    \| \mathbb{P}u^{\varepsilon}-w \|_{L^{\infty}(0,\infty;\dB_{q,1}^{\frac{d}{q}-1}) \cap L^1(
    0,\infty;\dB_{q,1}^{\frac{d}{q}+1})}
    \right)
    =0.\qquad\label{main:incomp-critical}
\end{align}
\end{thm}

\begin{rem}\label{rem:thm}
We mention some remarks on Theorem \ref{thm:2}.
\begin{itemize}
    \item [(1)]
    In the two dimensional case $d=2$,
    the assumption {\bf (A2)} is removable.
    Indeed, it is shown in \cite{Dan-Muc-17} that
    for any large initial data $\mathbb{P}u_0 \in \dB_{2,1}^{0}(\mathbb{R}^2)^2$, 
    there exists a unique solution $w$ of \eqref{main:eq-incomp} with $d=2$ in the class 
    \begin{align}
        w \in C([0,\infty);\dB_{2,1}^0(\mathbb{R}^2))^2 \cap L^1(0,\infty;\dB_{2,1}^2(\mathbb{R}^2))^2,
    \end{align}
    which is a stronger regularity class than \eqref{main:incomp-reg}.
    \item [(2)]
    In the case of $d \geqslant 3$, it is known that there exists a unique global solution $w$ of \eqref{main:incomp} in the class \eqref{main:incomp-reg} if the initial data satisfies 
    \begin{align}\label{rem:small}
        \| \mathbb{P}u_0 \|_{\dB_{q,1}^{\frac{d}{q}-1}} \leqslant \eta_q
    \end{align}
    for some positive constant $\eta_q$. 
    See \cite{Bah-Che-Dan-11}*{Theorem 5.10} for the proof.
    Therefore, if the incompressible part $\mathbb{P}u_0$ of the initial velocity is sufficiently small in $\dB_{q,1}^{\frac{d}{q}-1}(\mathbb{R}^d)$, then the assumption {\bf (A2)} is satisfied.
    Here, we note that by the continuous embedding $\dB_{2,1}^{\frac{d}{2}-1}(\mathbb{R}^d) \hookrightarrow \dB_{q,1}^{\frac{d}{q}-1}(\mathbb{R}^d)$, the smallness condition \eqref{rem:small} may hold even if $\mathbb{P}u_0$ is not necessarily small in $\dB_{2,1}^{\frac{d}{2}-1}(\mathbb{R}^d)$.
    Indeed, if $d=3$, $3<q<4$ and the incompressible part of the initial data is given by 
    \begin{align}
        \mathbb{P}u_0(x):=\sin(Mx_3)\big(-\partial_{x_2}\phi(x), \partial_{x_1}\phi(x),0\big),
    \end{align}
    where $M>0$ and a real valued Schwartz function $\phi \in \mathscr{S}(\mathbb{R}^3) \setminus \{ 0 \}$ whose Fourier transform has a compact support,
    then we see that
    \begin{align}
        \| \mathbb{P} u_0 \|_{\dB_{2,1}^{\frac{1}{2}}} \sim M^{\frac{1}{2}} \gg 1,
        \qquad
        \| \mathbb{P} u_0 \|_{\dB_{q,1}^{\frac{3}{q}-1}} \sim M^{\frac{3}{q}-1} \ll 1
    \end{align}
    as $M \gg 1$.
    \item [(3)]
    In contrast to the previous results of the low Mach number limit stated in \cites{Dan-02-R,Bah-Che-Dan-11,Dan-He-16}, our method to obtain \eqref{main:incomp} does not require different arguments for different spatial dimensions.
    The key ingredient of the proof is to control $\mathbb{P}(w \cdot \nabla)(u^{\varepsilon} - w)$ in the equation of $\mathbb{P}u^{\varepsilon} - w$ by using the energy method and a commutator estimate. 
    See Lemmas \ref{lemm:lin-P-2}, \ref{lemm:nonlin-3} and \ref{lemm:lim} for the detail.
    \item [(4)]
    The norm of the second incompressible limit \eqref{main:incomp-critical} is the scaling critical,
    whereas
    the previous results \cites{Dan-02-R,Bah-Che-Dan-11,Dan-He-16} for the low Mach number limit for the critical regularity solutions dealt with only low regularity setting, as is mentioned above.
    Therefore, the limit result \eqref{main:incomp-critical} is one of the novelties of this paper.
\end{itemize}
\end{rem}


Let us mention the idea of the proof of Theorem \ref{thm:2}.
Let $q$ and $r$ satisfy \eqref{main:q} and \eqref{main:r}.
Let $\delta_0>0$ be sufficiently small constant and let $w$ be the global solution to \eqref{main:eq-incomp}.
Then, for this $\delta_0$, we decompose the global time interval $[0,\infty)$ into finitely many parts $[T_0,T_1]$, $[T_1,T_2]$,...,$[T_{N_0-1},\infty)$ with 
\begin{align}\label{w-small}
    \| w \|_{L^r(T_{n-1},T_n; \dB_{q,1}^{\frac{d}{q}-1+\frac{2}{r}}) \cap L^1(T_{n-1},T_n; \dB_{q,1}^{\frac{d}{q}+1})} \leqslant \delta_0,\qquad
    n=1,2,...,N_0,
\end{align}
where $T_0=0$ and $T_{N_0}=\infty$.
We first construct the solution on the time interval $[0,T_1]$.
To this end, we consider the energy norm 
\begin{align}
    \| (a^{\varepsilon},u^{\varepsilon})\|_{X^{\varepsilon}(0,T)}
    :={}
    &
    \varepsilon 
    \| a^{\varepsilon} \|_{{L^{\infty}}(0,T;\dB_{2,1}^{\frac{d}{2}})}^{h;\frac{\beta_0}{\varepsilon}}
    +
    \frac{1}{\varepsilon}
    \| a^{\varepsilon} \|_{L^1(0,T;\dB_{2,1}^{\frac{d}{2}})}^{h;\frac{\beta_0}{\varepsilon}}\\
    &
    +
    \| a^{\varepsilon} \|_{
    {L^{\infty}}(0,T;\dB_{2,1}^{\frac{d}{2}-1})
    \cap 
    L^1(0,T;\dB_{2,1}^{\frac{d}{2}+1})}^{\ell;\frac{\beta_0}{\varepsilon}}\\
    &
    +
    \| u^{\varepsilon} \|_{
    {L^{\infty}}(0,T;\dB_{2,1}^{\frac{d}{2}-1})
    \cap
    L^1(0,T;\dB_{2,1}^{\frac{d}{2}+1})}.
\end{align}
In order to close the a priori estimate of the $X^{\varepsilon}(0,T_1)$-norm of the solution for large initial data, we introduce an auxiliary quantity $A_{q,r}^{\varepsilon,\alpha}[a,u](0,T_1)$ (defined in Section \ref{sec:a-propri}).
This quantity is bounded by $\delta_0$ for sufficiently small $\varepsilon$ by virtue of the choice of suitable $\alpha$, the energy estimates, the Strichartz estimates of the acoustic waves and \eqref{w-small}.
Then, the a priori estimate
\begin{align}
    \| (a^{\varepsilon},u^{\varepsilon}) \|_{X^{\varepsilon}(0,T)} 
    \leqslant{}& 
    C \| (a_0,u_0) \|_{(\dB_{2,1}^{\frac{d}{2}-1} \cap \dB_{2,1}^{\frac{d}{2}}) \times \dB_{2,1}^{\frac{d}{2}-1}}\\
    &
    + 
    C
    A_{q,r}^{\varepsilon,\alpha}[a^{\varepsilon},u^{\varepsilon}](0,T)\| (a^{\varepsilon},u^{\varepsilon})\|_{X^{\varepsilon}(0,T)}
\end{align}
and the continuation argument enable us to extend the local solution constructed in Lemma \ref{lemm:LWP} below to the time interval $[0,T_1]$.
Then, repeating $N_0$-th times of the similar procedures, we may extend the solution to the time intervals $[T_1,T_2]$,...,$[T_{N_0-1},\infty)$ and obtain the global large solution.
Next, we consider the low Mach number limit.
For the compressible part, following the ideas of \cites{Dan-02-R,Bah-Che-Dan-11,Dan-He-16}, we have 
\begin{align}
    \| (a^{\varepsilon},\mathbb{Q}u^{\varepsilon}) \|_{{L^r}(0,\infty;\dB_{q,1}^{\frac{d}{q}-1+\frac{1}{r}})}
    =O(\varepsilon^{\frac{1}{r}})
\end{align}
as $\varepsilon \to +0$ by the energy estimates for the high frequency part and the Strichartz estimates for the low frequency part.
The difficulty to obtain \eqref{main:incomp} arises from the incompressible part.
Let $v^{\varepsilon}:=u^{\varepsilon}-w$.
Then, the incompressible perturbation $\mathbb{P}v^{\varepsilon}=\mathbb{P}u^{\varepsilon}-w$ satisfies
\begin{align}\label{intro:eq:v}
    \partial_t \mathbb{P}v^{\varepsilon} 
    - 
    \mu \Delta \mathbb{P}v^{\varepsilon} 
    =
    -
    \mathbb{P}(w \cdot \nabla) v^{\varepsilon}
    -
    \mathbb{P}(v^{\varepsilon} \cdot \nabla) u^{\varepsilon}
    -
    \mathbb{P}(\mathcal{J}(\varepsilon a^{\varepsilon}) \mathcal{L}u^{\varepsilon}).
\end{align}
The term $\mathbb{P}(w \cdot \nabla) v^{\varepsilon}$ is the most harmful to obtain the appropriate estimate for $v^{\varepsilon}$.
To overcome this, we apply the energy method to \eqref{intro:eq:v} and use the commutator estimate for the harm term. (See Lemmas \ref{lemm:lin-P-2} and \ref{lemm:lim} below.)
This idea enables us to obtain the incompressible limit \eqref{main:incomp} by the same argument for $d=2$, $d=3$ and $d \geqslant 4$, differently from other previous results which needed specific arguments depending on $d$.
The other result \eqref{main:incomp-critical} of the incompressible limit in Theorem \ref{thm:2} is obtained by the a priori estimate for the solution of \eqref{intro:eq:v} and the estimates for $A_{q,r}^{\varepsilon,\alpha}[a^{\varepsilon},\mathbb{Q}u^{\varepsilon}](0,\infty)$.


This paper is organized as follows.
In Section \ref{sec:pre}, we prepare some notation and basic lemmas used in our analysis.
We mention linear and nonlinear estimates in Section \ref{sec:lin} and \ref{sec:nonlin}, respectively.
In Section \ref{sec:a-propri}, we establish some global a priori estimates of the solution.
Finally, we prove Theorem \ref{thm:2} in Section \ref{sec:pf}.


Throughout this paper, we denote by $C$ the constant, which may differ in each line. In particular, $C=C(a_1,...,a_n)$ means that $C$ depends only on $a_1,...,a_n$. We define a commutator for two operators $A$ and $B$ as $[A,B]=AB-BA$.
For two Banach spaces $X$ and $Y$ with $X \cap Y \neq \varnothing$, 
we write $\| \cdot \|_{X \cap Y} := \| \cdot \|_X + \| \cdot \|_Y$.
For a vector field $f=(f_1,...,f_d) \in X^d$, we set $\| f \|_{X}:= \| f_1\|_X + ... + \| f_d \|_X$.

\section{Preliminaries}\label{sec:pre}
In this section, we mention the definitions of function spaces and their basic properties related to our analysis. 

Let $\mathscr{S}(\mathbb{R}^d)$ be the set of all Schwartz functions on $\mathbb{R}^d$ and let $\mathscr{S}'(\mathbb{R}^d)$ be the set of all tempered distributions on $\mathbb{R}^d$.
We call the sequence $\{\phi_j\}_{j \in \mathbb{Z}} \subset \mathscr{S}(\mathbb{R}^d)$ as the Littlewood-Paley decomposition if it satisfies
\begin{align}
    0 \leqslant \widehat{\phi_0} \leqslant 1,
    \qquad
    \supp \widehat{\phi_0} \subset \{ \xi \in \mathbb{R}^d\ ;\ 2^{-1} \leqslant |\xi| \leqslant 2 \},
    \qquad
    \phi_j(x)=2^{dj}\phi_0(2^j x)
\end{align}
and 
\begin{align}
    \sum_{j \in \mathbb{Z}} \widehat{\phi_j}(\xi) = 1, \qquad \xi \in \mathbb{R}^d \setminus \{0 \}.
\end{align}
For each $j \in \mathbb{Z}$ and $f \in \mathscr{S}'(\mathbb{R}^d)$, we write 
\begin{align}
    \dot{\Delta}_j f := \phi_j * f,\qquad
    \dot{S}_j f := \sum_{j' \leqslant j} \dot{\Delta}_{j'}f.
\end{align}
Now we recall the definition of Besov spaces.
For $1\leqslant p,\sigma \leqslant \infty$ and $s \in \mathbb{R}$, we define the homogeneous Besov space $\dB_{p,\sigma}^s(\mathbb{R}^d)$ on $\mathbb{R}^d$ as 
\begin{align}
    \dB_{p,\sigma}^s(\mathbb{R}^d) 
    := &
    \left\{
    f \in \mathscr{S}'(\mathbb{R}^d)/\mathscr{P}(\mathbb{R}^d)
    \ ;\ 
    \| f \|_{\dB_{p,\sigma}^{s}}
    < 
    \infty
    \right\},\\
    \| f \|_{\dB_{p,\sigma}^s}
    := & 
    \left\| \{2^{sj} \| \dot{\Delta}_j f\|_{L^p} \}_{j \in \mathbb{Z}} \right\|_{\ell^{\sigma}},
\end{align}
where $\mathscr{P}(\mathbb{R}^d)$ denotes the set of all polynomials on $\mathbb{R}^d$.
It is well-known that if $s<d/p$ or $(s,\sigma)=(d/p,1)$, then $\dB_{p,\sigma}^s(\mathbb{R}^d)$ is identified  as 
\begin{align}
    \dB_{p,\sigma}^s(\mathbb{R}^d) 
    \sim 
    \left\{ 
    f \in \mathscr{S}'(\mathbb{R}^d) 
    \ ; \ 
    f = \sum_{j \in \mathbb{Z}}\dot{\Delta}_j f
    {\rm\ in\ }
    \mathscr{S}'(\mathbb{R}^d)
    {\rm\ and\ }
    \| f \|_{\dB_{p,\sigma}^s}
    < 
    \infty
    \right\}.
\end{align}
See \cite{Sa-18}*{Theorem 2.31} for the proof.
For $1 \leqslant p,q,r \leqslant \infty$, $s \in \mathbb{R}$ and an interval $I \subset \mathbb{R}$,
we also define the Chemin--Lerner space $\widetilde{L^r}(I;\dB_{p,\sigma}^s(\mathbb{R}^d))$ as
\begin{align}
    \widetilde{L^r}(I;\dB_{p,\sigma}^s(\mathbb{R}^d))
    := &
    \left\{
    F : I \to \mathscr{S}'(\mathbb{R}^d) / \mathscr{P}(\mathbb{R}^d)
    \ ;\ 
    \| F \|_{\widetilde{L^r}(I;\dB_{p,\sigma}^s)}
    <\infty
    \right\},\\
    \| F \|_{\widetilde{L^r}(I;\dB_{p,\sigma}^s)}
    := &
    \left\| \{2^{sj} \| \dot{\Delta}_j F\|_{L^r(I;L^p)} \}_{j \in \mathbb{Z}} \right\|_{\ell^{\sigma}}.
\end{align}
This space was first introduced by \cite{Che-Ler-95}.
We note that the Minkowski inequality yields
\begin{align}
    \widetilde{L^r}(I;\dB_{p,\sigma}^s(\mathbb{R}^d))
    \hookrightarrow
    L^r(I;\dB_{p,\sigma}^s(\mathbb{R}^d)),\qquad
    \sigma \leqslant r.
\end{align}
For $0<\alpha < \beta$, the truncated Besov semi-norms on the high, middle and low frequency parts are defined as
\begin{align}
    \| f \|_{\dB_{p,\sigma}^s}^{h;\beta}
    &:=
    \left\|
    \left\{ 
    2^{sj} 
    \| \dot{\Delta}_j f \|_{L^p}
    \right\}_{\beta \leqslant 2^j}
    \right\|_{\ell^{\sigma}},\\ 
    \| f \|_{\dB_{p,\sigma}^s}^{m;\alpha,\beta}
    &:=
    \left\|
    \left\{ 
    2^{sj} 
    \| \dot{\Delta}_j f \|_{L^p}
    \right\}_{\alpha \leqslant{} 2^j < \beta}
    \right\|_{\ell^{\sigma}},\\
    \| f \|_{\dB_{p,\sigma}^s}^{\ell;\alpha} 
    &:=
    \left\|
    \left\{ 
    2^{sj} 
    \| \dot{\Delta}_j f \|_{L^p}
    \right\}_{2^j < \alpha}
    \right\|_{\ell^{\sigma}}.
\end{align}
Similarly, we also define
\begin{align}
    \| F \|_{\widetilde{L^r}(I;\dB_{p,\sigma}^s)}^{h;\beta}
    :={}&
    \left\| \| F \|_{\dB_{p,\sigma}^s}^{h;\beta} \right\|_{\widetilde{L^r}(I)},\\
    \| F \|_{\widetilde{L^r}(I;\dB_{p,\sigma}^s)}^{m;\alpha,\beta}
    :={}&
    \left\| \| F \|_{\dB_{p,\sigma}^s}^{m;\alpha,\beta} \right\|_{\widetilde{L^r}(I)},\\
    \| F \|_{\widetilde{L^r}(I;\dB_{p,\sigma}^s)}^{\ell;\alpha} 
    :={}&
    \left\| \| F \|_{\dB_{p,\sigma}^s}^{\ell;\alpha}  \right\|_{\widetilde{L^r}(I)}
\end{align}
and
\begin{align}
    \| F \|_{\widetilde{L^r}(I;\dB_{p,\sigma}^s)}^{h;\beta}
    &:=
    \left\|
    \left\{ 
    2^{sj} 
    \| \dot{\Delta}_j F \|_{L^r(I;L^p)}
    \right\}_{\beta \leqslant 2^j}
    \right\|_{\ell^{\sigma}},\\ 
    \| F \|_{\widetilde{L^r}(I;\dB_{p,\sigma}^s)}^{m;\alpha,\beta}
    &:=
    \left\|
    \left\{ 
    2^{sj} 
    \| \dot{\Delta}_j F \|_{L^r(I;L^p)}
    \right\}_{\alpha \leqslant{} 2^j < \beta}
    \right\|_{\ell^{\sigma}},\\
    \| F \|_{\widetilde{L^r}(I;\dB_{p,\sigma}^s)}^{\ell;\alpha} 
    &:=
    \left\|
    \left\{ 
    2^{sj} 
    \| \dot{\Delta}_j F \|_{L^r(I;L^p)}
    \right\}_{2^j < \alpha}
    \right\|_{\ell^{\sigma}}.
\end{align}
We recall the maximal regularity estimates for the heat equation.
\begin{lemm}\cite{Bah-Che-Dan-11}*{Corollary 2.5}\label{lemm:lin-P-1}
Let $1 \leqslant p,r,\sigma \leqslant \infty$, $\mu>0$ and $s \in \mathbb{R}$
and consider the heat equation
\begin{align}
    \begin{dcases}
    \partial_t v -\Delta v = f, & \qquad t>0,x \in \mathbb{R}^d,\\
    v(0,x)=v_0(x), &\qquad x \in \mathbb{R}^d
    \end{dcases}
\end{align}
with
$v_0 \in \dB_{p,\sigma}^s(\mathbb{R}^d)$ and $f \in \widetilde{L^1}(0,T;\dB_{p,\sigma}^s(\mathbb{R}^d))$.
Then, there exists a positive constant $C=C(d,\mu)$ such that 
\begin{align}
    \| v \|_{\widetilde{L^r}(0,T;\dB_{p,\sigma}^{s+\frac{2}{r}})}
    &\leqslant
    C \| v_0 \|_{\dB_{p,\sigma}^s}
    +
    C \| f \|_{\widetilde{L^1}(0,T;\dB_{p,\sigma}^s)}.
\end{align}
\end{lemm}

Next, we consider product estimates of two functions in the Besov spaces.
To begin with, we introduce the following Bony para-product decomposition:
\begin{align}
    fg 
    = 
    \dot{T}_fg
    +\dot{R}(f,g)
    +\dot{T}_gf,
\end{align}
where
\begin{align}
    &
    \dot{T}_fg:=\sum_{j \in \mathbb{Z}}\dot{S}_{j-3}f\dot{\Delta}_jg
    =\sum_{j \in \mathbb{Z}}\left( \sum_{j' \leqslant j-3} \dot{\Delta}_{j'}f \right)\dot{\Delta}_jg,\\
    &
    \dot{R}(f,g):=\sum_{|j-j'| \leqslant 2}\dot{\Delta}_jf\dot{\Delta}_{j'}g.
\end{align}
It is well-known that the following estimates hold:
\begin{lemm}[\cite{Bah-Che-Dan-11}*{Theorems 2.47 and 2.52}]\label{lemm:para-0}
The following two propositions hold:
\begin{itemize}
\item [(1)]
Let $1 \leqslant{} p, p_1,p_2,\sigma \leqslant{} \infty$ and $s, s_1,s_2 \in \mathbb{R}$ satisfy
\begin{align}
    \frac{1}{p} = \frac{1}{p_1} + \frac{1}{p_2},\qquad
    s = s_1 + s_2,\qquad
    s_1 \leqslant{} 0.
\end{align}
Then, there exists a positive constant $C=C(d,s_1,s_2)$ such that
\begin{align}
     \| \dot{T}_fg \|_{\dB_{p,\sigma}^{s}}
     \leqslant{}
     C 
     \| f \|_{\dB_{p_1,1}^{s_1}}
     \| g \|_{\dB_{p_2,\sigma}^{s_2}}.
\end{align}
for all $f \in \dB_{p_1,1}^{s_1}(\mathbb{R}^d)$ and $g \in \dB_{p_2,\sigma}^{s_2} (\mathbb{R}^d)$.
\item [(2)]
Let $1 \leqslant{} p, p_1,p_2, \sigma,\sigma_1,\sigma_2 \leqslant{} \infty$ and $s, s_1,s_2 \in \mathbb{R}$ satisfy
\begin{gather}
    \frac{1}{p} = \frac{1}{p_1} + \frac{1}{p_2},\qquad
    \frac{1}{\sigma} \leqslant{} \frac{1}{\sigma_1} + \frac{1}{\sigma_2},\qquad
    s = s_1 + s_2 > 0.
\end{gather}
Then, there exists a positive constant $C=C(d,s_1,s_2)$ such that
\begin{align}
    \| \dot{R}(f,g) \|_{\dB_{p,\sigma}^s}
    \leqslant{}
    C
    \| f \|_{\dB_{p_1,\sigma_1}^{s_1}}
    \| g \|_{\dB_{p_2,\sigma_2}^{s_2}}
\end{align}
for all $f \in \dB_{p_1,\sigma_1}^{s_1}(\mathbb{R}^d)$ and $g \in \dB_{p_2,\sigma_2}^{s_2} (\mathbb{R}^d)$.
\end{itemize}
\end{lemm}
From Lemma \ref{lemm:para-0} and the continuous embedding $\dB_{p_k,1}^{\frac{d}{p_k}-\alpha_k}(\mathbb{R}^d) \hookrightarrow \dB_{\infty,1}^{-\alpha_k}(\mathbb{R}^d)$ ($k=1,2$), we immediately obtain the following product estimate:
\begin{lemm}\label{lemm:prod-1}
Let 
$1 \leqslant p, p_1, p_2, \sigma \leqslant \infty$,
$\alpha_1,\alpha_2 \geqslant 0$
and $s \in \mathbb{R}$ satisfy 
\begin{align}
    s + \frac{d}{p_1} > 0,\qquad
    \frac{1}{p} + \frac{1}{p_1} \leqslant 1.
\end{align}
Then, there exists a positive constant $C = C(d,p,p_1,p_2,\alpha_1,\alpha_2,\sigma)$
such that
\begin{align}
    \| fg \|_{\dB_{p,\sigma}^s}
    \leqslant
    C
    \bigg(
    \| f \|_{\dB_{p_1,1}^{\frac{d}{p_1}-\alpha_1}}
    \| g \|_{\dB_{p,\sigma}^{s+\alpha_1}}
    +
    \| f \|_{\dB_{p,\sigma}^{s+\alpha_2}}
    \| g \|_{\dB_{p_2,1}^{\frac{d}{p_2}-\alpha_2}}
    \bigg)
\end{align}
for all 
$f \in \dB_{p_1,1}^{\frac{d}{p_1}-\alpha_1}(\mathbb{R}^d) \cap \dB_{p,\sigma}^{s+\alpha_2}(\mathbb{R}^d)$ 
and 
$g \in \dB_{p,\sigma}^{s+\alpha_1}(\mathbb{R}^d) \cap \dB_{p_2,1}^{\frac{d}{p_2}-\alpha_2}(\mathbb{R}^d)$.
\end{lemm}
In the next two lemmas, we establish the para-product estimates for truncated Besov norms.
\begin{lemm}\label{lemm:para-1}
Let $1 \leqslant{} p, p_1,p_2,\sigma \leqslant{} \infty$ and $s, s_1,s_2 \in \mathbb{R}$ satisfy
\begin{align}
    \frac{1}{p} = \frac{1}{p_1} + \frac{1}{p_2},\qquad
    s = s_1 + s_2,\qquad
    s_1 \leqslant{} 0.
\end{align}
Then, there exists a positive constant $C=C(d,s_1,s_2)$ such that
\begin{align}\label{para-low}
     \| \dot{T}_fg \|_{\dB_{p,\sigma}^{s}}^{\ell;\beta}
     \leqslant{}
     C 
     \| f \|_{\dB_{p_1,1}^{s_1}}^{\ell;\beta}
     \| g \|_{\dB_{p_2,\sigma}^{s_2}}^{\ell; 4\beta}
\end{align}
for all $\beta > 0$ and all $f,g$ provided that the right hand side is finite.
\end{lemm}
\begin{proof}
Since it holds
\begin{align}
    \dot{\Delta}_j \dot{T}_fg = \sum_{|j'-j|\leqslant 2} \dot{\Delta}_j(\dot{S}_{j'-3}f\dot{\Delta}_{j'}g),
\end{align}
we have 
\begin{align}\label{para-pf-1}
    \| \dot{\Delta}_j \dot{T}_fg \|_{L^p}
    \leqslant{} & 
    C\sum_{|j'-j|\leqslant 2} \|\dot{S}_{j'-3}f\|_{L^{p_1}}\|\dot{\Delta}_{j'}g\|_{L^{p_2}}\\ 
    \leqslant{} &
    C\sum_{|j'-j|\leqslant 2} 
    \sum_{j'' \leqslant j'-3}
    \|\dot{\Delta}_{j''}f\|_{L^{p_1}}
    \|\dot{\Delta}_{j'}g\|_{L^{p_2}}\\
    \leqslant{} &
    C
    \sum_{j'' \leqslant j-1}
    2^{s_1j''}
    \|\dot{\Delta}_{j''}f\|_{L^{p_1}}
    \sum_{|j'-j|\leqslant 2} 
    2^{-s_1j'}
    \|\dot{\Delta}_{j'}g\|_{L^{p_2}}.
\end{align}
Thus, 
we see that
\begin{align}
    \| \dot{T}_fg \|_{\dB_{p,\sigma}^{s}}^{\ell;\beta}
    ={} & 
    \left\{
    \sum_{2^j < \beta}
    \left(
    2^{sj}
    \| \dot{\Delta}_j \dot{T}_fg \|_{L^p}
    \right)^{\sigma}
    \right\}^{\frac{1}{\sigma}}\\
    \leqslant{} &
    C
    \| f \|_{\dB_{p_1,1}^{s_1}}^{\ell;\frac{\beta}{2}}
    \left\{
    \sum_{2^j < \beta}\sum_{|j'-j|\leqslant 2} 
    (2^{s_2j'}
    \|\dot{\Delta}_{j'}g\|_{L^{p_2}})^{\sigma}
    \right\}^{\frac{1}{\sigma}}\\
    \leqslant{} & 
    \| f \|_{\dB_{p_1,1}^{s_1}}^{\ell;\beta}
    \| g \|_{\dB_{p_2,\sigma}^{s_2}}^{\ell; 4\beta},
\end{align}
which completes the proof.
\end{proof}

\begin{lemm}\label{lemm:para-2}
Let $1 \leqslant{} p, p_1,p_2,p_3,p_4, \sigma,\sigma_1,\sigma_2,\sigma_3,\sigma_4 \leqslant{} \infty$ and $s, s_1,s_2,s_3,s_4 \in \mathbb{R}$ satisfy
\begin{gather}
    \frac{1}{p} = \frac{1}{p_1} + \frac{1}{p_2} = \frac{1}{p_3} + \frac{1}{p_4},\qquad
    \frac{1}{\sigma} \leqslant{} \min \left\{ 
    \frac{1}{\sigma_1} + \frac{1}{\sigma_2},
    \frac{1}{\sigma_3} + \frac{1}{\sigma_4}\right\},\\
    s = s_1 + s_2 = s_3 + s_4 > 0.
\end{gather}
Then, there exists a positive constant $C=C(d,s_1,s_2,s_3,s_4)$ such that
\begin{align}
    \| \dot{R}(f,g) \|_{\dB_{p,\sigma}^s}
    \leqslant{}
    C\left(
    \| f \|_{\dB_{p_1,\sigma_1}^{s_1}}^{h;\frac{\beta}{4}}
    \| g \|_{\dB_{p_2,\sigma_2}^{s_2}}^{h;\beta}
    +
    \| f \|_{\dB_{p_3,\sigma_3}^{s_3}}^{\ell;4\beta}
    \| g \|_{\dB_{p_4,\sigma_4}^{s_4}}^{\ell;\beta}
    \right)
\end{align}
for all $\beta>0$, $f$ and $g$ provided that the right hand side is finite.
\end{lemm}
\begin{proof}
Since 
\begin{align}
    \dot{\Delta}_j \dot{R}(f,g) 
    = 
    \dot{\Delta}_j
    \sum_{j' \geqslant j-4}
    \sum_{|j' - j''| \leqslant 2}
    \dot{\Delta}_{j'}f
    \dot{\Delta}_{j''}g,
\end{align}
we see that
\begin{align}
    2^{sj}
    \| \dot{\Delta}_j \dot{R}(f,g) \|_{L^p}
    \leqslant {}&
    C
    \sum_{j' \geqslant j-4}
    2^{s(j-j')}
    \sum_{|j' - j''| \leqslant 2}
    2^{sj'}
    \|\dot{\Delta}_{j'}f\dot{\Delta}_{j''}g\|_{L^p}\\
    ={}&
    C
    \left(
    \left\{2^{-sj'}{\bf 1}_{\{ j' ; j' \geqslant -4 \}}\right\}_{j'}
    *
    \left\{
    \sum_{|j' - j''| \leqslant 2}
    2^{sj'}
    \|\dot{\Delta}_{j'}f\dot{\Delta}_{j''}g\|_{L^p}
    \right\}_{j'}
    \right)(j),
\end{align}
where $*$ stands for the convolution on $\mathbb{Z}$.
Hence, taking the $\ell^{\sigma}(\mathbb{Z})$-norm with respect to $j$ and using the Hausdorff--Young inequality for the discrete convolution, we have
\begin{align}
    \| \dot{R}(f,g) \|_{\dB_{p,\sigma}^s}
    \leqslant{} &
    C
    \left(\sum_{j' \geqslant -4}2^{-sj'}\right)
    \left\{
    \sum_{j' \in \mathbb{Z}}
    \left(
    2^{sj'}
    \sum_{|j' - j''| \leqslant 2}
    \|
    \dot{\Delta}_{j'}f
    \dot{\Delta}_{j''}g
    \|_{L^p}
    \right)^{\sigma}
    \right\}^{\frac{1}{\sigma}}\\
    \leqslant{} &
    C
    \left\{
    \sum_{2^{j'} < \beta}
    \left(
    \sum_{|j' - j''| \leqslant 2}
    2^{s_1j'}
    \|
    \dot{\Delta}_{j'}f
    \|_{L^{p_1}}
    2^{s_2j''}
    \|
    \dot{\Delta}_{j''}g
    \|_{L^{p_2}}
    \right)^{\sigma}
    \right\}^{\frac{1}{\sigma}}\\
    &
    +
    C
    \left\{
    \sum_{\beta \leqslant 2^{j'}}
    \left(
    \sum_{|j' - j''| \leqslant 2}
    2^{s_3j'}
    \|
    \dot{\Delta}_{j'}f
    \|_{L^{p_3}}
    2^{s_4j''}
    \|
    \dot{\Delta}_{j''}g
    \|_{L^{p_4}}
    \right)^{\sigma}
    \right\}^{\frac{1}{\sigma}}\\
    \leqslant{}&
    C\left(
    \| f \|_{\dB_{p_1,\sigma_1}^{s_1}}^{h;\frac{\beta}{4}}
    \| g \|_{\dB_{p_2,\sigma_2}^{s_2}}^{h;\beta}
    +
    \| f \|_{\dB_{p_3,\sigma_3}^{s_3}}^{\ell;4\beta}
    \| g \|_{\dB_{p_4,\sigma_4}^{s_4}}^{\ell;\beta}
    \right).
\end{align}
This completes the proof.
\end{proof}

Using Lemmas \ref{lemm:para-0} and \ref{lemm:para-1}, we obtain the following lemma:

\begin{lemm}\label{lemm:prod-2}
Let
$1 \leqslant q, \sigma \leqslant \infty$
and
$s, s_1, s_2, s_3, s_4 \in \mathbb{R}$
satisfy 
\begin{align}
2 \leqslant q \leqslant 4,\qquad
    s_1, s_4 \leqslant d \left( \frac{2}{q} - \frac{1}{2} \right),\qquad
    s = s_1 +s_2 = s_3 + s_4 > 0.
\end{align}
Then, there exists a positive constant $C = C(d,q,s,s_1,s_2,s_3,s_4,\sigma)$
such that
\begin{align}
    \| fg \|_{\dB_{2,\sigma}^{s - d( \frac{2}{q} - \frac{1}{2})}}^{\ell;\beta}
    \leqslant
    C
    \left(
    \| f \|_{\dB_{q,1}^{s_1}}^{\ell;\beta}
    \| g \|_{\dB_{q,\sigma}^{s_2}}^{\ell;4\beta}
    +
    \| f \|_{\dB_{q,\sigma}^{s_3}}
    \| g \|_{\dB_{q,1}^{s_4}}
    \right)
\end{align}
for all $0 < \beta \leqslant \infty$, $f$ and $g$ provided that the right hand side is finite.
\end{lemm}
\begin{rem}
    In the case of $\beta=\infty$, we see that
    \begin{align}
    \| fg \|_{\dB_{2,\sigma}^{s - d( \frac{2}{q} - \frac{1}{2})}}
    \leqslant
    C
    \left(
    \| f \|_{\dB_{q,1}^{s_1}}
    \| g \|_{\dB_{q,\sigma}^{s_2}}
    +
    \| f \|_{\dB_{q,\sigma}^{s_3}}
    \| g \|_{\dB_{q,1}^{s_4}}
    \right).
    \end{align}
    This estimate is used in the proof of Lemma \ref{lemm:nonlin-3}.
\end{rem}
\begin{proof}
Let $q^*$ satisfy $1/2 = 1/q + 1/q^*$.
Then, by Lemma \ref{lemm:para-1} and the embedding $\dB_{q,1}^{s_1}(\mathbb{R}^d) \hookrightarrow \dB_{q^*,1}^{s_1-d(\frac{2}{q}-\frac{1}{2})}(\mathbb{R}^d)$, we have
\begin{align}\label{pf:prod-1}
    \| \dot{T}_fg \|_{\dB_{2,\sigma}^{s-d(\frac{2}{q}-\frac{1}{2})}}^{\ell;\beta} 
    \leqslant 
    C
    \| f \|_{\dB_{q^*,1}^{s_1 - d(\frac{2}{q}-\frac{1}{2})}}^{\ell;\beta}
    \| g \|_{\dB_{q,\sigma}^{s_2}}^{\ell;4\beta}
    \leqslant
    C
    \| f \|_{\dB_{q,1}^{s_1}}^{\ell;\beta}
    \| g \|_{\dB_{q,\sigma}^{s_2}}^{\ell;4\beta}.
\end{align}
Similarly, we have
\begin{align}\label{pf:prod-2}
    \| \dot{T}_gf \|_{\dB_{2,\sigma}^{s-d(\frac{2}{q}-\frac{1}{2})}}^{\ell;\beta} 
    \leqslant
    C
    \| g \|_{\dB_{q,1}^{s_4}}^{\ell;\beta}
    \| f \|_{\dB_{q,\sigma}^{s_3}}^{\ell;4\beta}
    \leqslant
    C
    \| g \|_{\dB_{q,1}^{s_4}}
    \| f \|_{\dB_{q,\sigma}^{s_3}}.
\end{align}
It follows from Lemma \ref{lemm:para-0} and the embedding $\dB_{\frac{q}{2},\sigma}^s(\mathbb{R}^d) \hookrightarrow \dB_{2,\sigma}^{s-d(\frac{2}{q}-\frac{1}{2})}(\mathbb{R}^d)$ that 
\begin{align}\label{pf:prod-3}
    \| \dot{R}(f,g) \|_{\dB_{2,\sigma}^{s-d(\frac{2}{q}-\frac{1}{2})}}^{\ell;\beta}
    \leqslant
    C
    \| \dot{R}(f,g) \|_{\dB_{\frac{q}{2},\sigma}^{s}}
    \leqslant
    C
    \| f \|_{\dB_{q,\sigma}^{s_3}}
    \| g \|_{\dB_{q,1}^{s_4}}.
\end{align}
Combining \eqref{pf:prod-1}, \eqref{pf:prod-2} and \eqref{pf:prod-3}, we complete the proof.
\end{proof}

\begin{lemm}\label{lemm:comm}
Let 
$\eta = \eta (\xi)$ be a zero-th order homogeneous smooth function on $\mathbb{R}^d \setminus \{ 0 \}$.
Let
$1 \leqslant p, p_1, p_2, \sigma \leqslant \infty$,
$\alpha_1,\alpha_2 \geqslant 0$
and $s \in \mathbb{R}$ satisfy 
\begin{align}
    s + \frac{d}{p_1} > 0,\qquad
    \frac{1}{p} + \frac{1}{p_1} \leqslant 1.
\end{align}
Then, there exists a positive constant $C = C(d,\eta,p,p_1,p_2,\alpha_1,\alpha_2,\sigma,s)$
such that
\begin{align}
    &
    \left\| 
    \left\{
    2^{sj}
    \left\|[u\cdot \nabla, \eta(D)\dot{\Delta}_j]a\right\|_{L^p}
    \right\}_{j \in \mathbb{Z}}
    \right\|_{\ell^{\sigma}(\mathbb{Z})}\\
    &\qquad 
    \leqslant
    C
    \bigg(
    \| u \|_{\dB_{p_1,1}^{\frac{d}{p_1}+1-\alpha_1}}
    \| a \|_{\dB_{p,\sigma}^{s+\alpha_1}}
    +
    \| u \|_{\dB_{p,\sigma}^{s+1+\alpha_2}}
    \| a \|_{\dB_{p_2,1}^{\frac{d}{p_2}-\alpha_2}}
    \bigg)
\end{align}
for all 
$u \in \dB_{p_1,1}^{\frac{d}{p_1}+1-\alpha_1}(\mathbb{R}^d)^d \cap \dB_{p,\sigma}^{s+1+\alpha_2}(\mathbb{R}^d)^d$ 
and 
$a \in \dB_{p,\sigma}^{s+\alpha_1}(\mathbb{R}^d) \cap \dB_{p_2,1}^{\frac{d}{p_2}-\alpha_2}(\mathbb{R}^d)$.
Here, $\eta(D)$ denotes the Fourier-multiplier associated to $\eta=\eta(\xi)$.
\end{lemm}
We omit the proof of Lemma \ref{lemm:comm} since it is obtained by the quite similar argument as in \cites{Bah-Che-Dan-11,Dan-He-16}.
Next, we introduce a composition lemma.
\begin{lemm}\label{lemm:composition}
Let $F$ be an smooth function on an open interval $I \subset \mathbb{R}$ satisfying $(-R,R) \subset I$ with some $R>0$
and assume $F(0)=0$.
Then, for any $1 \leqslant{} p \leqslant{} \infty$, $1 \leqslant p_1 < \infty$ and $s \in \mathbb{R}$ satisfying
\begin{align}
    s + \frac{d}{p_1} > 0,\qquad
    \frac{1}{p} + \frac{1}{p_1} \leqslant 1,
\end{align}
there exists a positive constant $C=C(d,p,p_1,s,F,R)$ such that
\begin{align}\label{composition}
    \| F( \psi ) \|_{\dB_{p,1}^{s}} \leqslant{} C \bigg(1 + \| \psi \|_{\dB_{p_1,1}^{\frac{d}{p_1}}}\bigg) \| \psi \|_{\dB_{p,1}^s}
\end{align}
for all $\psi \in \dB_{p,1}^s(\mathbb{R}^d) \cap \dB_{p_1,1}^{\frac{d}{p_1}}(\mathbb{R}^d)$ satisfying $\| \psi \|_{L^{\infty}} < R$.  
\end{lemm}
\begin{proof}
Although the case of $p=p_1$ is proved in \cite{Dan-Xu-17}*{Proposition A.3} and the strategy of the proof of Lemma \ref{lemm:composition} is analogous,
we give the proof for the readers' convenience.
Since $F(0)=0$ and $F$ is smooth, the Taylor theorem implies that there exists a smooth function $G$ such that $G(0)=0$ and
\begin{align}
    F(\psi) = F'(0)\psi + G(\psi)\psi.
\end{align}
Hence, by the second estimate of Lemma \cite{Dan-Xu-17}*{the second assertion of Proposition A.1} with $p_2=q=p$, $\sigma_1=d/p_1$ and $\sigma_2=s$, we have
\begin{align}
    \| F( \psi ) \|_{\dB_{p,1}^{s}} 
    \leqslant{}& 
    |F'(0)|\cdot\| \psi \|_{\dB_{p,1}^{s}} 
    +
    C\| G(\psi) \|_{\dB_{p_1,1}^{\frac{d}{p_1}}}\| \psi \|_{\dB_{p,1}^{s}}.
\end{align}
Here, using the standard composition lemma (see \cite{Bah-Che-Dan-11}*{Theorem 2.61}), we see that
\begin{align}
    \| G(\psi) \|_{\dB_{p_1,1}^{\frac{d}{p_1}}}
    \leqslant 
    C
    \| \psi \|_{\dB_{p_1,1}^{\frac{d}{p_1}}}.
\end{align}
Thus, we complete the proof.
\end{proof}
Finally, we introduce the following lemma.
\begin{lemm}\label{lemm:time}
Let $1 \leqslant p < \infty$ and let $X$ be a Banach space.
Then, for every $f \in L^p(0,\infty;X) \setminus \{0\}$ and $0 < \delta < \| f \|_{L^p(0,\infty;X)}$, 
there exists a positive integer $N_{\delta}=N_{\delta}(p,X,f)$ and a time sequence $\{ T_n=T_n(p,X,f)\}_{n=0}^{N_{\delta}}$ such that
\begin{gather}
    0 = T_0 < T_1 < ... < T_{N_{\delta}-1} < T_{N_{\delta}} = \infty,\\
    \| f \|_{L^p(T_{n-1},T_n;X)} \leqslant \delta \qquad n=1,...,N_{\delta}.
\end{gather}
\end{lemm}

\begin{proof}
Let 
\begin{align}
    N_{\delta}:=
    \min \left\{ N \in \mathbb{N}\ ;\  N \geqslant \frac{\| f \|_{L^p(0,\infty;X)}^p}{\delta^p} \right\}
\end{align}
and determine the time sequence
inductively as follows:
\begin{align}
    \begin{dcases}
    T_0 := 0,\\
    T_n := \sup
    \bigg\{
    T \in (T_{n-1},\infty)\ ;\ 
    \| f \|_{L^p(T_{n-1},T;X)} \leqslant \delta
    \bigg\}, &\qquad
    n=1,...,N_{\delta}.
    \end{dcases}
\end{align}
Here, if $T_{n_0}=\infty$ for some $n_0 \leqslant N_{\delta}-1$, then we see that
\begin{align}
    \| f \|_{L^p(0,\infty;X)}^p
    ={}
    \sum_{n=1}^{n_0}
    \| f \|_{L^p(T_{n-1},T_n;X)}^p
    \leqslant{}
    n_0\delta^p,
\end{align}
which contradicts the definition of $N_{\delta}$ and thus we have $T_{n}<\infty$ for all $n=1,...,N_{\delta}-1$.
Then, we easily see that $0=T_0 <T_1 <...<T_{N_{\delta}-1}<\infty$ and
\begin{align}
    \| f \|_{L^p(T_{n-1},T_n;X)} = \delta, \qquad n=1,...,N_{\delta}-1.
\end{align}
On the other hand, 
for any $T \geqslant T_{N_{\delta}-1}$, we have
\begin{align}
    N_{\delta}\delta^p
    \geqslant{}&
    \| f \|_{L^p(0,\infty;X)}^p\\
    \geqslant{}&
    \| f \|_{L^p(0,T;X)}^p\\
    ={}&
    \sum_{n=1}^{N_{\delta}-1}
    \| f \|_{L^p(T_{n-1},T_n;X)}^p
    +
    \| f \|_{L^p(T_{N_{\delta}-1},T;X)}^p\\
    ={}&
    (N_{\delta}-1)\delta^p
    +
    \| f \|_{L^p(T_{N_{\delta}-1},T;X)}^p,
\end{align}
which implies $\| f \|_{L^p(T_{N_{\delta}-1},T;X)} \leqslant \delta$. 
As $T \geqslant T_{N_{\delta}-1}$ may be taken arbitrarily large,
we obtain $T_{N_{\delta}}=\infty$.
This completes the proof.
\end{proof}

\section{Linear analysis}\label{sec:lin}
In this section, we consider the linearized system:
\begin{align}\label{eq:lin-1}
    \begin{dcases}
    \partial_t a + \frac{1}{\varepsilon} \div u = f, & \qquad t>0,x \in\mathbb{R}^d,\\
    \partial_t u - \mathcal{L}u  + \frac{1}{\varepsilon}\nabla a = g, & \qquad t>0,x \in\mathbb{R}^d,\\
    a(0,x)=a_0(x),\quad u(0,x)=u_0(x), & \qquad x \in \mathbb{R}^d.
    \end{dcases}
\end{align}
Here, $f$ and $g$ are given external forces.
First, we give an estimate for the incompressible part of the solution to \eqref{eq:lin-1}.
\begin{lemm}\label{lemm:lin-P-2}
There exists a positive constant $C = C(d,\mu)$ such that 
for any 
$2 \leqslant p < \infty$, 
$s \in \mathbb{R}$, 
an interval $I \subset [0,\infty)$ with the infimum $t_0$
and
the vector field $w$ satisfying $\div w =0$, 
the incompressible part $\mathbb{P}u$ of the solution to \eqref{eq:lin-1} satisfies
\begin{align}
    &
    \| \mathbb{P}u \|_{\widetilde{L^{\infty}}(I;\dB_{p,1}^{s})\cap L^1(I;\dB_{p,1}^{s+2})}\\
    &
    \quad
    \leqslant{}
    C
    \| \mathbb{P}u(t_0) \|_{\dB_{p,1}^{s}}
    + 
    C
    \sum_{j \in \mathbb{Z}}
    2^{sj}
    \| \dot{\Delta}_j \mathbb{P} g + ( w \cdot \nabla) \dot{\Delta}_j \mathbb{P} u \|_{L^1(I;L^p)},
\end{align}
provided that the right hand side is finite.
\end{lemm}
\begin{proof}
Applying $\dot{\Delta}_j \mathbb{P}$ to the second equation of \eqref{eq:lin-1}, we have 
\begin{align}
    \partial_t \dot{\Delta}_j \mathbb{P}u  - \mu \Delta \dot{\Delta}_j \mathbb{P}u + (w \cdot \nabla) \dot{\Delta}_j \mathbb{P} u = \dot{\Delta}_j \mathbb{P}g + (w \cdot \nabla) \dot{\Delta}_j\mathbb{P} u. 
\end{align}
Taking inner product of this with $|\dot{\Delta}_j \mathbb{P} u |^{p-2}\dot{\Delta}_j \mathbb{P} u$ and integrating over $\mathbb{R}^d$, 
we have 
\begin{align}\label{pf:lemm:lin-P-2}
    \frac{1}{p}\frac{d}{dt}\| \dot{\Delta}_j \mathbb{P}u \|_{L^p}^p
    + 
    c2^{2j}\| \dot{\Delta}_j \mathbb{P}u \|_{L^p}^p
    \leqslant
    \| \dot{\Delta}_j g + (w \cdot \nabla) \dot{\Delta}_j u \|_{L^p}\| \dot{\Delta}_j \mathbb{P}u \|_{L^p}^{p-1},
\end{align}
where, we have used the following well-known fact:
\begin{align}
    &
    \int_{\mathbb{R}^d} (w \cdot \nabla)\dot{\Delta}_j \mathbb{P} u \cdot (|\dot{\Delta}_j \mathbb{P}u|^{p-2}\dot{\Delta}_j \mathbb{P}u)dx = 0, \label{div-free}\\
    &
    \int_{\mathbb{R}^d} (-\mu\Delta\dot{\Delta}_j \mathbb{P} u) \cdot (|\dot{\Delta}_j \mathbb{P}u|^{p-2}\dot{\Delta}_j \mathbb{P}u)dx \geqslant c 2^{2j} \| \dot{\Delta}_j \mathbb{P}u \|_{L^p}^p.\label{dissp}
\end{align}
Here, we note that \eqref{div-free} is ensured by $\div w =0$ and integration by parts. \eqref{dissp} is proved in \cites{Can-Pla-02,Che-Mia-Zha-07,Pla-00}.
Dividing \eqref{pf:lemm:lin-P-2} by $\| \dot{\Delta}_j \mathbb{P}u \|_{L^p}^{p-1}$ and then integrating it with respect to the time variable, 
we complete the proof. 
\end{proof}
Next, we focus on the estimates for the compressible part of the solution to \eqref{eq:lin-1}.
In the following lemma, we consider the estimate for the high frequency part.
\begin{lemm}\label{lemm:lin-Q-1}
There exist positive constants
$\beta_0=\beta_0(d)$ and $C = C(d)$ such that 
for any 
$s \in \mathbb{R}$, 
$\varepsilon > 0$, 
$1 \leqslant{} p \leqslant{} \infty$, 
an interval $I \subset \mathbb{R}$ with the infimum $t_0$
and 
given space-time vector field $v=v(t,x)$ on $I \times \mathbb{R}^d$,
the compressible part $(a,\mathbb{Q}u)$ of the solution to \eqref{eq:lin-1} satisfies
\begin{align}
    \begin{split}
    &\varepsilon \| a \|_{\widetilde{L^{\infty}}(I;\dB_{p,1}^{s+1})}^{h;\frac{\beta_0}{\varepsilon}}
    +
    \frac{1}{\varepsilon}
    \| a \|_{L^1(I;\dB_{p,1}^{s+1})}^{h;\frac{\beta_0}{\varepsilon}}
    +
    \| \mathbb{Q}u \|_{\widetilde{L^{\infty}}(I;\dB_{p,1}^{s})\cap L^1(I;\dB_{p,1}^{s+2})}^{h;\frac{\beta_0}{\varepsilon}}\\
    &\quad 
    \leqslant{}
    C
    \bigg(\varepsilon\|a(t_0)\|_{\dB_{p,1}^{s+1}}^{h;\frac{\beta_0}{\varepsilon}}
    +
    \| \mathbb{Q}u(t_0) \|_{\dB_{p,1}^{s}}^{h;\frac{\beta_0}{\varepsilon}}
    \bigg)
    + 
    C\|(f,\mathbb{Q}g)\|_{L^1(I;\dB_{p,1}^{s})}^{h;\frac{\beta_0}{\varepsilon}}\\
    &\qquad
    +
    C\varepsilon\sum_{2^j \geqslant \frac{\beta_0}{\varepsilon}}2^{(s+1)j}
    \| \dot{\Delta}_j f +v\cdot \nabla \dot{\Delta}_j a\|_{L^1(I;L^p)}
    +
    \frac{C\varepsilon}{p}
    \left\|
        \| \div v \|_{L^{\infty}}
        \| a \|_{\dB_{p,1}^{s+1}}
    \right\|_{L^1(I)},
    \end{split}
\end{align}
provided that the right hand side is finite.
\end{lemm}
\begin{proof}
        The proof is based on the standard effective velocity argument proposed in \cites{Cha-Dan-10, Has-11-2}
        and introduce the effective velocity $w:=\mathbb{Q}u+\varepsilon^{-1}\nabla(-\Delta)^{-1}a$.
        Then, there holds 
        \begin{align}\label{eq:eff-ve}
            \begin{dcases}
            \varepsilon \partial_t a + \frac{1}{\varepsilon} a 
            = 
            \varepsilon f - \div w \\
            \begin{aligned}
            \partial_t w - \Delta w 
            ={}& 
            \frac{1}{\varepsilon}\nabla (-\Delta)^{-1}f + \mathbb{Q}g  
            \\
            &
            +
            \frac{1}{\varepsilon^2}w
            -
            \frac{1}{\varepsilon^3}\nabla(-\Delta)^{-1}a,
            \end{aligned}\\
            a(0,x) = a_0(x), \quad w(0,x)=w_0(x):=\mathbb{Q}u_0+\varepsilon^{-1}\nabla(-\Delta)^{-1}a_0,
            \end{dcases}
        \end{align}
        Applying $\dot{\Delta}_j$ to \eqref{eq:eff-ve} we see that 
        \begin{align}\label{eq:eff-ve-j}
            \varepsilon \partial_t \dot{\Delta}_j a + \frac{1}{\varepsilon} \dot{\Delta}_j a
            +
            \varepsilon 
            u \cdot \nabla \dot{\Delta}_j a
            = 
            \varepsilon (\dot{\Delta}_j f + u \cdot \nabla \dot{\Delta}_j a) 
            - 
            \div \dot{\Delta}_j w
        \end{align}
        Multiplying \eqref{eq:eff-ve-j} by $|\dot{\Delta}_j a|^{p-2}\dot{\Delta}_ja$ and integrating over $\mathbb{R}^d$, we obtain 
        \begin{align}
            \frac{\varepsilon}{p}
            \frac{d}{dt}
            \| \dot{\Delta}_j a \|_{L^p}^p
            +
            \frac{1}{\varepsilon}\| \dot{\Delta}_j a \|_{L^p}^p
            \leqslant{}&
            \left(
            \varepsilon\| \dot{\Delta}_j f + v \cdot \nabla \dot{\Delta}_j a \|_{L^p}
            +
            \| \div \dot{\Delta}_j w \|_{L^p}
            \right)
            \| \dot{\Delta}_j a \|_{L^p}^{p-1}\\
            &
            +
            \frac{1}{p}
            \int_{\mathbb R^d} | \div v | |\dot{\Delta}_j a |^p dx.
        \end{align}
        Here, we have used 
        \begin{align}
            \int_{\mathbb R^d} ( v \cdot \nabla \dot{\Delta}_j a )\cdot|\dot{\Delta}_ja|^{p-2}\dot{\Delta}_ja dx
            =
            -
            \frac{1}{p}
            \int_{\mathbb R^d} (\div v)
            |\dot{\Delta}_ja|^p dx,
        \end{align}
        which is obtained by the integration by parts.
        Thus, we obtain 
        \begin{align}\label{eff-a}
            \begin{split}
            &\varepsilon
            \| \dot{\Delta}_j a \|_{L^{\infty}(I;L^p)}
            +
            \frac{1}{\varepsilon}
            \| \dot{\Delta}_j a \|_{L^1(I;L^p)}
            \leqslant
            \varepsilon 
            \| \dot{\Delta}_j a(t_0) \|_{L^p}
            +
            C_*2^j\| \dot{\Delta}_j w \|_{L^1(I;L^p)}\\
            &\quad
            +
            \varepsilon
            \| \dot{\Delta}_j f + v \cdot \nabla \dot{\Delta}_j a\|_{L^1(I;L^p)}
            +
            \frac{\varepsilon}{p}
            \left\|
            \| \div v \|_{L^{\infty}}
            \| \dot{\Delta}_ja \|_{L^p}
            \right\|_{L^1(I)}
            \end{split}
        \end{align}
        for some positive constant $C_*=C_*(d)$.
        It follows from Lemma \ref{lemm:lin-P-1} that
        \begin{align}\label{eff-w}
            \begin{split}
            &
            \| \dot{\Delta}_j w \|_{L^{\infty}(I;L^p)}
            +
            2^{2j}
            \| \dot{\Delta}_j w \|_{L^1(I;L^p)}\\
            &\quad
            \leqslant{}
            \| \dot{\Delta}_j w(t_0) \|_{L^p}
            +
            \frac{C_{**}}{\varepsilon2^j}\| \dot{\Delta}_j f \|_{L^1(I;L^p)}
            +
            C\| \dot{\Delta}_j \mathbb{Q}g \|_{L^1(I;L^p)}\\
            &\qquad
            +
            \frac{C_{**}}{(\varepsilon2^j)^2}
            2^{2j}
            \| \dot{\Delta}_j w \|_{L^1( I; L^p )}
            +
            \frac{C_{**}}{(\varepsilon2^j)^2}
            \cdot
            \frac{2^{j}}{\varepsilon}
            \| \dot{\Delta}_j a \|_{L^1( I; L^p )}
            \end{split}
        \end{align}
        for some positive constant $C_{**}=C_{**}(d)$.
        Let $\delta$ be a positive constant to be determined later.
        Then, combining \eqref{eff-a} and \eqref{eff-w}, we have
        \begin{align}\label{est-eff-1}
            \begin{split}
            &\delta 2^j
            \varepsilon
            \| \dot{\Delta}_j a \|_{L^{\infty}(I;L^p)}
            +
            \frac{\delta2^j}{\varepsilon}
            \| \dot{\Delta}_j a \|_{L^1(I;L^p)}
            +
            \| \dot{\Delta}_j w \|_{L^{\infty}(I;L^p)}
            +
            2^{2j}
            \| \dot{\Delta}_j w \|_{L^1(I;L^p)}\\
            &\quad
            \leqslant
            \delta 2^j
            \varepsilon 
            \| \dot{\Delta}_j a(t_0) \|_{L^p}
            +
            \| \dot{\Delta}_j w(t_0) \|_{L^p}
            +
            \delta 
            C_*2^{2j}\| \dot{\Delta}_j w \|_{L^1(I;L^p)}\\
            &\qquad
            +
            \delta 2^j 
            \varepsilon
            \| \dot{\Delta}_j f + v \cdot \nabla \dot{\Delta}_j a\|_{L^1(I;L^p)}
            +
            \frac{\delta 2^j \varepsilon}{p}
            \left\|
            \| \div v \|_{L^{\infty}}
            \|\dot{\Delta}_j a \|_{L^p}
            \right\|_{L^1(I)}\\
            &\qquad 
            +
            \frac{C_{**}}{2^j\varepsilon}\| \dot{\Delta}_j f \|_{L^1(I;L^p)}
            +
            C\| \dot{\Delta}_j \mathbb{Q}g \|_{L^1(I;L^p)}\\
            &\qquad
            +
            \frac{C_{**}}{(\varepsilon2^j)^2}
            2^{2j}
            \| \dot{\Delta}_j w \|_{L^1( I; L^p )}
            +
            \frac{C_{**}}{(\varepsilon2^j)^2}
            \cdot
            \frac{2^{j}}{\varepsilon}
            \| \dot{\Delta}_j a \|_{L^1( I; L^p )}.
            \end{split}
        \end{align}
        Let us choose $\delta$ so that $1-C_*\delta \geqslant \delta$ and assume that $j$ satisfies $C_{**}/(2^j\varepsilon) \leqslant \delta/3$ and $C_{**}/(2^j\varepsilon)^2 \leqslant \delta/3$.
        Then, we see that 
        \begin{align}
            &
            2^j
            \varepsilon
            \| \dot{\Delta}_j a \|_{L^{\infty}(I;L^p)}
            +
            \frac{2^j}{\varepsilon}
            \| \dot{\Delta}_j a \|_{L^1(I;L^p)}
            +
            \| \dot{\Delta}_j w \|_{L^{\infty}(I;L^p)}
            +
            2^{2j}
            \| \dot{\Delta}_j w \|_{L^1(I;L^p)}\\
            &\quad
            \leqslant
            2^j
            \varepsilon 
            \| \dot{\Delta}_j a(t_0) \|_{L^p}
            +
            \| \dot{\Delta}_j w(t_0) \|_{L^p}
            +
            C\| \dot{\Delta}_j (f,\mathbb{Q}g) \|_{L^1(I;L^p)}\\
            &\qquad
            +
            C
            2^j 
            \varepsilon
            \| \dot{\Delta}_j f + v \cdot \nabla \dot{\Delta}_j a\|_{L^1(I;L^p)}
            +
            \frac{C2^j \varepsilon}{p}
            \left\|
            \| \div v \|_{L^{\infty}}
            \|\dot{\Delta}_j a \|_{L^p}
            \right\|_{L^1(I)}.
        \end{align}
        Here, since we see that $2^j
            \varepsilon 
            \| \dot{\Delta}_j a(t_0) \|_{L^p}
            +
            \| \dot{\Delta}_j w(t_0) \|_{L^p} 
            \sim 
            2^j
            \varepsilon 
            \| \dot{\Delta}_j a(t_0) \|_{L^p}
            +
            \| \dot{\Delta}_j \mathbb{Q}u(t_0) \|_{L^p}$
        and
        \begin{gather}
            2^j
            \varepsilon
            \| \dot{\Delta}_j a \|_{L^{\infty}(I;L^p)}
            +
            \| \dot{\Delta}_j w \|_{L^{\infty}(I;L^p)}
            \sim 
            2^j
            \varepsilon
            \| \dot{\Delta}_j a \|_{L^{\infty}(I;L^p)}
            +
            \| \dot{\Delta}_j \mathbb{Q}u\|_{L^{\infty}(I;L^p)},\\
            \frac{2^j}{\varepsilon}
            \| \dot{\Delta}_j a \|_{L^1(I;L^p)}
            +
            2^{2j}
            \| \dot{\Delta}_j w \|_{L^1(I;L^p)}
            \sim
            \frac{2^j}{\varepsilon}
            \| \dot{\Delta}_j a \|_{L^1(I;L^p)}
            +
            2^{2j}
            \| \dot{\Delta}_j \mathbb{Q}u \|_{L^1(I;L^p)}
        \end{gather}
        provided that $2^j\varepsilon \geqslant R$ for some sufficiently large constant $R=R(d)>1$, we complete the proof.
    \end{proof}

The next lemma provides an estimate for the low frequency part of the compressible part of the solution to \eqref{eq:lin-1}.
\begin{lemm}\label{lemm:lin-Q-2}
There exists a positive constant $C = C(d)$ such that 
for any 
$\varepsilon>0$, 
$s \in \mathbb{R}$, 
$\alpha, \beta\geqslant0$ with $\alpha <\beta/\varepsilon$
and 
an interval $I \subset [0,\infty)$ with the infimum $t_0$,
the compressible part $(a,\mathbb{Q}u)$ of the solution to \eqref{eq:lin-1} satisfies
\begin{align}\label{pt:ene-F}
    \begin{split}
    &
    \| (a,\mathbb{Q}u) \|_{\widetilde{L^{\infty}}(I;\dB_{2,1}^s)}^{m;\alpha,\frac{\beta}{\varepsilon}}
    +
    \frac{1}{(1+\beta)^2}
    \| (a,\mathbb{Q}u) \|_{L^1(I;\dB_{2,1}^{s+2})}^{m;\alpha,\frac{\beta}{\varepsilon}}\\
    &\quad 
    \leqslant{}
    C\left(1+\beta \right)
    \bigg(
    \| (a,\mathbb{Q}u)(t_0) \|_{\dB_{2,1}^s}^{m;\alpha,\frac{\beta}{\varepsilon}}
    +
    \|(f,\mathbb{Q}g)\|_{L^1(I;\dB_{2,1}^s)}^{m;\alpha,\frac{\beta}{\varepsilon}}
    \bigg),
    \end{split}
\end{align}
provided that the right hand side is finite.
\end{lemm}
\begin{rem}
In the estimate \eqref{pt:ene-F}, the case $\alpha=0$ implies 
\begin{align}
    &
    \| (a,\mathbb{Q}u) \|_{\widetilde{L^{\infty}}(I;\dB_{2,1}^s)}^{\ell;\frac{\beta}{\varepsilon}}
    +
    \frac{1}{(1+\beta)^2}
    \| (a,\mathbb{Q}u) \|_{L^1(I;\dB_{2,1}^{s+2})}^{\ell;\frac{\beta}{\varepsilon}}\\
    &\quad 
    \leqslant{}
    C\left(1+\beta \right)
    \bigg(
    \| (a,\mathbb{Q}u)(t_0) \|_{\dB_{2,1}^s}^{\ell;\frac{\beta}{\varepsilon}}
    +
    \|(f,\mathbb{Q}g)\|_{L^1(I;\dB_{2,1}^s)}^{\ell;\frac{\beta}{\varepsilon}}
    \bigg).
\end{align}
\end{rem}
We may prove Lemma \ref{lemm:lin-Q-2} by the pointwise energy calculation in the Fourier-side.
We refer to \cites{Dan-00-G,Dan-16} for the detail proof.

Finally, we consider the Strichartz estimates for the solution to \eqref{eq:lin-1}.
\begin{lemm}\label{lemm:lin-Q-3}
Let $2\leqslant{} q,r \leqslant{} \infty$ satisfy 
\begin{align}
    \frac{1}{r} \leqslant{} \frac{d-1}{2}\left( \frac{1}{2} - \frac{1}{q} \right),\quad
    (q,r,d) \neq (\infty,2,3).
\end{align}
Let an interval $I \subset [0,\infty)$ with the infimum $t_0$.
Then, there exists a positive constant $C=C(d,q,r)$ such that
\begin{align}
    \| (a, \mathbb{Q}u) \|_{\widetilde{L^r}(I;\dB_{q,1}^{\frac{d}{q}+\frac{1}{r}+\zeta})}^{\ell;\alpha}
    \leqslant{}
    C\varepsilon^{\frac{1}{r}}(1+\varepsilon \alpha)^3
    \bigg(
    \| (a,\mathbb{Q}u)(t_0) \|_{\dB_{2,1}^{\frac{d}{2}+\zeta}}^{\ell;\alpha}
    +
    \| (f,\mathbb{Q}g) \|_{L^1(I;\dB_{2,1}^{\frac{d}{2}+\zeta})}^{\ell;\alpha}
    \bigg)
\end{align}
for all $\varepsilon>0$, $\alpha > 0$ and $\zeta\in \mathbb{R}$.
\end{lemm}
\begin{proof}
Although the proof is based on the argument in \cites{Dan-02-R,Dan-He-16}, 
we give the outline of the proof for the readers' convinience.
Let $v:= |\nabla|^{-1}\div \mathbb{Q}u$.
Then, it is easy to see that $(a,v)$ satisfies 
\begin{align}
    \begin{dcases}
    \partial_t a + \frac{1}{\varepsilon}|\nabla| v = f,\\
    \partial_t v - \frac{1}{\varepsilon}|\nabla| a = \Delta v + |\nabla|^{-1}\div \mathbb{Q}g,\\
    a(0,x)=a_0(x),\quad v(0,x)=v_0(x):=|\nabla|^{-1}\div \mathbb{Q}u_0(x).
    \end{dcases}
\end{align}
Thus, we see that
\begin{align}
    \begin{pmatrix}
    a(t)\\
    v(t)
    \end{pmatrix}
    ={}&
    U\left(\frac{t-t_0}{\varepsilon}\right)
    \begin{pmatrix}
    a(t_0)\\
    v(t_0)
    \end{pmatrix}\\
    &
    +
    \int_{t_0}^t
    U\left(\frac{t-\tau}{\varepsilon}\right)
    \begin{pmatrix}
    f\\
    \Delta v + |\nabla|^{-1}\div \mathbb{Q}g
    \end{pmatrix}
    (\tau)d\tau,
\end{align}
where 
\begin{align}
    U\left( t \right)
    :={}&
    \begin{pmatrix}
    \cos |\nabla|t & -\sin |\nabla| t \\
    \sin |\nabla|t & \cos|\nabla|t
    \end{pmatrix}. 
\end{align}
Then, from the Strichartz estimates for the wave equation on $\mathbb{R}^d$ (see \cite{Bah-Che-Dan-11}*{Propositions 8.15 and 10.30}), we obtain
\begin{align}\label{wave-1}
    \| (a, v) \|_{\widetilde{L^r}(I;\dB_{q,1}^{\frac{d}{q}+\frac{1}{r}+\zeta})}^{\ell;\alpha}
    \leqslant{}&
    C\varepsilon^{\frac{1}{r}}
    \bigg(
    \| (a,\mathbb{Q}u)(t_0)\|_{\dB_{2,1}^{\frac{d}{2}+\zeta}}^{\ell;\alpha}
    +
    \| (f,\mathbb{Q}g) \|_{L^1(I;\dB_{2,1}^{\frac{d}{2}+\zeta})}^{\ell;\alpha}
    \bigg)\\
    &
    +
    C\varepsilon^{\frac{1}{r}}
    \| \mathbb{Q}u \|_{L^1(I;\dB_{2,1}^{\frac{d}{2}+2+\zeta})}^{\ell;\alpha}.
\end{align}
From Lemma \ref{lemm:lin-Q-2}, we have 
\begin{align}\label{wave-2}
    \| \mathbb{Q}u \|_{L^1(I;\dB_{2,1}^{\frac{d}{2}+2+\zeta})}^{\ell;\alpha}
    \leqslant
    C(1+\varepsilon \alpha)^3
    \bigg(
    \| (a,\mathbb{Q}u)(t_0) \|_{\dB_{2,1}^{\frac{d}{2}+\zeta}}^{\ell;\alpha}
    +
    \| (f,\mathbb{Q}g) \|_{L^1(I;\dB_{2,1}^{\frac{d}{2}+\zeta})}^{\ell;\alpha}
    \bigg).
\end{align}
Combining \eqref{wave-1}, \eqref{wave-2} and
\begin{align}
    \| \mathbb{Q}u \|_{\widetilde{L^r}(I;\dB_{q,1}^{\frac{d}{q}+\frac{1}{r}+\zeta})}^{\ell;\alpha}
    =
    \| \nabla|\nabla|^{-1}v \|_{\widetilde{L^r}(I;\dB_{q,1}^{\frac{d}{q}+\frac{1}{r}+\zeta})}^{\ell;\alpha}
    \leqslant
    C
    \| v \|_{\widetilde{L^r}(I;\dB_{q,1}^{\frac{d}{q}+\frac{1}{r}+\zeta})}^{\ell;\alpha},
\end{align}
we complete the proof.
\end{proof}

\section{Nonlinear estimates}\label{sec:nonlin}
In this section, we establish several nonlinear estimates by making use of para-product estimates prepared in Section \ref{sec:pre}.
In the following of this section, $I$ denotes an interval of $\mathbb{R}$.
\begin{lemm}\label{lemm:nonlin-2}
Let $\varepsilon, \beta>0$ and $1 \leqslant p,q \leqslant \infty$ satisfy
\begin{align}
    0 < \frac{1}{p} + \frac{1}{q} \leqslant 1.
\end{align}
Then, there exists a positive constant $C=C(d,p,q)$ such that
\begin{align}
    &
    \varepsilon \| a \div \mathbb{Q}u \|_{L^1(I;\dB_{p,1}^{\frac{d}{p}})}
    +
    \varepsilon
    \sum_{j \in \mathbb{Z}}
    2^{\frac{d}{p}j}
    \| [u\cdot \nabla, \dot{\Delta}_j]a \|_{L^1(I;L^p)}
    +
    \varepsilon
    \left\|
        \| \div u \|_{L^{\infty}}
        \| a \|_{\dB_{p,1}^{\frac{d}{p}}}
    \right\|_{L^1(I)}\\
    &\quad
    \leqslant
    C 
    \varepsilon
    \| a \|_{L^{\infty}(I;\dB_{q,1}^{\frac{d}{q}})}
    \| u \|_{L^1(I;\dB_{p,1}^{\frac{d}{p}+1})}^{h;\frac{\beta}{\varepsilon}}
    +
    C 
    \varepsilon
    \| a \|_{L^{\infty}(I;\dB_{p,1}^{\frac{d}{p}})}
    \| u \|_{L^1(I;\dB_{q,1}^{\frac{d}{q}+1})}^{h;\frac{\beta}{\varepsilon}}\\
    &\qquad
    +
    C
    \beta
    \| a \|_{L^2(I;\dB_{q,1}^{\frac{d}{q}})}
    \| u \|_{L^2(I;\dB_{p,1}^{\frac{d}{p}})}
    +
    C
    \beta
    \| a \|_{L^2(I;\dB_{p,1}^{\frac{d}{p}})}
    \| u \|_{L^2(I;\dB_{q,1}^{\frac{d}{q}})},
\end{align}
provided that the right hand side is finite.
\end{lemm}
\begin{proof}
By Lemmas \ref{lemm:prod-1} and \ref{lemm:comm}, we have
\begin{align}
    &
    \varepsilon \| a \div \mathbb{Q}u \|_{L^1(I;\dB_{p,1}^{\frac{d}{p}})}
    +
    \varepsilon
    \sum_{j \in \mathbb{Z}}
    2^{\frac{d}{p}j}
    \| [u\cdot \nabla, \mathbb{P}\dot{\Delta}_j]a \|_{L^1(I;L^p)}
    +
    \varepsilon
    \left\|
        \| \div u \|_{L^{\infty}}
        \| a \|_{\dB_{p,1}^{\frac{d}{p}}}
    \right\|_{L^1(I)}\\
    &\quad
    \leqslant
    C 
    \varepsilon
    \left\|
    \| a \|_{\dB_{q,1}^{\frac{d}{q}}}
    \| u \|_{\dB_{p,1}^{\frac{d}{p}+1}}
    \right\|_{L^1(I)}
    +
    C\varepsilon
    \left\|
    \| a \|_{\dB_{p,1}^{\frac{d}{p}}}
    \| u \|_{\dB_{q,1}^{\frac{d}{q}+1}}
    \right\|_{L^1(I)}.
\end{align}
Here, it holds
\begin{align}\label{est:au-1}
    \begin{split}
    \varepsilon
    \left\|
    \| a \|_{\dB_{q,1}^{\frac{d}{q}}}
    \| u \|_{\dB_{p,1}^{\frac{d}{p}+1}}
    \right\|_{L^1(I)}
    \leqslant{}&
    \varepsilon
    \left\|
    \| a \|_{\dB_{q,1}^{\frac{d}{q}}}
    \| u \|_{\dB_{p,1}^{\frac{d}{p}+1}}^{h;\frac{\beta}{\varepsilon}}
    \right\|_{L^1(I)}
    +
    \beta
    \left\|
    \| a \|_{\dB_{q,1}^{\frac{d}{q}}}
    \| u \|_{\dB_{p,1}^{\frac{d}{p}}}^{\ell;\frac{\beta}{\varepsilon}}
    \right\|_{L^1(I)}\\
    \leqslant{}&
    \varepsilon
    \| a \|_{L^{\infty}(I;\dB_{q,1}^{\frac{d}{q}})}
    \| u \|_{L^1(I;\dB_{p,1}^{\frac{d}{p}+1})}^{h;\frac{\beta}{\varepsilon}}\\
    &
    +
    \beta
    \| a \|_{L^2(I;\dB_{q,1}^{\frac{d}{q}})}
    \| u \|_{L^2(I;\dB_{p,1}^{\frac{d}{p}})}.
    \end{split}
\end{align}
Similarly, we have
\begin{align}\label{est:au-2}
    \varepsilon
    \left\|
    \| a \|_{\dB_{p,1}^{\frac{d}{p}}}
    \| u \|_{\dB_{q,1}^{\frac{d}{q}+1}}
    \right\|_{L^1(I)}
    \leqslant{}
    &
    \varepsilon
    \| a \|_{L^{\infty}(I;\dB_{p,1}^{\frac{d}{p}})}
    \| u \|_{L^1(I;\dB_{q,1}^{\frac{d}{q}+1})}^{h;\frac{\beta}{\varepsilon}}\\
    &
    +
    \beta
    \| a \|_{L^2(I;\dB_{p,1}^{\frac{d}{p}})}
    \| u \|_{L^2(I;\dB_{q,1}^{\frac{d}{q}})}.
\end{align}
Thus, we complete the proof.
\end{proof}

\begin{lemm}\label{lemm:nonlin-1}
Let $\varepsilon>0$.
Then, the following two properties hold:
\begin{itemize}
\item [(1)]
Let $1 \leqslant p,q \leqslant \infty$ satisfy
\begin{align}
    0 < \frac{1}{p} + \frac{1}{q} \leqslant 1.
\end{align}
Then, there exists a positive constant $C=C(d,p,q)$ such that 
\begin{align}
    \| au \|_{L^1(I;\dB_{p,1}^{\frac{d}{p}})}
    \leqslant{}&
    C
    \| a \|_{L^2(I;\dB_{q,1}^{\frac{d}{q}})}
    \| u \|_{L^2(I;\dB_{p,1}^{\frac{d}{p}})}
    +
    C
    \| a \|_{L^2(I;\dB_{p,1}^{\frac{d}{p}})}
    \| u \|_{L^2(I;\dB_{q,1}^{\frac{d}{q}})},
\end{align}
provided that the right hand side is finite.
\item[(2)]
Let $\beta > 0$ and let $2 \leqslant q \leqslant 4$ and $2 \leqslant r \leqslant \infty$ satisfy
\begin{align}
    \frac{1}{r} \leqslant \frac{1}{2} - \frac{d}{2}\left( \frac{1}{2} - \frac{1}{q} \right).
\end{align}
Then, there exists a positive constant $C=C(d,q,r)$ such that 
\begin{align}
    \| au \|_{L^1(I;\dB_{2,1}^{\frac{d}{2}})}^{\ell;\frac{\beta}{\varepsilon}}
    \leqslant{}&
    C
    \| a \|_{L^r(I;\dB_{q,1}^{\frac{d}{q}-1+\frac{2}{r}})}
    \| u \|_{L^{r'}(I;\dB_{q,1}^{\frac{d}{q}-1+\frac{2}{r'}})}\\
    &
    +
    C
    \| u \|_{L^r(I;\dB_{q,1}^{\frac{d}{q}-1+\frac{2}{r}})}^{\ell;\frac{\beta}{\varepsilon}}
    \| a \|_{L^{r'}(I;\dB_{q,1}^{\frac{d}{q}-1+\frac{2}{r'}})}
    ^{\ell;\frac{4\beta}{\varepsilon}},
\end{align}
provided that the right hand side is finite.
\end{itemize}
\end{lemm}

\begin{proof}
From Lemma \ref{lemm:prod-1}, we obtain (1).
By Lemma \ref{lemm:prod-2}, we have (2)
and complete the proof.
\end{proof}

\begin{lemm}\label{lemm:nonlin-3}
The following three properties hold:
\begin{itemize}
\item[(1)]
Let $1 \leqslant p,q \leqslant \infty$ satisfy
\begin{align}
    \frac{1}{d} < \frac{1}{p} + \frac{1}{q} \leqslant 1.
\end{align}
Then, there exists a positive constant $C=C(p,q)$ such that 
\begin{align}
    \| (u \cdot \nabla)v \|_{L^1(I;\dB_{p,1}^{\frac{d}{p}-1})}
    \leqslant
    C
    \| u \|_{L^2(I;\dB_{p,1}^{\frac{d}{p}})}
    \| v \|_{L^2(I;\dB_{q,1}^{\frac{d}{q}})}
    +
    C
    \| u \|_{L^2(I;\dB_{q,1}^{\frac{d}{q}})}
    \| v \|_{L^2(I;\dB_{p,1}^{\frac{d}{p}})},
\end{align}
provided that the right hand side is finite.
\item[(2)]
Let $2 \leqslant q \leqslant 4$ and $2 \leqslant r \leqslant \infty$ satisfy
\begin{align}
    \frac{1}{r} \leqslant \frac{1}{2} - \frac{d}{2}\left( \frac{1}{2} - \frac{1}{q} \right),\qquad
    \frac{2d}{q} - 1 >0.
\end{align}
Then, there exists a positive constant $C=C(d,q,r)$ such that 
\begin{align}
    \| (v \cdot \nabla)u \|_{L^1(I;\dB_{2,1}^{\frac{d}{2}-1})}
    \leqslant{}&
    C
    \| u \|_{L^r(I;\dB_{q,1}^{\frac{d}{q}-1+\frac{2}{r}})}
    \| v \|_{L^{r'}(I;\dB_{q,1}^{\frac{d}{q}-1+\frac{2}{r'}})}\\
    &
    +
    C
    \| u \|_{L^{r'}(I;\dB_{q,1}^{\frac{d}{q}-1+\frac{2}{r'}})}
    \| v \|_{L^r(I;\dB_{q,1}^{\frac{d}{q}-1+\frac{2}{r}})},
\end{align}
provided that the right hand side is finite.
\item[(3)]
Let $m=0,1$, $1 \leqslant q <2d$ and $2 \leqslant r \leqslant \infty$ satisfy 
\begin{align}
    \frac{2d}{q} -1 -\frac{m}{r} >0.
\end{align}
Then, there exists a positive constant $C=C(d,q,r)$ such that
\begin{align}
    &
    \| (v \cdot \nabla)u \|_{L^1(I;\dB_{q,1}^{\frac{d}{q}-1-\frac{m}{r}})}
    \leqslant
    C
    \| u \|_{L^{r'}(I;\dB_{q,1}^{\frac{d}{q}-1+\frac{2}{r'}})}
    \| v \|_{L^{r}(I;\dB_{q,1}^{\frac{d}{q}-1+\frac{2-m}{r}})},\\
    &
    \sum_{j \in \mathbb{Z}}
    2^{(\frac{d}{q}-1-\frac{m}{r})j}
    \| [w \cdot \nabla, \mathbb{P}\dot{\Delta}_j]v \|_{L^q}
    \leqslant
    C
    \| w \|_{L^{r'}(I;\dB_{q,1}^{\frac{d}{q}-1+\frac{2}{r'}})}
    \| v \|_{L^{r}(I;\dB_{q,1}^{\frac{d}{q}-1+\frac{2-m}{r}})},
\end{align}
provided that the right hand side is finite.
\end{itemize}
\end{lemm}
\begin{proof}
From Lemma \ref{lemm:prod-1}, we immediately obtain (1).
For the proof of (2),
it follows from Lemma \ref{lemm:prod-2} with $\beta=\infty$ that
\begin{align}
    \| (v \cdot \nabla)u \|_{L^1(I;\dB_{2,1}^{\frac{d}{2}-1})}
    \leqslant{}
    &
    C
    \| v \|_{L^r(I;\dB_{q,1}^{\frac{d}{q}-1+\frac{2}{r}})}
    \| \nabla u \|_{L^{r'}(I;\dB_{q,1}^{\frac{d}{q}-2+\frac{2}{r'}})}\\
    &
    +
    C
    \| v \|_{L^{r'}(I;\dB_{q,1}^{\frac{d}{q}-1+\frac{2}{r'}})}
    \| \nabla u \|_{L^r(I;\dB_{q,1}^{\frac{d}{q}-2+\frac{2}{r}})}\\
    \leqslant{}
    &
    C
    \| u \|_{L^r(I;\dB_{q,1}^{\frac{d}{q}-1+\frac{2}{r}})}
    \| v \|_{L^{r'}(I;\dB_{q,1}^{\frac{d}{q}-1+\frac{2}{r'}})}\\
    &
    +
    C
    \| u \|_{L^{r'}(I;\dB_{q,1}^{\frac{d}{q}-1+\frac{2}{r'}})}
    \| v \|_{L^r(I;\dB_{q,1}^{\frac{d}{q}-1+\frac{2}{r}})}.
\end{align}
Next, we focus on (3).
By Lemma \ref{lemm:prod-1}, it holds
\begin{align}
    \| (v \cdot \nabla)u \|_{L^1(I;\dB_{q,1}^{\frac{d}{q}-1-\frac{m}{r}})}
    \leqslant{}&
    C
    \| v \|_{L^{r}(I;\dB_{q,1}^{\frac{d}{q}-1+\frac{2-m}{r}})}
    \| \nabla u \|_{L^{r'}(I;\dB_{q,1}^{\frac{d}{q}-2+\frac{2}{r'}})}\\
    \leqslant{}
    &
    C
    \| v \|_{L^{r}(I;\dB_{q,1}^{\frac{d}{q}-1+\frac{2-m}{r}})}
    \| u \|_{L^{r'}(I;\dB_{q,1}^{\frac{d}{q}-1+\frac{2}{r'}})},
\end{align}
which proves the first assertion.
Lemma \ref{lemm:comm} yields the second assertion.
Thus, we complete the proof.
\end{proof}

\begin{lemm}\label{lemm:nonlin-4}
Let $\varepsilon, \beta>0$.
Then, the following three properties hold:
\begin{itemize}
\item [(1)]
Let $1 \leqslant p,q < \infty$ satisfy
\begin{align}
    \frac{1}{d} < \frac{1}{p} + \frac{1}{q} \leqslant 1.
\end{align}
Then, there exists a positive constant $C=C(d,\mu,p,q)$ such that 
\begin{align}
    \| \mathcal{J}(\varepsilon a) \mathcal{L}u \|_{L^1(I;\dB_{p,1}^{\frac{d}{p}-1})}
    \leqslant{}
    &
    C 
    \varepsilon
    \| a \|_{L^{\infty}(I;\dB_{q,1}^{\frac{d}{q}})}
    \| u \|_{L^1(I;\dB_{p,1}^{\frac{d}{p}+1})}^{h;\frac{\beta}{\varepsilon}}\\
    &+
    C 
    \varepsilon
    \| a \|_{L^{\infty}(I;\dB_{p,1}^{\frac{d}{p}})}
    \| u \|_{L^1(I;\dB_{q,1}^{\frac{d}{q}+1})}^{h;\frac{\beta}{\varepsilon}}\\
    &+
    C
    \beta
    \| a \|_{L^2(I;\dB_{q,1}^{\frac{d}{q}})}
    \| u \|_{L^2(I;\dB_{p,1}^{\frac{d}{p}})}\\
    &+
    C
    \beta
    \| a \|_{L^2(I;\dB_{p,1}^{\frac{d}{p}})}
    \| u \|_{L^2(I;\dB_{q,1}^{\frac{d}{q}})},
\end{align}
provided that the right hand side is finite and $\varepsilon  \| a\|_{L^{\infty}(I;L^{\infty}\cap\dB_{q,1}^{\frac{d}{q}})} \leqslant 1/2$.
\item[(2)]
Let $2 \leqslant q \leqslant 4$ and $2 \leqslant r \leqslant \infty$ satisfy
\begin{align}
    \frac{1}{r} \leqslant \frac{1}{2} - \frac{d}{2}\left( \frac{1}{2} - \frac{1}{q} \right),\qquad
    \frac{2d}{q} - 1 >0.
\end{align}
Then, there exists a positive constant $C=C(d,\mu,q,r)$ such that 
\begin{align}
    \| \mathcal{J}(\varepsilon a) \mathcal{L}u \|_{L^1(I;\dB_{2,1}^{\frac{d}{2}-1})}^{\ell;\frac{\beta}{\varepsilon}}
    \leqslant{}
    &
    C 
    \varepsilon
    \| a \|_{L^{\infty}(I;\dB_{q,1}^{\frac{d}{q}})}
    \| u \|_{L^1(I;\dB_{q,1}^{\frac{d}{q}+1})}^{h;\frac{\beta}{\varepsilon}}\\
    &+
    C
    \beta
    \| a \|_{L^2(I;\dB_{q,1}^{\frac{d}{q}})}
    \| u \|_{L^2(I;\dB_{q,1}^{\frac{d}{q}})}\\
    &+
    C
    \beta
    C
    \| a \|_{L^r(I;\dB_{q,1}^{\frac{d}{q}-1+\frac{2}{r}})}
    \| u \|_{L^{r'}(I;\dB_{q,1}^{\frac{d}{q}-1+\frac{2}{r'}})}\\
    &
    +
    C
    \beta
    \| u \|_{L^r(I;\dB_{q,1}^{\frac{d}{q}-1+\frac{2}{r}})}^{\ell;\frac{\beta}{\varepsilon}}
    \| a \|_{L^{r'}(I;\dB_{q,1}^{\frac{d}{q}-1+\frac{2}{r'}})}
    ^{\ell;\frac{4\beta}{\varepsilon}},
\end{align}
provided that the right hand side is finite and $\varepsilon  \| a\|_{L^{\infty}(I;L^{\infty}\cap\dB_{q,1}^{\frac{d}{q}})} \leqslant 1/2$.
\item[(3)]
For $1 \leqslant{}  q < 2d$ and $2 \leqslant{} r \leqslant{} \infty$ satisfying
\begin{align}
    \frac{1}{r} < \frac{2d}{q} - 1,
\end{align}
there exists a positive constant $C=C(d,\mu,q,r)$ such that 
\begin{align}\label{nonlin-3-3}
    \begin{split}
    \| \mathcal{J}(\varepsilon a ) \mathcal{L} u \|_{L^1(I;\dB_{q,1}^{\frac{d}{q}-1-\frac{1}{r}})}
    \leqslant{}
    &
    C 
    \beta
    \| a \|_{L^r(I;\dB_{q,1}^{\frac{d}{q}-1+\frac{1}{r}})}
    \| u \|_{L^{r'}(I;\dB_{q,1}^{\frac{d}{q}-1+\frac{2}{r'}})}\\
    &
    +
    C
    \beta^{-\frac{1}{r}}\varepsilon^{\frac{1}{r}+1}
    \| a \|_{L^{\infty}(I;\dB_{q,1}^{\frac{d}{q}})}
    \| u \|_{L^1(I;\dB_{q,1}^{\frac{d}{q}+1})}^{h;\frac{\beta}{\varepsilon}}\\
    &
    +
    C
    \beta^{-\frac{1}{r}+1}
    \varepsilon^{\frac{1}{r}}
    \| a \|_{L^2(I;\dB_{q,1}^{\frac{d}{q}})}\| u \|_{L^2(I;\dB_{q,1}^{\frac{d}{q}})},
    \end{split}
\end{align}
provided that the right hand side is finite and $\varepsilon  \| a\|_{L^{\infty}(I;L^{\infty}\cap\dB_{q,1}^{\frac{d}{q}})} \leqslant 1/2$.
\end{itemize}
\end{lemm}

\begin{proof}
We first show (1). 
It follows from Lemmas \ref{lemm:prod-1}, \ref{lemm:composition} and \eqref{est:au-1}, \eqref{est:au-2} that
\begin{align}
    \| \mathcal{J}(\varepsilon a) \mathcal{L}u \|_{L^1(I;\dB_{p,1}^{\frac{d}{p}-1})}
    \leqslant{}&
    C 
    \varepsilon
    \left\|
    \| a \|_{\dB_{q,1}^{\frac{d}{q}}}
    \| u \|_{\dB_{p,1}^{\frac{d}{p}+1}}
    \right\|_{L^1(I)}
    +
    C 
    \varepsilon
    \left\|
    \| a \|_{\dB_{p,1}^{\frac{d}{p}}}
    \| u \|_{\dB_{q,1}^{\frac{d}{q}+1}}
    \right\|_{L^1(I)}\\
    \leqslant{}
    &
    C 
    \varepsilon
    \| a \|_{L^{\infty}(I;\dB_{q,1}^{\frac{d}{q}})}
    \| u \|_{L^1(I;\dB_{p,1}^{\frac{d}{p}+1})}^{h;\frac{\beta}{\varepsilon}}\\
    &+
    C 
    \varepsilon
    \| a \|_{L^{\infty}(I;\dB_{p,1}^{\frac{d}{p}})}
    \| u \|_{L^1(I;\dB_{q,1}^{\frac{d}{q}+1})}^{h;\frac{\beta}{\varepsilon}}\\
    &+
    C
    \beta
    \| a \|_{L^2(I;\dB_{q,1}^{\frac{d}{q}})}
    \| u \|_{L^2(I;\dB_{p,1}^{\frac{d}{p}})}\\
    &+
    C
    \beta
    \| a \|_{L^2(I;\dB_{p,1}^{\frac{d}{p}})}
    \| u \|_{L^2(I;\dB_{q,1}^{\frac{d}{q}})}.
\end{align}

Next, we prove (2).
By Lemmas \ref{lemm:prod-2} and \ref{lemm:composition}, we see that
\begin{align}
    \| \mathcal{J}(\varepsilon a)\mathcal{L}u \|_{L^1(I;\dB_{2,1}^{\frac{d}{2}-1})}^{\ell;\frac{\beta}{\varepsilon}}
    \leqslant{}
    &
    C
    \left\|
    \| \mathcal{L}u \|_{\dB_{q,1}^{\frac{d}{q}-1}}
    \| \mathcal{J}(\varepsilon a) \|_{\dB_{q,1}^{\frac{d}{q}}}
    \right\|_{L^1(I)}\\
    &
    +
    C\| \mathcal{J}(\varepsilon a) \|_{L^r(I;\dB_{q,1}^{\frac{d}{q}-1+\frac{2}{r}})}
    \| \mathcal{L}u \|_{L^{r'}(I;\dB_{q,1}^{\frac{d}{q}-2+\frac{2}{r'}})}^{\ell;\frac{\beta}{4\varepsilon}}\\
    \leqslant{}
    &
    C 
    \varepsilon
    \| a \|_{L^{\infty}(I;\dB_{q,1}^{\frac{d}{q}})}
    \| u \|_{L^1(I;\dB_{q,1}^{\frac{d}{q}+1})}^{h;\frac{\beta}{\varepsilon}}\\
    &+
    C
    \beta
    \| a \|_{L^2(I;\dB_{q,1}^{\frac{d}{q}})}
    \| u \|_{L^2(I;\dB_{q,1}^{\frac{d}{q}})}\\
    &+
    C
    \beta
    \| a \|_{L^r(I;\dB_{q,1}^{\frac{d}{q}-1+\frac{2}{r}})}
    \| u \|_{L^{r'}(I;\dB_{q,1}^{\frac{d}{q}-1+\frac{2}{r'}})}.
\end{align}

Finally, we show (3).
By the Bernstein inequality and (1), we have
\begin{align}
    \| \mathcal{J}(\varepsilon a) \mathcal{L}u \|_{L^1(I;\dB_{q,1}^{\frac{d}{q}-1-\frac{1}{r}})}^{h;\frac{\beta}{\varepsilon}}
    \leqslant{}
    &
    \beta^{-\frac{1}{r}}
    \varepsilon^{\frac{1}{r}}
    \| \mathcal{J}(\varepsilon a) \mathcal{L}u \|_{L^1(I;\dB_{q,1}^{\frac{d}{q}-1})}\\
    \leqslant{}
    &
    C 
    \beta^{-\frac{1}{r}}
    \varepsilon^{\frac{1}{r}+1}
    \| a \|_{L^{\infty}(I;\dB_{q,1}^{\frac{d}{q}})}
    \| u \|_{L^1(I;\dB_{q,1}^{\frac{d}{q}+1})}^{h;\frac{\beta}{\varepsilon}}\\
    &
    +
    C
    \beta^{-\frac{1}{r}+1}
    \varepsilon^{\frac{1}{r}}
    \| a \|_{L^2(I;\dB_{q,1}^{\frac{d}{q}})}
    \| u \|_{L^2(I;\dB_{q,1}^{\frac{d}{q}})}.
\end{align}
For the low frequency estimate, by the Bony decomposition and the Bernstein inequality, we see that
\begin{align}
    \| \mathcal{J}(\varepsilon a)\mathcal{L}u \|_{L^1(I;\dB_{q,1}^{\frac{d}{q}-1-\frac{1}{r}})}^{\ell;\frac{\beta}{\varepsilon}}
    \leqslant{}
    &
    \| \dot{T}_{\mathcal{J}(\varepsilon a)}\mathcal{L}u \|_{L^1(I;\dB_{q,1}^{\frac{d}{q}-1-\frac{1}{r}})}^{\ell;\frac{\beta}{\varepsilon}}
    +
    \| \dot{T}_{\mathcal{L}u}\mathcal{J}(\varepsilon a) \|_{L^1(I;\dB_{q,1}^{\frac{d}{q}-1-\frac{1}{r}})}^{\ell;\frac{\beta}{\varepsilon}}\\
    &
    +
    C
    \| \dot{R}(\mathcal{J}(\varepsilon a),\mathcal{L}u) \|_{L^1(I;\dB_{\frac{q}{2},1}^{\frac{2d}{q}-1-\frac{1}{r}})}.
\end{align}
Lemmas \ref{lemm:para-1} and \ref{lemm:composition} yield
\begin{align}
    \| \dot{T}_{\mathcal{J}(\varepsilon a)}\mathcal{L}u \|_{L^1(I;\dB_{q,1}^{\frac{d}{q}-1-\frac{1}{r}})}^{\ell;\frac{\beta}{\varepsilon}}
    \leqslant{}
    &
    C 
    \| \mathcal{J}(\varepsilon a ) \|_{L^r(I;\dB_{\infty,1}^{-1+\frac{1}{r}})}
    \| \mathcal{L}u \|_{L^{r'}(I;\dB_{q,1}^{\frac{d}{q}-\frac{2}{r}})}^{\ell;\frac{4\beta}{\varepsilon}}\\
    \leqslant{}
    &
    C
    \beta
    \|  a  \|_{L^r(I;\dB_{q,1}^{\frac{d}{q}-1+\frac{1}{r}})}
    \| u \|_{L^{r'}(I;\dB_{q,1}^{\frac{d}{q}-1+\frac{2}{r'}})}^{\ell;\frac{4\beta}{\varepsilon}}
\end{align}
and
\begin{align}
    \| \dot{T}_{\mathcal{L}u}\mathcal{J}(\varepsilon a) \|_{L^1(I;\dB_{q,1}^{\frac{d}{q}-1-\frac{1}{r}})}^{\ell;\frac{\beta}{\varepsilon}}
    \leqslant{}
    &
    C
    \| \mathcal{L} u \|_{L^{r'}(I;\dB_{\infty,1}^{-\frac{2}{r}})}^{\ell;\frac{\beta}{\varepsilon}}
    \| \mathcal{J}(\varepsilon a) \|_{L^r(I;\dB_{q,1}^{\frac{d}{q}-1+\frac{1}{r}})}\\
    \leqslant{}
    &
    C
    \beta
    \|  a  \|_{L^r(I;\dB_{q,1}^{\frac{d}{q}-1+\frac{1}{r}})}
    \| u \|_{L^{r'}(I;\dB_{q,1}^{\frac{d}{q}-1+\frac{2}{r'}})}^{\ell;\frac{4\beta}{\varepsilon}}.
\end{align}
By Lemmas \ref{lemm:para-2} and \ref{lemm:composition}, we see
\begin{align}
    \| \dot{R}(\mathcal{J}(\varepsilon a),\mathcal{L}u) \|_{L^1(I;\dB_{\frac{q}{2},1}^{\frac{2d}{q}-1-\frac{1}{r}})}
    \leqslant{}
    &
    C 
    \left\|
    \| \mathcal{J}(\varepsilon a) \|_{\dB_{q,1}^{\frac{d}{q}}}
    \| \mathcal{L} u \|_{\dB_{q,1}^{\frac{d}{q}-1-\frac{1}{r}}}^{h;\frac{\beta}{\varepsilon}}\right.\\
    &\qquad
    +
    \left.
    \| \mathcal{J}(\varepsilon a) \|_{\dB_{q,1}^{\frac{d}{q}-1+\frac{1}{r}}}
    \| \mathcal{L} u \|_{\dB_{q,1}^{\frac{d}{q}-\frac{2}{r}}}^{\ell;\frac{\beta}{\varepsilon}}
    \right\|_{L^1(I)}\\
    \leqslant{}
    &
    C 
    \beta^{-\frac{1}{r}}
    \varepsilon^{\frac{1}{r}+1}
    \| a \|_{L^{\infty}(I;\dB_{q,1}^{\frac{d}{q}})}
    \| u \|_{L^1(I;\dB_{q,1}^{\frac{d}{q}+1})}^{h;\frac{\beta}{\varepsilon}}\\
    &
    +
    C
    \beta
    \| a \|_{L^r(I;\dB_{q,1}^{\frac{d}{q}-1+\frac{1}{r}})}
    \| u \|_{L^{r'}(I;\dB_{q,1}^{\frac{d}{q}-1+\frac{2}{r'}})}.
\end{align}
This completes the proof.
\end{proof}

\begin{lemm}\label{lemm:nonlin-5}
Let $\varepsilon, \beta>0$.
Then, the following two properties hold:
\begin{itemize}
\item [(1)]
Let $1 \leqslant p,q < \infty$ satisfy
\begin{align}
    \frac{1}{d} < \frac{1}{p} + \frac{1}{q} \leqslant 1.
\end{align}
Then, there exists a positive constant $C=C(d,p,q,P)$ such that 
\begin{align}
    \frac{1}{\varepsilon}\| \mathcal{K}(\varepsilon a) \nabla a \|_{L^1(I;\dB_{p,1}^{\frac{d}{p}-1})}
    \leqslant{}
    C
    \| a \|_{L^2(I;\dB_{q,1}^{\frac{d}{q}})}
    \| a \|_{L^2(I;\dB_{p,1}^{\frac{d}{p}})},
\end{align}
provided that the right hand side is finite and $\varepsilon  \| a\|_{L^{\infty}(I;L^{\infty}\cap\dB_{q,1}^{\frac{d}{q}})} \leqslant 1/2$.
\item[(2)]
Let $2 \leqslant q \leqslant 4$ and $2 \leqslant r \leqslant \infty$ satisfy
\begin{align}
    \frac{1}{r} \leqslant \frac{1}{2} - \frac{d}{2}\left( \frac{1}{2} - \frac{1}{q} \right),\qquad
    \frac{2d}{q} - 1 >0.
\end{align}
Then, there exists a positive constant $C=C(d,q,r,P)$ such that 
\begin{align}
    \frac{1}{\varepsilon}
    \| \mathcal{K}(\varepsilon a) \nabla a \|_{L^1(I;\dB_{2,1}^{\frac{d}{2}-1})}^{\ell;\frac{\beta}{\varepsilon}}
    \leqslant{}
    &
    C
    \| a \|_{L^2(I;\dB_{q,1}^{\frac{d}{q}})}^2\\
    &+
    C
    \| a \|_{L^r(I;\dB_{q,1}^{\frac{d}{q}-1+\frac{2}{r}})}
    \| a \|_{L^{r'}(I;\dB_{q,1}^{\frac{d}{q}-1+\frac{2}{r'}})}^{\ell;\frac{4\beta}{\varepsilon}},
\end{align}
provided that the right hand side is finite and $\varepsilon  \| a\|_{L^{\infty}(I;L^{\infty}\cap\dB_{q,1}^{\frac{d}{q}})} \leqslant 1/2$.
\end{itemize}
\end{lemm}

\begin{proof}
For the proof of (1), 
Lemmas \ref{lemm:prod-1} and \ref{lemm:composition} yield
\begin{align}
    &\frac{1}{\varepsilon}\| \mathcal{K}(\varepsilon a) \nabla a \|_{L^1(I;\dB_{p,1}^{\frac{d}{p}-1})}\\
    &\quad \leqslant{}
    \frac{C}{\varepsilon}
    \| \mathcal{K}(\varepsilon a) \|_{L^2(I;\dB_{q,1}^{\frac{d}{q}})}
    \| \nabla a \|_{L^2(I;\dB_{p,1}^{\frac{d}{p}-1})}
    +
    \frac{C}{\varepsilon}
    \| \mathcal{K}(\varepsilon a) \|_{L^2(I;\dB_{p,1}^{\frac{d}{p}})}
    \| \nabla a \|_{L^2(I;\dB_{q,1}^{\frac{d}{q}-1})}\\
    &\quad \leqslant{}
    C
    \| a \|_{L^2(I;\dB_{q,1}^{\frac{d}{q}})}
    \| a \|_{L^2(I;\dB_{p,1}^{\frac{d}{p}})}.
\end{align}

Next, we show (2).
By Lemmas \ref{lemm:prod-2} and \ref{lemm:composition}, we have
\begin{align}
    &
    \frac{1}{\varepsilon}
    \| \mathcal{K}(\varepsilon a) \nabla a \|_{L^1(I;\dB_{2,1}^{\frac{d}{2}-1})}^{\ell;\frac{\beta}{\varepsilon}}\\
    &\quad
    \leqslant{}
    \frac{C}{\varepsilon}
    \| \mathcal{K}(\varepsilon a) \|_{L^2(I;\dB_{q,1}^{\frac{d}{q}})}
    \| \nabla a \|_{L^2(I;\dB_{q,1}^{\frac{d}{q}-1})}
    +
    \frac{C}{\varepsilon}
    \| \mathcal{K}(\varepsilon a) \|_{L^r(I;\dB_{q,1}^{\frac{d}{q}-1+\frac{2}{r}})}
    \| \nabla a \|_{L^{r'}(I;\dB_{q,1}^{\frac{d}{q}-2+\frac{2}{r'}})}^{\ell;\frac{4\beta}{\varepsilon}}\\
    &\quad
    \leqslant{}
    C
    \| a \|_{L^2(I;\dB_{q,1}^{\frac{d}{q}})}^2
    +
    C
    \| a \|_{L^r(I;\dB_{q,1}^{\frac{d}{q}-1+\frac{2}{r}})}
    \| a \|_{L^{r'}(I;\dB_{q,1}^{\frac{d}{q}-1+\frac{2}{r'}})}^{\ell;\frac{4\beta}{\varepsilon}}.
\end{align}
This completes the proof.
\end{proof}

\section{A priori estimates}\label{sec:a-propri}
In this section, we establish several a priori estimates for the solution to \eqref{eq:re_comp-1}.
Before we calculate the global a priori estimates, we prepare definitions and the basic properties of some norms of the initial data and the solutions.
\begin{df}\label{df1}
Let $\beta_0$ be the positive constant appearing in Lemma \ref{lemm:lin-Q-1}.
For $2 \leqslant q \leqslant \infty$, $\alpha>0$ and $0<\varepsilon \leqslant \beta_0/\alpha$ we define
\begin{align}
    \| (a_0,u_0)\|_{D_q^{\varepsilon}}
    :={}&
    \| (a_0,\mathbb{Q}u_0)\|_{D_q^{\varepsilon}}^{h;\alpha}
    +
    \| (a_0,\mathbb{Q}u_0)\|_{D^{\varepsilon}}^{\ell;\alpha}
    +
    \| \mathbb{P}u_0 \|_{\dB_{2,1}^{\frac{d}{2}-1}},
\end{align}
where
\begin{align}
    \| (a_0,\mathbb{Q}u_0)\|_{D_q^{\varepsilon}}^{h;\alpha}
    :={}&
    \varepsilon
    \| a_0 \|_{\dB_{q,1}^{\frac{d}{q}}}^{h;\frac{\beta_0}{\varepsilon}}
    +
    \| \mathbb{Q}u_0 \|_{\dB_{q,1}^{\frac{d}{q}-1}}^{h;\frac{\beta_0}{\varepsilon}}
    +
    \| (a_0,\mathbb{Q}u_0) \|_{\dB_{2,1}^{\frac{d}{2}-1}}^{m;\alpha,\frac{\beta_0}{\varepsilon}},\\
    \| (a_0,\mathbb{Q}u_0)\|_{D^{\varepsilon}}^{\ell;\alpha}
    :={}&
    \| (a_0, \mathbb{Q}u_0) \|_{\dB_{2,1}^{\frac{d}{2}-1}}^{\ell;\alpha}
\end{align}
for all 
$(a_0,u_0) \in (\dB_{2,1}^{\frac{d}{2}-1}(\mathbb{R}^d) \cap \dB_{2,1}^{\frac{d}{2}}(\mathbb{R}^d)) \times \dB_{2,1}^{\frac{d}{2}-1}(\mathbb{R}^d)^d$.
Moreover, we use the abbreviation $\| (a_0,u_0)\|_{D^{\varepsilon}}:=\| (a_0,u_0)\|_{D_2^{\varepsilon}}$ for $q=2$.
We note that it holds
\begin{align}
    \| (a_0,u_0)\|_{D^{\varepsilon}}
    \sim{}&
    \varepsilon \| a_0 \|_{\dB_{2,1}^{\frac{d}{2}}}^{h;\frac{\beta_0}{\varepsilon}}
    +
    \| a_0 \|_{\dB_{2,1}^{\frac{d}{2}-1}}^{\ell;\frac{\beta_0}{\varepsilon}}
    +
    \| u_0 \|_{\dB_{2,1}^{\frac{d}{2}-1}}.
\end{align}
\end{df}
\begin{df}\label{df2}
Let $\beta_0$ be the positive constant appearing in Lemma \ref{lemm:lin-Q-1}.
Let $2 \leqslant q,r \leqslant \infty$, $\alpha>0$, $0 < \varepsilon \leqslant \beta_0/\alpha$ and let $I \subset \mathbb{R}$ be an interval.
\begin{itemize}
\item[(1)]
Let $X^{\varepsilon}(I)$ be the set of all space-time distributions $(a,u)$ on $I \times \mathbb{R}^d$ with the following norm finite:
\begin{align}
    \| (a,u)\|_{X^{\varepsilon}(I)}
    :={}
    &
    \varepsilon \| a \|_{{L^{\infty}}(I;\dB_{2,1}^{\frac{d}{2}})}^{h;\frac{\beta_0}{\varepsilon}}
    +
    \frac{1}{\varepsilon}
    \| a \|_{L^1(I;\dB_{2,1}^{\frac{d}{2}})}^{h;\frac{\beta_0}{\varepsilon}}\\
    &
    +
    \| a \|_{
    {L^{\infty}}(I;\dB_{2,1}^{\frac{d}{2}-1})
    \cap 
    L^1(I;\dB_{2,1}^{\frac{d}{2}+1})}^{\ell;\frac{\beta_0}{\varepsilon}}\\
    &
    +
    \| u \|_{
    {L^{\infty}}(I;\dB_{2,1}^{\frac{d}{2}-1})
    \cap
    L^1(I;\dB_{2,1}^{\frac{d}{2}+1})}.
\end{align}
\item [(2)]
We define $Y^{\varepsilon,\alpha}_{q,r}(I)$ by the set of all space-time distributions $(a,u)$ on $I \times \mathbb{R}^d$ with the following norm finite:
\begin{align}
    \| (a,u)\|_{Y_{q,r}^{\varepsilon,\alpha}(I)}
    :={}
    \| (a,\mathbb{Q}u)\|_{Y_{q,r}^{\varepsilon,\alpha}(I)}^{h;\alpha}
    +
    \| (a,\mathbb{Q}u)\|_{Y_{q,r}^{\varepsilon,\alpha}(I)}^{\ell;\alpha}
    +
    \| \mathbb{P}u \|_{Y_{q,r}^{\varepsilon,\alpha}(I)},
\end{align}
where
\begin{align}
    \| (a,\mathbb{Q}u)\|_{Y_{q,r}^{\varepsilon,\alpha}(I)}^{h;\alpha}
    :={}
    &
    \varepsilon \| a \|_{{L^{\infty}}(I;\dB_{q,1}^{\frac{d}{q}})}^{h;\frac{\beta_0}{\varepsilon}}
    +
    \frac{1}{\varepsilon}
    \| a \|_{L^1(I;\dB_{q,1}^{\frac{d}{q}})}^{h;\frac{\beta_0}{\varepsilon}}\\
    &
    +
    \| \mathbb{Q}u \|_{{L^{\infty}}(I;\dB_{q,1}^{\frac{d}{q}-1})
    \cap 
    L^1(I;\dB_{q,1}^{\frac{d}{q}+1})}^{h;\frac{\beta_0}{\varepsilon}}\\
    &
    +
    \| (a,\mathbb{Q}u) \|_{{L^{\infty}}(I;\dB_{2,1}^{\frac{d}{2}-1})
    \cap 
    L^1(I;\dB_{2,1}^{\frac{d}{2}+1})}^{m;\alpha,\frac{\beta_0}{\varepsilon}},\\
    \| (a,\mathbb{Q}u)\|_{Y_{q,r}^{\varepsilon,\alpha}(I)}^{\ell;\alpha}
    :={}
    &
    \| (a, \mathbb{Q}u) \|_{{L^r}(I;\dB_{q,1}^{\frac{d}{q}-1+\frac{2}{r}})}^{\ell;\alpha},\\
    \| \mathbb{P}u \|_{Y_{q,r}^{\varepsilon,\alpha}(I)}
    :={}
    &
    \| \mathbb{P}u \|_{
    {L^r}(I;\dB_{q,1}^{\frac{d}{q}-1+\frac{2}{r}})
    \cap L^1(I;\dB_{q,1}^{\frac{d}{q}+1})}.
\end{align}
\item[(3)]
We set
\begin{align}
    A_{q,r}^{\varepsilon,\alpha}[a,u](I)
    :={}
    &
    \alpha \varepsilon
    \| (a, \mathbb{Q}u)\|_{X^{\varepsilon}(I)}
    +
    \| (a,u)\|_{Y_{q,r}^{\varepsilon,\alpha}(I)}\\
    &
    +
    \left(
    \| (a, \mathbb{Q}u)\|_{Y_{q,r}^{\varepsilon,\alpha}(I)}^{\ell;\alpha}
    \right)^{\frac{1}{r-1}}
    \| (a, \mathbb{Q}u)\|_{X^{\varepsilon}(I)}^{\frac{r-2}{r-1}}
\end{align}
for all $(a,u) \in X^{\varepsilon}(I)$.
\end{itemize}
\end{df}
\begin{rem}
Let us explain how the quantities defined in Definition \ref{df2} work in the proof of Theorem \ref{thm:2}.
$\| (a,u) \|_{X^{\varepsilon}(I)}$ represents the energy norm of the solution, which may be large if the initial data is large.
The quantity $A_{q,r}^{\varepsilon,\alpha}[a,u](I)$ gets arbitrarily small for suitable $\alpha$, provided that the Mach number $\varepsilon$ and the time integral of the incompressible solution $w$ of \eqref{main:eq-incomp} on some time interval $I$ are sufficiently small, even if the initial data is not necessarily small.
\end{rem}
By the Bernstein inequality and the interpolation, 
we immediately obtain the following estimates:
\begin{lemm}\label{lemm:emb}
Let $I \subset \mathbb{R}$ be an interval.
For $2 \leqslant q,r \leqslant \infty$,
there exists a positive constant $C=C(d,q,r)$ such that
\begin{align}
    \| (a,u) \|_{Y_{q,r}^{\varepsilon,\alpha}(I)}
    \leqslant
    A_{q,r}^{\varepsilon,\alpha}[a,u](I)
    \leqslant
    C
    \| (a,u) \|_{X^{\varepsilon}(I)}
\end{align}
for all $\alpha>0$, $0 < \varepsilon < \beta_0/\alpha$ and $(a,u) \in X^{\varepsilon}(I)$.
\end{lemm}
We investigate the relationship between $A_{q,r}^{\varepsilon,\alpha}[a,u](I)$ and the various space time norms of $(a,u)$.
\begin{lemm}\label{lemm:A}
Let $I \subset \mathbb{R}$ be an interval.
For $2 \leqslant q,r \leqslant \infty$,
there exists a positive constant $C_1=C_1(d,q,r)$ such that
\begin{align}\label{A:a}
    \begin{split}
    \varepsilon
    \| a \|_{L^{\infty}(I;\dB_{q,1}^{\frac{d}{q}})}
    &
    +
    \| a \|_{L^r(I;\dB_{q,1}^{\frac{d}{q}-1+\frac{2}{r}})}\\
    &
    +
    \| a \|_{L^2(I;\dB_{q,1}^{\frac{d}{q}})}
    +
    \| a \|_{L^{r'}(I;\dB_{q,1}^{\frac{d}{q}-1+\frac{2}{r'}})}^{\ell;\frac{4\beta_0}{\varepsilon}}
    \leqslant{}
    C_1
    A_{q,r}^{\varepsilon,\alpha}[a,\mathbb{Q}u](I),
    \end{split}
\end{align} 
and 
\begin{align}\label{A:u}
    \begin{split}
    \| u \|_{L^1(I;\dB_{q,1}^{\frac{d}{q}+1})}^{h;\frac{\beta_0}{\varepsilon}}
    &
    +
    \| u \|_{L^r(I;\dB_{q,1}^{\frac{d}{q}-1+\frac{2}{r}})}\\
    &
    +
    \| u \|_{L^2(I;\dB_{q,1}^{\frac{d}{q}})}
    +
    \| u \|_{L^{r'}(I;\dB_{q,1}^{\frac{d}{q}-1+\frac{2}{r'}})}
    \leqslant{}
    C_1
    A_{q,r}^{\varepsilon,\alpha}[a,u](I)
    \end{split}
\end{align}
for all $\alpha>0$, $\varepsilon \in (0,\beta_0/\alpha)$ and $(a,u) \in X^{\varepsilon}(I)$.
\end{lemm}
\begin{proof}
We first show \eqref{A:a}.
We see that
\begin{align}
    \varepsilon
    \| a \|_{L^{\infty}(I;\dB_{q,1}^{\frac{d}{q}})}
    \leqslant{}&
    \varepsilon
    \| a \|_{L^{\infty}(I;\dB_{q,1}^{\frac{d}{q}})}^{h;\frac{\beta_0}{\varepsilon}}
    +
    C
    \| a \|_{L^{\infty}(I;\dB_{2,1}^{\frac{d}{2}-1})}^{m;\alpha,\frac{\beta_0}{\varepsilon}}
    +
    C
    \alpha
    \varepsilon 
    \| a \|_{L^{\infty}(I;\dB_{2,1}^{\frac{d}{2}-1})}^{\ell;\alpha}\\
    \leqslant{}&
    C
    \| (a,\mathbb{Q}u) \|_{Y^{\varepsilon,\alpha}_{q,r}(I)}
    +
    C
    \varepsilon \alpha
    \| (a,\mathbb{Q}u) \|_{X^{\varepsilon}(I)}\\
    \leqslant{}&
    C
    A_{q,r}^{\varepsilon,\alpha}[a,\mathbb{Q}u](I)
\end{align}
and 
\begin{align}
    \| a \|_{L^r(I;\dB_{q,1}^{\frac{d}{q}-1+\frac{2}{r}})}
    \leqslant{}&
    C\left( \frac{\beta_0}{\varepsilon} \right)^{-1+\frac{2}{r}}
    \| a \|_{L^r(I;\dB_{q,1}^{\frac{d}{q}})}^{h;\frac{\beta_0}{\varepsilon}}
    +
    C\| a \|_{L^r(I;\dB_{2,1}^{\frac{d}{2}-1+\frac{2}{r}})}^{m;\alpha;\frac{\beta_0}{\varepsilon}}
    +
    \| a \|_{L^r(I;\dB_{q,1}^{\frac{d}{q}-1+\frac{2}{r}})}^{\ell;\alpha}\\
    \leqslant{}&
    C\varepsilon\| a \|_{L^{\infty}(I;\dB_{q,1}^{\frac{d}{q}})}^{h;\frac{\beta_0}{\varepsilon}}
    +
    \frac{C}{\varepsilon}
    \| a \|_{L^{\infty}(I;\dB_{q,1}^{\frac{d}{q}})}^{h;\frac{\beta_0}{\varepsilon}}\\
    &+
    C\| a \|_{L^{\infty}(I;\dB_{2,1}^{\frac{d}{2}-1}) \cap L^1(I;\dB_{2,1}^{\frac{d}{2}+1})}^{m;\alpha;\frac{\beta_0}{\varepsilon}}
    +
    \| a \|_{L^r(I;\dB_{q,1}^{\frac{d}{q}-1+\frac{2}{r}})}^{\ell;\alpha}\\
    \leqslant{}&
    C\| (a,\mathbb{Q}u) \|_{Y_{q,r}^{\varepsilon,\alpha}(I)}
    \leqslant
    C
    A_{q,r}^{\varepsilon,\alpha}[a,\mathbb{Q}u](I).
\end{align}
From the interpolation and the Bernstein inequlities, we see that
\begin{align}
    \| a \|_{L^2(I;\dB_{q,1}^{\frac{d}{q}})}
    \leqslant{}&
    \| a \|_{L^2(I;\dB_{q,1}^{\frac{d}{q}})}^{h;\frac{\beta_0}{\varepsilon}}
    +
    C\| a \|_{L^2(I;\dB_{2,1}^{\frac{d}{2}})}^{m;\alpha;\frac{\beta_0}{\varepsilon}}
    +
    \| a \|_{L^2(I;\dB_{q,1}^{\frac{d}{q}})}^{\ell;\alpha}\\
    \leqslant{}&
    \varepsilon\| a \|_{L^{\infty}(I;\dB_{q,1}^{\frac{d}{q}})}^{h;\frac{\beta_0}{\varepsilon}}
    +
    \frac{1}{\varepsilon}
    \| a \|_{L^{\infty}(I;\dB_{2,1}^{\frac{d}{2}})}^{h;\frac{\beta_0}{\varepsilon}}\\
    &+
    C\| a \|_{L^{\infty}(I;\dB_{2,1}^{\frac{d}{2}-1}) \cap L^1(I;\dB_{2,1}^{\frac{d}{2}+1})}^{m;\alpha;\frac{\beta_0}{\varepsilon}}\\
    &+
    \left( \| a \|_{L^r(I;\dB_{q,1}^{\frac{d}{q}-1+\frac{2}{r}})}^{\ell;\alpha}  \right)^{\frac{r}{2(r-1)}}
    \left( \| a \|_{L^1(I;\dB_{q,1}^{\frac{d}{q}+1})}^{\ell;\alpha}  \right)^{\frac{r-2}{2(r-1)}}\\
    \leqslant{}&
    \| (a,\mathbb{Q}u)\|_{Y_{q,r}^{\varepsilon,\alpha}(I)}
    +
    C
    \left(
    \| (a,\mathbb{Q}u)\|_{Y_{q,r}^{\varepsilon,\alpha}(I)}^{\ell;\alpha}
    \right)^{\frac{r}{2(r-1)}}
    \left(
    \| (a,\mathbb{Q}u)\|_{X^{\varepsilon}(I)}^{\ell;\alpha}
    \right)^{\frac{r-2}{2(r-1)}}\\
    \leqslant{}&
    C
    \| (a,\mathbb{Q}u)\|_{Y_{q,r}^{\varepsilon,\alpha}(I)}
    +
    C
    \left(
    \| (a,\mathbb{Q}u)\|_{Y_{q,r}^{\varepsilon,\alpha}(I)}^{\ell;\alpha}
    \right)^{\frac{1}{r-1}}
    \left(
    \| (a,\mathbb{Q}u)\|_{X^{\varepsilon}(I)}^{\ell;\alpha}
    \right)^{\frac{r-2}{r-1}}\\
    \leqslant{}&
    CA_{q,r}^{\varepsilon,\alpha}[a,\mathbb{Q}u](I)
\end{align}
Here, we have used the following equality:
\begin{align}
    A^{\frac{r}{2(r-1)}}B^{\frac{r-2}{2(r-1)}}
    \leqslant
    A + A^{\frac{1}{r-1}}B^{\frac{r-2}{r-1}},
    \qquad
    A,B \geqslant 0.
\end{align}
Similarly, we have
\begin{align}
    &\| a \|_{L^{r'}(I;\dB_{q,1}^{\frac{d}{q}-1+\frac{2}{r'}})}^{\ell;\frac{4\beta_0}{\varepsilon}}\\
    &\quad 
    \leqslant{}
    C\left( \frac{4\beta_0}{\varepsilon} \right)^{-1+\frac{2}{r'}}
    \| a \|_{L^{r'}(I;\dB_{q,1}^{\frac{d}{q}})}^{m;\frac{\beta_0}{\varepsilon},\frac{4\beta_0}{\varepsilon}}
    +
    C\| a \|_{L^{r'}(I;\dB_{2,1}^{\frac{d}{2}-1+\frac{2}{r'}})}^{m;\alpha;\frac{\beta_0}{\varepsilon}}
    +
    \| a \|_{L^{r'}(I;\dB_{q,1}^{\frac{d}{q}-1+\frac{2}{r'}})}^{\ell;\alpha}\\
    &\quad\leqslant{}
    C\varepsilon
    \| a \|_{L^{\infty}(I;\dB_{q,1}^{\frac{d}{q}})}^{h;\frac{\beta_0}{\varepsilon}}
    +
    \frac{C}{\varepsilon}
    \| a \|_{L^{\infty}(I;\dB_{q,1}^{\frac{d}{q}})}^{h;\frac{\beta_0}{\varepsilon}}\\
    &\qquad
    +
    C
    \| a \|_{L^{\infty}(I;\dB_{2,1}^{\frac{d}{2}-1}) \cap L^1(I;\dB_{2,1}^{\frac{d}{2}+1})}^{m;\alpha;\frac{\beta_0}{\varepsilon}}\\
    &\qquad+
    \left( \| a \|_{L^r(I;\dB_{q,1}^{\frac{d}{q}-1+\frac{2}{r}})}^{\ell;\alpha}  \right)^{\frac{1}{r-1}}
    \left( \| a \|_{L^1(I;\dB_{q,1}^{\frac{d}{q}+1})}^{\ell;\alpha}  \right)^{\frac{r-2}{r-1}}\\
    &\quad\leqslant{}
    C
    \| (a,\mathbb{Q}u)\|_{Y_{q,r}^{\varepsilon,\alpha}(I)}
    +
    C
    \left(
    \| (a,\mathbb{Q}u)\|_{Y_{q,r}^{\varepsilon,\alpha}(I)}^{\ell;\alpha}
    \right)^{\frac{1}{r-1}}
    \left(
    \| (a,\mathbb{Q}u)\|_{X^{\varepsilon}(I)}^{\ell;\alpha}
    \right)^{\frac{r-2}{r-1}}\\
    &\quad\leqslant{}
    CA_{q,r}^{\varepsilon,\alpha}[a,\mathbb{Q}u](I).
\end{align}
Thus, we obtain \eqref{A:a}.

Next, we prove \eqref{A:u}.
By $u=\mathbb{P}u+\mathbb{Q}u$, we have
\begin{align}
    &
    \| u \|_{L^1(I;\dB_{q,1}^{\frac{d}{q}+1})}^{h;\frac{\beta_0}{\varepsilon}}
    +
    \| u \|_{L^r(I;\dB_{q,1}^{\frac{d}{q}-1+\frac{2}{r}})}
    +
    \| u \|_{L^2(I;\dB_{q,1}^{\frac{d}{q}})}
    +
    \| u \|_{L^{r'}(I;\dB_{q,1}^{\frac{d}{q}-1+\frac{2}{r'}})}\\
    &\quad\leqslant
    \| \mathbb{Q}u \|_{L^1(I;\dB_{q,1}^{\frac{d}{q}+1})}^{h;\frac{\beta_0}{\varepsilon}}
    +
    \| \mathbb{Q}u \|_{L^r(I;\dB_{q,1}^{\frac{d}{q}-1+\frac{2}{r}})}
    +
    \| \mathbb{Q}u \|_{L^2(I;\dB_{q,1}^{\frac{d}{q}})}
    +
    \| \mathbb{Q}u \|_{L^{r'}(I;\dB_{q,1}^{\frac{d}{q}-1+\frac{2}{r'}})}\\
    &\qquad
    +
    C
    \| \mathbb{P}u \|_{L^r(I;\dB_{q,1}^{\frac{d}{q}-1+\frac{2}{r}}) \cap L^1(I;\dB_{q,1}^{\frac{d}{q}+1})}\\
    &\quad\leqslant
    \| \mathbb{Q}u \|_{L^r(I;\dB_{q,1}^{\frac{d}{q}-1+\frac{2}{r}})}
    +
    \| \mathbb{Q}u \|_{L^2(I;\dB_{q,1}^{\frac{d}{q}})}
    +
    \| \mathbb{Q}u \|_{L^{r'}(I;\dB_{q,1}^{\frac{d}{q}-1+\frac{2}{r'}})}
    +
    C\|(a,u) \|_{Y_{q,r}^{\varepsilon,\alpha}(I)}.
\end{align}
By the similar calculations as in the proof \eqref{A:a}, we have
\begin{align}
    \| \mathbb{Q}u \|_{L^r(I;\dB_{q,1}^{\frac{d}{q}-1+\frac{2}{r}})}
    \leqslant{}&
    C\| \mathbb{Q}u \|_{L^{\infty}(I;\dB_{q,1}^{\frac{d}{q}-1}) \cap L^1(I;\dB_{q,1}^{\frac{d}{q}+1})}^{h;\frac{\beta_0}{\varepsilon}}\\
    &
    +
    C\| \mathbb{Q}u \|_{L^{\infty}(I;\dB_{2,1}^{\frac{d}{2}-1}) \cap L^1(I;\dB_{2,1}^{\frac{d}{2}+1})}^{m;\alpha;\frac{\beta_0}{\varepsilon}}
    +
    \| \mathbb{Q}u \|_{L^r(I;\dB_{q,1}^{\frac{d}{q}-1+\frac{2}{r}})}^{\ell;\alpha}\\
    \leqslant{}&
    C\| (a,\mathbb{Q}u) \|_{Y_{q,r}^{\varepsilon,\alpha}(I)}
    \leqslant
    C
    A_{q,r}^{\varepsilon,\alpha}[a,\mathbb{Q}u](I).
\end{align}
We also see that
\begin{align}
    &\| \mathbb{Q}u \|_{L^2(I;\dB_{q,1}^{\frac{d}{q}})}\\
    &\quad\leqslant{}
    C\| \mathbb{Q}u \|_{L^{\infty}(I;\dB_{q,1}^{\frac{d}{q}-1}) \cap L^1(I;\dB_{q,1}^{\frac{d}{q}+1})}^{h;\frac{\beta_0}{\varepsilon}}
    +
    C\| \mathbb{Q}u \|_{L^{\infty}(I;\dB_{2,1}^{\frac{d}{2}-1}) \cap L^1(I;\dB_{2,1}^{\frac{d}{2}+1})}^{m;\alpha;\frac{\beta_0}{\varepsilon}}\\
    &\qquad+
    \left( \| \mathbb{Q}u \|_{L^r(I;\dB_{q,1}^{\frac{d}{q}-1+\frac{2}{r}})}^{\ell;\alpha}  \right)^{\frac{r}{2(r-1)}}
    \left( \| \mathbb{Q}u \|_{L^1(I;\dB_{q,1}^{\frac{d}{q}+1})}^{\ell;\alpha}  \right)^{\frac{r-2}{2(r-1)}}\\
    &\quad 
    \leqslant{}
    \| (a,\mathbb{Q}u)\|_{Y_{q,r}^{\varepsilon,\alpha}(I)}
    +
    C
    \left(
    \| (a,\mathbb{Q}u)\|_{Y_{q,r}^{\varepsilon,\alpha}(I)}^{\ell;\alpha}
    \right)^{\frac{r}{2(r-1)}}
    \left(
    \| (a,\mathbb{Q}u)\|_{X^{\varepsilon}(I)}^{\ell;\alpha}
    \right)^{\frac{r-2}{2(r-1)}}\\
    &\quad\leqslant{}
    C
    \| (a,\mathbb{Q}u)\|_{Y_{q,r}^{\varepsilon,\alpha}(I)}
    +
    C
    \left(
    \| (a,\mathbb{Q}u)\|_{Y_{q,r}^{\varepsilon,\alpha}(I)}^{\ell;\alpha}
    \right)^{\frac{1}{r-1}}
    \left(
    \| (a,\mathbb{Q}u)\|_{X^{\varepsilon}(I)}^{\ell;\alpha}
    \right)^{\frac{r-2}{r-1}}\\
    &\quad\leqslant{}
    CA_{q,r}^{\varepsilon,\alpha}[a,\mathbb{Q}u](I).
\end{align}
Similarily, we have
\begin{align}
    &\| \mathbb{Q}u \|_{L^{r'}(I;\dB_{q,1}^{\frac{d}{q}-1+\frac{2}{r'}})}\\
    &\quad\leqslant{}
    C\| \mathbb{Q}u \|_{L^{\infty}(I;\dB_{q,1}^{\frac{d}{q}-1}) \cap L^1(I;\dB_{q,1}^{\frac{d}{q}+1})}^{h;\frac{\beta_0}{\varepsilon}}
    +
    C\| \mathbb{Q}u \|_{L^{\infty}(I;\dB_{2,1}^{\frac{d}{2}-1}) \cap L^1(I;\dB_{2,1}^{\frac{d}{2}+1})}^{m;\alpha;\frac{\beta_0}{\varepsilon}}\\
    &\qquad+
    \left( \| \mathbb{Q}u \|_{L^r(I;\dB_{q,1}^{\frac{d}{q}-1+\frac{2}{r}})}^{\ell;\alpha}  \right)^{\frac{1}{r-1}}
    \left( \| \mathbb{Q}u \|_{L^1(I;\dB_{q,1}^{\frac{d}{q}+1})}^{\ell;\alpha}  \right)^{\frac{r-2}{r-1}}\\
    &\quad 
    \leqslant{}
    C
    \| (a,\mathbb{Q}u)\|_{Y_{q,r}^{\varepsilon,\alpha}(I)}
    +
    C
    \left(
    \| (a,\mathbb{Q}u)\|_{Y_{q,r}^{\varepsilon,\alpha}(I)}^{\ell;\alpha}
    \right)^{\frac{1}{r-1}}
    \left(
    \| (a,\mathbb{Q}u)\|_{X^{\varepsilon}(I)}^{\ell;\alpha}
    \right)^{\frac{r-2}{r-1}}\\
    &\quad\leqslant{}
    CA_{q,r}^{\varepsilon,\alpha}[a,\mathbb{Q}u](I).
\end{align}
Hence, we complete the proof.
\end{proof}

Now, we consider the global a priori estimates for the solution to \eqref{eq:re_comp-1}.
First of all, 
we recall the local well-posedness of  \eqref{eq:re_comp-1}.
\begin{lemm}[\cites{Che-Mia-Zha-10,Dan-14}]\label{lemm:LWP}
Let $\varepsilon>0$ and
let
$(a_0,u_0) \in  \dB_{2,1}^{\frac{d}{2}}(\mathbb{R}^d) \times \dB_{2,1}^{\frac{d}{2}-1}(\mathbb{R}^d)^d$
satisfy $\inf_{x \in \mathbb{R}^d}(1+\varepsilon a_0(x))>0$.
Then, there exists a positive time $\widetilde{T_0^{\varepsilon}}=\widetilde{T_0^{\varepsilon}}(d,a_0,u_0)$ such that \eqref{eq:re_comp-1} possesses a unique solution $(a,u)$ satisfying 
\begin{align}
    a^{\varepsilon} \in C([0,\widetilde{T_0^{\varepsilon}}];\dB_{2,1}^{\frac{d}{2}}(\mathbb{R}^d)),\quad
    u^{\varepsilon} \in C([0,\widetilde{T_0^{\varepsilon}}];\dB_{2,1}^{\frac{d}{2}-1}(\mathbb{R}^d))^d \cap L^1(0,\widetilde{T_0^{\varepsilon}};\dB_{2,1}^{\frac{d}{2}+1}(\mathbb{R}^d))^d
\end{align}
with $\rho^{\varepsilon}(t,x) = 1 + \varepsilon a^{\varepsilon}(t,x) > 0$.
Moreover, if in addition $a_0 \in \dB_{2,1}^{\frac{d}{2}-1}(\mathbb{R}^d)$, then $(a^{\varepsilon},u^{\varepsilon}) \in X^{\varepsilon}(0,T)$ for all $0<T<T_{\rm max}^{\varepsilon}$, where $T_{\rm max}^{\varepsilon}$ denotes the maximal existence time.
\end{lemm}
In the following three lemmas,  
applying the linear estimates in Section \ref{sec:lin} to \eqref{eq:re_comp-1} with the nonlinear estimates in Section \ref{sec:nonlin}, 
we establish global a priori estimates for the local solution constructed in Lemma \ref{lemm:LWP}.
\begin{lemm}\label{lemm:X}
Let $\alpha > 0$, $0< \varepsilon < \beta_0/\alpha$ and let $q$ and $r$ satisfy
\begin{align}\label{ene:p-q-r-1}
    2\leqslant{}  q \leqslant 4,\qquad
    0 \leqslant \frac{1}{r} \leqslant \frac{1}{2} - \frac{d}{2}\left( \frac{1}{2} - \frac{1}{q} \right),\qquad
    \frac{2d}{q} - 1 >0.
\end{align}
Let $(a^{\varepsilon},u^{\varepsilon})$ be the local solution to \eqref{eq:re_comp-1} on $[0,T_{\rm max}^{\varepsilon})$ with the initial data $(a_0,u_0) \in (\dB_{2,1}^{\frac{d}{2}-1}(\mathbb{R}^d) \cap \dB_{2,1}^{\frac{d}{2}}(\mathbb{R}^d)) \times \dB_{2,1}^{\frac{d}{2}-1}(\mathbb{R}^d)^d $,
where $T_{\rm max}^{\varepsilon}$ denotes the maximal existence time.
Then, 
there exists a positive constant $C_2=C_2(d,q,r,\mu,P)$ such that  
for any interval $I \subset [0,T^{\varepsilon}_{\rm max})$ with the infimum $t_0$,
\begin{align}\label{est:X}
    \| (a^{\varepsilon},u^{\varepsilon}) \|_{X^{\varepsilon}(I)}
    \leqslant{}
    &
    C_2
    \| (a^{\varepsilon},u^{\varepsilon})(t_0) \|_{D^{\varepsilon}}
    +
    C_2
    A_{q,r}^{\varepsilon,\alpha}[a^{\varepsilon},u^{\varepsilon}](I)
    \| (a^{\varepsilon}, u^{\varepsilon})\|_{X^{\varepsilon}(I)}
\end{align}
provided that $\varepsilon\| a^{\varepsilon} \|_{L^{\infty}(I;L^{\infty}\cap\dB_{q,1}^{\frac{d}{q}})} \leqslant 1/2$.
\end{lemm}
\begin{proof}
We first consider the compressible part.
Applying Lemma \ref{lemm:lin-Q-1} to \eqref{eq:re_comp-1} and using Lemmas \ref{lemm:nonlin-2}, \ref{lemm:nonlin-1}, \ref{lemm:nonlin-3}, \ref{lemm:nonlin-4}, \ref{lemm:nonlin-5} and \ref{lemm:A}, we have 
\begin{align}\label{X-1}
    \begin{split}
    &\varepsilon \| a^{\varepsilon} \|_{{L^{\infty}}(I;\dB_{2,1}^{\frac{d}{2}})}^{h;\frac{\beta_0}{\varepsilon}}
    +
    \frac{1}{\varepsilon}
    \| a^{\varepsilon} \|_{L^1(I;\dB_{2,1}^{\frac{d}{2}})}^{h;\frac{\beta_0}{\varepsilon}}
    +
    \| \mathbb{Q}u^{\varepsilon} \|_{{L^{\infty}}(I;\dB_{2,1}^{\frac{d}{2}-1})\cap L^1(I;\dB_{2,1}^{\frac{d}{2}+1})}^{h;\frac{\beta_0}{\varepsilon}}\\
    &\quad 
    \leqslant{}
    C
    \bigg(\varepsilon\|a^{\varepsilon}(t_0)\|_{\dB_{2,1}^{\frac{d}{2}}}^{h;\frac{\beta_0}{\varepsilon}}
    +
    \| \mathbb{Q}u^{\varepsilon}(t_0) \|_{\dB_{2,1}^{\frac{d}{2}-1}}^{h;\frac{\beta_0}{\varepsilon}}
    \bigg)\\
    &\qquad
    + 
    C\|a^{\varepsilon}u^{\varepsilon}\|_{L^1(I;\dB_{2,1}^{\frac{d}{2}})}^{h;\frac{\beta_0}{\varepsilon}}
    + 
    C\|\mathcal{N}^{\varepsilon}[a^{\varepsilon},u^{\varepsilon}]\|_{L^1(I;\dB_{2,1}^{\frac{d}{2}-1})}^{h;\frac{\beta_0}{\varepsilon}}
    +
    C\varepsilon 
    \| a^{\varepsilon} \div \mathbb{Q}u^{\varepsilon} \|_{L^1(I;\dB_{2,1}^{\frac{d}{2}})}^{h;\frac{\beta_0}{\varepsilon}}\\
    &\qquad
    +
    C\varepsilon\sum_{2^j \geqslant \frac{\beta_0}{\varepsilon}}2^{\frac{d}{2}j}
    \| [u^{\varepsilon}\cdot \nabla, \dot{\Delta}_j] a^{\varepsilon}\|_{L^1(I;L^2)}
    +
    C\varepsilon
    \left\|
        \| \div u^{\varepsilon} \|_{L^{\infty}}
        \| a^{\varepsilon} \|_{\dB_{2,1}^{\frac{d}{2}}}
    \right\|_{L^1(0,t)}\\
    &\quad 
    \leqslant{}
    C\| (a^{\varepsilon},\mathbb{Q}u^{\varepsilon})(t_0) \|_{D^{\varepsilon}}^{h;\frac{\beta_0}{\varepsilon}}
    +
    C 
    A_{q,r}^{\varepsilon,\alpha}[a^{\varepsilon},u^{\varepsilon}](I)
    \| (a^{\varepsilon}, u^{\varepsilon})\|_{X^{\varepsilon}(I)}.
    \end{split}
\end{align}
From Lemma \ref{lemm:lin-Q-2} and Lemmas \ref{lemm:nonlin-1}, \ref{lemm:nonlin-3}, \ref{lemm:nonlin-4}, \ref{lemm:nonlin-5} and \ref{lemm:A}, it follows that
\begin{align}\label{X-2}
    \begin{split}
    &
    \| (a^{\varepsilon},\mathbb{Q}u^{\varepsilon}) \|_{{L^{\infty}}(I;\dB_{2,1}^{\frac{d}{2}-1})\cap L^1(I;\dB_{2,1}^{\frac{d}{2}+1})}^{\ell;\frac{\beta_0}{\varepsilon}}\\
    &\quad 
    \leqslant{}
    C
    \| (a^{\varepsilon},\mathbb{Q}u^{\varepsilon})(t_0) \|_{\dB_{2,1}^{\frac{d}{2}-1}}^{\ell;\frac{\beta_0}{\varepsilon}}
    +
    \| a^{\varepsilon}u^{\varepsilon} \|_{L^1(I;\dB_{2,1}^{\frac{d}{2}})}^{\ell;\frac{\beta_0}{\varepsilon}}
    +
    C
    \|\mathcal{N}^{\varepsilon}[a^{\varepsilon},u^{\varepsilon}]\|_{L^1(I;\dB_{2,1}^{\frac{d}{2}-1})}^{\ell;\frac{\beta_0}{\varepsilon}}\\ 
    &\quad 
    \leqslant{}
    C\| (a^{\varepsilon},\mathbb{Q}u^{\varepsilon})(t_0) \|_{D^{\varepsilon}}^{\ell;\frac{\beta_0}{\varepsilon}}
    +
    C 
    A_{q,r}^{\varepsilon,\alpha}[a^{\varepsilon},u^{\varepsilon}](I)
    \| (a^{\varepsilon}, u^{\varepsilon})\|_{X^{\varepsilon}(I)}.
    \end{split}
\end{align}

Next, we consider the incompressible part.
Applying $\mathbb{P}$ to the second equation of \eqref{eq:re_comp-1}, we see that
\begin{align}
    \begin{dcases}
        \partial_t \mathbb{P}u^{\varepsilon} - \mu \Delta \mathbb{P}u^{\varepsilon} = - \mathbb{P}\mathcal{N}^{\varepsilon}[a^{\varepsilon},u^{\varepsilon}],&\qquad t > 0, x \in \mathbb{R}^d,\\
        \mathbb{P}u^{\varepsilon}(0,x) = \mathbb{P}u_0(x),&\qquad x \in \mathbb{R}^d.
    \end{dcases}
\end{align}
Then, by Lemma \ref{lemm:lin-P-1} and Lemmas \ref{lemm:nonlin-3}, \ref{lemm:nonlin-4}, \ref{lemm:nonlin-5} and \ref{lemm:A}, it holds
\begin{align}\label{X-3}
    \| \mathbb{P}u^{\varepsilon} \|_{{L^{\infty}}(I;\dB_{2,1}^{\frac{d}{2}-1})\cap L^1(I;\dB_{2,1}^{\frac{d}{2}+1})}
    \leqslant{}&
    C\| \mathbb{P}u^{\varepsilon}(t_0) \|_{\dB_{2,1}^{\frac{d}{2}-1}}
    +
    C\| \mathcal{N}^{\varepsilon}[a^{\varepsilon},u^{\varepsilon}] \|_{L^1(I;\dB_{2,1}^{\frac{d}{2}-1})}\\
    \leqslant{}&
    C\| \mathbb{P}u^{\varepsilon}(t_0) \|_{\dB_{2,1}^{\frac{d}{2}-1}}
    +
    C
    A_{q,r}^{\varepsilon,\alpha}[a^{\varepsilon},u^{\varepsilon}](I)
    \| (a^{\varepsilon}, u^{\varepsilon})\|_{X^{\varepsilon}(I)}.
\end{align}
Hence, combining \eqref{X-1}, \eqref{X-2} and \eqref{X-3}, 
we complete the proof.
\end{proof}

\begin{lemm}\label{lemm:Y}
Let $\alpha>0$ and $0< \varepsilon < \beta_0/\alpha$. 
Let  $2 \leqslant q \leqslant 4$ and $2 \leqslant r \leqslant \infty$ satisfy
\begin{gather}\label{Y-p_q_r}
    \frac{1}{r} \leqslant \min \left\{ \frac{1}{2} - \frac{d}{2}\left( \frac{1}{2} - \frac{1}{q} \right),\frac{d-1}{2}\left( \frac{1}{2} - \frac{1}{q} \right)\right\},\quad
    \frac{2d}{q} - 1 > 0.
\end{gather}
Let $(a^{\varepsilon},u^{\varepsilon})$ be the local solution to \eqref{eq:re_comp-1} on $[0,T_{\rm max}^{\varepsilon})$ with the initial data $(a_0,u_0) \in (\dB_{2,1}^{\frac{d}{2}-1}(\mathbb{R}^d) \cap \dB_{2,1}^{\frac{d}{2}}(\mathbb{R}^d)) \times \dB_{2,1}^{\frac{d}{2}-1}(\mathbb{R}^d)^d $,
where $T_{\rm max}^{\varepsilon}$ denotes the maximal existence time.
Let $w$ be a global solution to the incompressible Navier--Stokes equation \eqref{main:eq-incomp}.
Then, there exists a positive constant $C_3=C_3(d,q,r,\mu,P)$ such that
for any interval $I \subset [0,T_{\rm max}^{\varepsilon})$ with the infimum $t_0$, it holds
\begin{align}
&
\begin{aligned}\label{est:Y-comp-high}
    \| (a^{\varepsilon},\mathbb{Q}u^{\varepsilon}) \|_{Y_{q,r}^{\varepsilon,\alpha}(I)}^{h;\alpha}
    \leqslant{}&
    C_3 
    \left( 
    \| (a^{\varepsilon},\mathbb{Q}u^{\varepsilon})(t_0) \|_{D_q^{\varepsilon}}^{h;\alpha}
    +
    \| (w \cdot \nabla)w  \|_{L^1(I;\dB_{2,1}^{\frac{d}{2}-1})}^{h;\alpha}
    \right)\\
    &
    +
    C_3
    \widetilde{A_{q,r}^{\varepsilon,\alpha}}[a^{\varepsilon},u^{\varepsilon};w](I)
    A_{q,r}^{\varepsilon,\alpha}[a^{\varepsilon},u^{\varepsilon}-w](I),
\end{aligned}\\
&
\begin{aligned}\label{est:Y-comp-low}
    \| (a^{\varepsilon},\mathbb{Q}u^{\varepsilon}) \|_{Y_{q,r}^{\varepsilon,\alpha}(I)}^{\ell;\alpha}
    \leqslant{}&
    C_3 
    (\alpha \varepsilon)^{\frac{1}{r}}
    \| (a^{\varepsilon},\mathbb{Q}u^{\varepsilon})(t_0) \|_{D^{\varepsilon}}^{\ell;\alpha}\\
    &
    +
    C_3(\alpha \varepsilon)^{\frac{1}{r}}
    A_{q,r}^{\varepsilon,\alpha}[a^{\varepsilon},u^{\varepsilon}](I)
    \| (a^{\varepsilon},u^{\varepsilon}) \|_{X^{\varepsilon}(I)},
\end{aligned}
\end{align}
and
\begin{align}\label{est:Y-incomp}
    \begin{split}
    \| \mathbb{P}u^{\varepsilon}-w \|_{
    {L^{\infty}}(I;\dB_{q,1}^{\frac{d}{q}-1})
    \cap
    L^1(I;\dB_{q,1}^{\frac{d}{q}+1})}
    \leqslant{}&
    C_3
    \| (\mathbb{P}u^{\varepsilon} - w)(t_0) \|_{\dB_{q,1}^{\frac{d}{q}-1}}\\
    &
    +
    C_3
    \widetilde{A_{q,r}^{\varepsilon,\alpha}}[a^{\varepsilon},u^{\varepsilon};w](I)
    A_{q,r}^{\varepsilon,\alpha}[a^{\varepsilon},u^{\varepsilon}-w](I),
    \end{split}
\end{align}
provided that $\varepsilon\| a^{\varepsilon} \|_{L^{\infty}(I;L^{\infty}\cap\dB_{q,1}^{\frac{d}{q}})} \leqslant 1/2$.
Here, we have put 
\begin{align}\label{At}
    \widetilde{A_{q,r}^{\varepsilon,\alpha}}[a^{\varepsilon},u^{\varepsilon};w](I)
    :=
    \| w \|_{L^{r}(I;\dB_{q,1}^{\frac{d}{q}-1+\frac{2}{r}}) \cap L^1(I;\dB_{q,1}^{\frac{d}{q}+1})} 
    +  
    A_{q,r}^{\varepsilon,\alpha}[a^{\varepsilon},u^{\varepsilon}](I).
\end{align}
\end{lemm}
\begin{proof}
Let $v^{\varepsilon}:=u^{\varepsilon}-w$.
We first show \eqref{est:Y-comp-high}.
It follows from Lemma \ref{lemm:lin-Q-1} that 
\begin{align}\label{Y-h-1}
    \begin{split}
    &\varepsilon \| a^{\varepsilon} \|_{{L^{\infty}}(I;\dB_{q,1}^{\frac{d}{q}})}^{h;\frac{\beta_0}{\varepsilon}}
    +
    \frac{1}{\varepsilon}
    \| a^{\varepsilon} \|_{L^1(I;\dB_{q,1}^{\frac{d}{q}})}^{h;\frac{\beta_0}{\varepsilon}}
    +
    \| \mathbb{Q}u^{\varepsilon} \|_{{L^{\infty}}(I;\dB_{q,1}^{\frac{d}{q}-1})\cap L^1(I;\dB_{q,1}^{\frac{d}{2}+1})}^{h;\frac{\beta_0}{\varepsilon}}\\
    &\quad 
    \leqslant{}
    C
    \bigg(\varepsilon\|a^{\varepsilon}(t_0)\|_{\dB_{q,1}^{\frac{d}{q}}}^{h;\frac{\beta_0}{\varepsilon}}
    +
    \| \mathbb{Q}u^{\varepsilon}(t_0) \|_{\dB_{q,1}^{\frac{d}{q}-1}}^{h;\frac{\beta_0}{\varepsilon}}
    \bigg)\\
    &\qquad
    + 
    C\|a^{\varepsilon}u^{\varepsilon}\|_{L^1(I;\dB_{q,1}^{\frac{d}{q}})}^{h;\frac{\beta_0}{\varepsilon}}
    + 
    C\|\mathcal{N}^{\varepsilon}[a^{\varepsilon},u^{\varepsilon}]\|_{L^1(I;\dB_{q,1}^{\frac{d}{q}-1})}^{h;\frac{\beta_0}{\varepsilon}}
    +
    C\varepsilon 
    \| a^{\varepsilon} \div \mathbb{Q}u^{\varepsilon} \|_{L^1(I;\dB_{q,1}^{\frac{d}{q}})}^{h;\frac{\beta_0}{\varepsilon}}\\
    &\qquad
    +
    C\varepsilon\sum_{2^j \geqslant \frac{\beta_0}{\varepsilon}}2^{\frac{d}{q}j}
    \| [u^{\varepsilon}\cdot \nabla, \dot{\Delta}_j] a^{\varepsilon}\|_{L^1(I;L^q)}
    +
    C\varepsilon
    \left\|
        \| \div u^{\varepsilon} \|_{L^{\infty}}
        \| a^{\varepsilon} \|_{\dB_{q,1}^{\frac{d}{q}}}
    \right\|_{L^1(0,t)}.
    \end{split}
\end{align}
Using Lemmas \ref{lemm:nonlin-2}, \ref{lemm:nonlin-1} and \ref{lemm:A}, we have 
\begin{align}\label{Y-h-2}
    \begin{split}
    &
    \varepsilon\sum_{2^j \geqslant \frac{\beta_0}{\varepsilon}}2^{\frac{d}{q}j}
    \| [u^{\varepsilon}\cdot \nabla, \dot{\Delta}_j] a^{\varepsilon}\|_{L^1(I;L^q)}
    +
    \varepsilon 
    \| a^{\varepsilon} \div \mathbb{Q}u^{\varepsilon} \|_{L^1(I;\dB_{q,1}^{\frac{d}{q}})}^{h;\frac{\beta_0}{\varepsilon}}\\
    &\quad
    +
    \varepsilon
    \left\|
        \| \div u^{\varepsilon} \|_{L^{\infty}}
        \| a^{\varepsilon} \|_{\dB_{q,1}^{\frac{d}{q}}}
    \right\|_{L^1(0,t)}
    +
    \|a^{\varepsilon}u^{\varepsilon}\|_{L^1(I;\dB_{q,1}^{\frac{d}{q}})}^{h;\frac{\beta_0}{\varepsilon}}\\
    &\qquad 
    \leqslant{}
    C
    \varepsilon\| a \|_{L^{\infty}(I;\dB_{q,1}^{\frac{d}{q}})}
    \| u^{\varepsilon} \|_{L^1( I; \dB_{q,1}^{\frac{d}{q}+1} )}^{h;\frac{\beta_0}{\varepsilon}}
    +
    C
    \| a^{\varepsilon} \|_{L^2( I ; \dB_{q,1}^{\frac{d}{q}} )}
    \| u^{\varepsilon} \|_{L^2( I ; \dB_{q,1}^{\frac{d}{q}} )}\\
    &\qquad 
    \leqslant{}
    C
    A_{q,r}^{\varepsilon,\alpha}[a^{\varepsilon},\mathbb{Q}u^{\varepsilon}](I)
    A_{q,r}^{\varepsilon,\alpha}[a^{\varepsilon},u^{\varepsilon}](I)\\
    &\qquad 
    \leqslant{}
    C
    A_{q,r}^{\varepsilon,\alpha}[a^{\varepsilon},v^{\varepsilon}](I)
    A_{q,r}^{\varepsilon,\alpha}[a^{\varepsilon},u^{\varepsilon}](I).
    \end{split}
\end{align}
For the estimates of $\mathcal{N}^{\varepsilon}[a^{\varepsilon},u^{\varepsilon}]$, we decompose it as follows.
\begin{align}\label{Y-h-3}
    \begin{split}
    \|\mathcal{N}^{\varepsilon}[a^{\varepsilon},u^{\varepsilon}]\|_{L^1(I;\dB_{q,1}^{\frac{d}{q}-1})}^{h;\frac{\beta_0}{\varepsilon}}
    \leqslant{}&
    \|(w \cdot \nabla)w \|_{L^1(I;\dB_{q,1}^{\frac{d}{q}-1})}^{h;\frac{\beta_0}{\varepsilon}}\\
    &
    +
    \|\mathcal{N}^{\varepsilon}[a^{\varepsilon},u^{\varepsilon}]-(w \cdot \nabla)w \|_{L^1(I;\dB_{q,1}^{\frac{d}{q}-1})}^{h;\frac{\beta_0}{\varepsilon}}.
    \end{split}
\end{align}
Since 
\begin{align}\label{Y-h-4}
    \begin{split}
    \mathcal{N}^{\varepsilon}[a^{\varepsilon},u^{\varepsilon}]-(w \cdot \nabla)w
    ={}&
    (v^{\varepsilon} \cdot \nabla)u^{\varepsilon} + (w \cdot \nabla)v^{\varepsilon} \\
    &
    + 
    \mathcal{J}(\varepsilon a^{\varepsilon})\mathcal{L}u^{\varepsilon}
    +
    \frac{1}{\varepsilon}\mathcal{K}(\varepsilon a^{\varepsilon})\nabla a^{\varepsilon},
    \end{split}
\end{align}
we see by Lemmas \ref{lemm:nonlin-3}, \ref{lemm:nonlin-4} and \ref{lemm:nonlin-5} that
\begin{align}\label{Y-h-5}
    \begin{split}
    &\|\mathcal{N}^{\varepsilon}[a^{\varepsilon},u^{\varepsilon}]-(w \cdot \nabla)w \|_{L^1(I;\dB_{q,1}^{\frac{d}{q}-1})}^{h;\frac{\beta_0}{\varepsilon}}\\
    &\quad
    \leqslant
    C
    \| (w,u^{\varepsilon}) \|_{L^{r}(I;\dB_{q,1}^{\frac{d}{q}-1+\frac{2}{r}}) \cap L^{r'}(I;\dB_{q,1}^{\frac{d}{q}-1+\frac{2}{r'}})}
    \| v^{\varepsilon} \|_{L^{r}(I;\dB_{q,1}^{\frac{d}{q}-1+\frac{2}{r}}) \cap L^{r'}(I;\dB_{q,1}^{\frac{d}{q}-1+\frac{2}{r'}})}\\
    &\qquad
    +
    C 
    \varepsilon
    \| a^{\varepsilon} \|_{L^{\infty}(I;\dB_{q,1}^{\frac{d}{q}})}
    \| u^{\varepsilon} \|_{L^1(I;\dB_{q,1}^{\frac{d}{q}+1})}^{h;\frac{\beta_0}{\varepsilon}}
    +
    C
    \| a^{\varepsilon} \|_{L^2(I;\dB_{q,1}^{\frac{d}{q}})}
    \| u^{\varepsilon} \|_{L^2(I;\dB_{q,1}^{\frac{d}{q}})}
    +
    C
    \| a \|_{L^2(I;\dB_{q,1}^{\frac{d}{q}})}^2\\
    &\quad 
    \leqslant{}
    C
    \widetilde{A_{q,r}^{\varepsilon,\alpha}}[a^{\varepsilon},u^{\varepsilon};w](I)
    A_{q,r}^{\varepsilon,\alpha}[a^{\varepsilon},v^{\varepsilon}](I).
    \end{split}
\end{align}
Hence, combining \eqref{Y-h-1}, \eqref{Y-h-2}, \eqref{Y-h-3}, and \eqref{Y-h-5}, we obtain 
\begin{align}\label{Y-h-6}
    &\varepsilon \| a^{\varepsilon} \|_{{L^{\infty}}(I;\dB_{q,1}^{\frac{d}{q}})}^{h;\frac{\beta_0}{\varepsilon}}
    +
    \frac{1}{\varepsilon}
    \| a^{\varepsilon} \|_{L^1(I;\dB_{q,1}^{\frac{d}{q}})}^{h;\frac{\beta_0}{\varepsilon}}
    +
    \| \mathbb{Q}u^{\varepsilon} \|_{{L^{\infty}}(I;\dB_{q,1}^{\frac{d}{q}-1})\cap L^1(I;\dB_{q,1}^{\frac{d}{2}+1})}^{h;\frac{\beta_0}{\varepsilon}}\\
    &\quad 
    \leqslant{}
    C
    \|(a^{\varepsilon},\mathbb{Q}u^{\varepsilon})(t_0)\|_{D_q^{\varepsilon}}
    +
    C
    \|(w \cdot \nabla)w \|_{L^1(I;\dB_{q,1}^{\frac{d}{q}-1})}^{h;\frac{\beta_0}{\varepsilon}}\\
    &
    \qquad
    +
    C
    \widetilde{A_{q,r}^{\varepsilon,\alpha}}[a^{\varepsilon},u^{\varepsilon};w](I)
    A_{q,r}^{\varepsilon,\alpha}[a^{\varepsilon},v^{\varepsilon}](I).
\end{align}
On the other hand, from Lemma \ref{lemm:lin-Q-2} it follows that 
\begin{align}\label{Y-h-7}
    \begin{split}
    &
    \| (a^{\varepsilon},\mathbb{Q}u^{\varepsilon}) \|_{{L^{\infty}}(I;\dB_{2,1}^{\frac{d}{2}-1})\cap L^1(I;\dB_{2,1}^{\frac{d}{2}+1})}^{m;\alpha,\frac{\beta_0}{\varepsilon}}\\
    &\quad 
    \leqslant{}
    C
    \| (a^{\varepsilon},\mathbb{Q}u^{\varepsilon})(t_0) \|_{\dB_{2,1}^{\frac{d}{2}-1}}^{m;\alpha,\frac{\beta_0}{\varepsilon}}
    +
    C
    \| (\div(a^{\varepsilon}u^{\varepsilon}), \mathcal{N}^{\varepsilon}[a^{\varepsilon},u^{\varepsilon}]) \|_{L^1(I;\dB_{2,1}^{\frac{d}{2}-1})}^{m;\alpha,\frac{\beta_0}{\varepsilon}}\\
    &\quad 
    \leqslant{}
    C
    \| (a^{\varepsilon},\mathbb{Q}u^{\varepsilon})(t_0) \|_{D_q^{\varepsilon}}^{h;\alpha}
    +
    C
    \| (w \cdot \nabla)w  \|_{L^1(I;\dB_{2,1}^{\frac{d}{2}-1})}^{m;\alpha,\frac{\beta_0}{\varepsilon}}\\
    &\qquad
    +
    C
    \| a^{\varepsilon}u^{\varepsilon}\|_{L^1(I;\dB_{2,1}^{\frac{d}{2}})}^{\ell;\frac{\beta_0}{\varepsilon}} 
    +
    C
    \|\mathcal{N}^{\varepsilon}[a^{\varepsilon},u^{\varepsilon}] - (w \cdot \nabla)w  \|_{L^1(I;\dB_{2,1}^{\frac{d}{2}-1})}^{\ell;\frac{\beta_0}{\varepsilon}}.
    \end{split}
\end{align}
Here, by virtue of \eqref{Y-h-4} and 
Lemmas \ref{lemm:nonlin-1}, \ref{lemm:nonlin-3}, \ref{lemm:nonlin-4}, \ref{lemm:nonlin-5} and \ref{lemm:A}, there holds
\begin{align}\label{Y-h-8}
    \begin{split}
    &
    \| a^{\varepsilon}u^{\varepsilon}\|_{L^1(I;\dB_{2,1}^{\frac{d}{2}})}^{\ell;\frac{\beta_0}{\varepsilon}} 
    +  
    \|\mathcal{N}^{\varepsilon}[a^{\varepsilon},u^{\varepsilon}] - (w \cdot \nabla)w  \|_{L^1(I;\dB_{2,1}^{\frac{d}{2}-1})}^{\ell;\frac{\beta_0}{\varepsilon}}\\
    &\quad
    \leqslant{}
    C
    \| (w,u^{\varepsilon}) \|_{L^{r}(I;\dB_{q,1}^{\frac{d}{q}-1+\frac{2}{r}}) \cap L^{r'}(I;\dB_{q,1}^{\frac{d}{q}-1+\frac{2}{r'}})}
    \| v^{\varepsilon} \|_{L^{r}(I;\dB_{q,1}^{\frac{d}{q}-1+\frac{2}{r}}) \cap L^{r'}(I;\dB_{q,1}^{\frac{d}{q}-1+\frac{2}{r'}})}\\
    &\qquad
    +
    C 
    \varepsilon
    \| a^{\varepsilon} \|_{L^{\infty}(I;\dB_{q,1}^{\frac{d}{q}})}
    \| u^{\varepsilon} \|_{L^1(I;\dB_{q,1}^{\frac{d}{q}+1})}^{h;\frac{\beta_0}{\varepsilon}}
    +
    C
    \| a^{\varepsilon} \|_{L^2(I;\dB_{q,1}^{\frac{d}{q}})}
    \| u^{\varepsilon} \|_{L^2(I;\dB_{q,1}^{\frac{d}{q}})}
    +
    C
    \| a \|_{L^2(I;\dB_{q,1}^{\frac{d}{q}})}^2\\ 
    &\qquad
    +
    C
    \left(
    \| a^{\varepsilon} \|_{L^r(I;\dB_{q,1}^{\frac{d}{q}-1+\frac{2}{r}})}
    +
    \| a^{\varepsilon} \|_{L^{r'}(I;\dB_{q,1}^{\frac{d}{q}-1+\frac{2}{r'}})}^{\ell;\frac{4\beta_0}{\varepsilon}}
    \right)
    \| u^{\varepsilon} \|_{L^r(I;\dB_{q,1}^{\frac{d}{q}-1+\frac{2}{r}}) \cap L^{r'}(I;\dB_{q,1}^{\frac{d}{q}-1+\frac{2}{r'}})}\\
    &\quad 
    \leqslant{}
    C
    \widetilde{A_{q,r}^{\varepsilon,\alpha}}[a^{\varepsilon},u^{\varepsilon};w](I)
    A_{q,r}^{\varepsilon,\alpha}[a^{\varepsilon},v^{\varepsilon}](I).
    \end{split}
\end{align}
Hence, by \eqref{Y-h-7} and \eqref{Y-h-8}, we obtain \eqref{est:Y-comp-high}.

Next, we show \eqref{est:Y-comp-low}.
It follows from Lemma \ref{lemm:lin-Q-3} and Lemmas \ref{lemm:nonlin-1}, \ref{lemm:nonlin-3}, \ref{lemm:nonlin-4}, \ref{lemm:nonlin-5} and \ref{lemm:A} that
\begin{align}
    \begin{split}
    &\| (a^{\varepsilon},\mathbb{Q}u^{\varepsilon}) \|_{Y_{q,r}^{\varepsilon,\alpha}(I)}^{\ell;\alpha}\\
    &\quad 
    \leqslant{}
    C
    \varepsilon^{\frac{1}{r}}
    \| (a^{\varepsilon},\mathbb{Q}u^{\varepsilon})(t_0) \|_{\dB_{2,1}^{\frac{d}{2}-1+\frac{1}{r}}}^{\ell;\alpha}
    +
    C
    \varepsilon^{\frac{1}{r}}
    \| (\div(a^{\varepsilon}u^{\varepsilon}),\mathbb{Q}\mathcal{N}^{\varepsilon}[a^{\varepsilon},u^{\varepsilon}]) \|_{L^1(I;\dB_{2,1}^{\frac{d}{2}-1+\frac{1}{r}})}^{\ell;\alpha}\\
    &\quad 
    \leqslant{}
    C
    (\alpha\varepsilon)^{\frac{1}{r}}
    \| (a^{\varepsilon},\mathbb{Q}u^{\varepsilon})(t_0) \|_{\dB_{2,1}^{\frac{d}{2}-1}}^{\ell;\alpha}
    +
    C
    (\alpha\varepsilon)^{\frac{1}{r}}
    \| (\div(a^{\varepsilon}u^{\varepsilon}),\mathcal{N}^{\varepsilon}[a^{\varepsilon},u^{\varepsilon}]) \|_{L^1(I;\dB_{2,1}^{\frac{d}{2}-1})}^{\ell;\alpha}\\
    &\quad\leqslant{}
    C
    (\alpha \varepsilon)^{\frac{1}{r}}
    \| (a^{\varepsilon},\mathbb{Q}u^{\varepsilon})(t_0) \|_{D^{\varepsilon}}^{\ell;\alpha}
    +
    C
    (\alpha \varepsilon)^{\frac{1}{r}}
    \left( A_{q,r}^{\varepsilon,\alpha}[a^{\varepsilon},u^{\varepsilon}](I)\right)^2\\
    &\quad\leqslant{}
    C
    (\alpha \varepsilon)^{\frac{1}{r}}
    \| (a^{\varepsilon},\mathbb{Q}u^{\varepsilon})(t_0) \|_{D^{\varepsilon}}^{\ell;\alpha}
    +
    C
    (\alpha \varepsilon)^{\frac{1}{r}}
    A_{q,r}^{\varepsilon,\alpha}[a^{\varepsilon},u^{\varepsilon}](I)
    \| (a^{\varepsilon},u^{\varepsilon}) \|_{X^{\varepsilon}(I)}
    \end{split}
\end{align}
By \eqref{Y-h-8}, we obtain \eqref{est:Y-comp-low}.

Finally, we prove \eqref{est:Y-incomp}.
Then, we see that $v^{\varepsilon}=u^{\varepsilon}-w$ satisfies
\begin{align}\label{eq:v}
    \partial_t \mathbb{P}v^{\varepsilon}
    - 
    \mu \Delta \mathbb{P}v^{\varepsilon}
    =
    -
    \mathbb{P}(w \cdot \nabla) v^{\varepsilon}
    -
    \mathbb{P}(v^{\varepsilon} \cdot \nabla) u^{\varepsilon}
    -
    \mathbb{P}(\mathcal{J}(\varepsilon a^{\varepsilon}) \mathcal{L}u^{\varepsilon}).
\end{align}
where we have used $\mathbb{P}(\mathcal{K}(\varepsilon a^{\varepsilon}) \nabla a^{\varepsilon}) = 0$.
The key ingredient of the proof of the estimate for the incompressible part is to use the linear estimate in Lemma \ref{lemm:lin-P-2} and the commutator estimate in Lemma \ref{lemm:nonlin-3}.
Using Lemma \ref{lemm:lin-P-2}, we have
\begin{align}
    &
    \| \mathbb{P}v^{\varepsilon} \|_{{L^{\infty}}(I;\dB_{q,1}^{\frac{d}{q}-1}) \cap L^1(I;\dB_{q,1}^{\frac{d}{q}+1})}
    \\
    &\quad
    \leqslant
    C
    \| \mathbb{P}v^{\varepsilon}(t_0) \|_{\dB_{q,1}^{\frac{d}{q}-1}}
    +
    C
    \sum_{j \in \mathbb{Z}}
    2^{(\frac{d}{q}-1)j}
    \| [w \cdot \nabla,  \mathbb{P}\dot{\Delta}_j]v^{\varepsilon} \|_{L^1(I;L^q)}\\
    &\qquad
    +
    C
    \| (v^{\varepsilon} \cdot \nabla) u^{\varepsilon} \|_{L^1(I;\dB_{q,1}^{\frac{d}{q}-1})}
    +
    C
    \| \mathcal{J}(\varepsilon a^{\varepsilon}) \mathcal{L}u^{\varepsilon} \|_{L^1(I;\dB_{q,1}^{\frac{d}{q}-1})}.
\end{align}
From Lemmas \ref{lemm:nonlin-3}, \ref{lemm:nonlin-4} and \ref{lemm:A}, it follows that
\begin{align}
    &
    \| \mathbb{P}v^{\varepsilon} \|_{{L^{\infty}}(I;\dB_{q,1}^{\frac{d}{q}-1}) \cap L^1(I;\dB_{q,1}^{\frac{d}{q}+1})}
    \\
    &\quad
    \leqslant
    C
    \| \mathbb{P}v^{\varepsilon}(t_0) \|_{\dB_{q,1}^{\frac{d}{q}-1}}
    +
    C
    \| (w,u^{\varepsilon}) \|_{L^{r'}(I;\dB_{q,1}^{\frac{d}{q}-1+\frac{2}{r'}})}
    \| v^{\varepsilon} \|_{L^{r}(I;\dB_{q,1}^{\frac{d}{q}-1+\frac{2}{r}})}\\
    &\qquad
    +
    C 
    \varepsilon
    \| a^{\varepsilon} \|_{L^{\infty}(I;\dB_{q,1}^{\frac{d}{q}})}
    \| u^{\varepsilon} \|_{L^1(I;\dB_{q,1}^{\frac{d}{q}+1})}^{h;\frac{\beta_0}{\varepsilon}}
    +
    C
    \| a^{\varepsilon} \|_{L^2(I;\dB_{q,1}^{\frac{d}{q}})}
    \| u^{\varepsilon} \|_{L^2(I;\dB_{q,1}^{\frac{d}{q}})}\\
    &\quad 
    \leqslant
    C
    \| \mathbb{P}v^{\varepsilon}(t_0) \|_{\dB_{q,1}^{\frac{d}{q}-1}}
    +
    C
    \widetilde{A_{q,r}^{\varepsilon,\alpha}}[a^{\varepsilon},u^{\varepsilon};w](I)
    A_{q,r}^{\varepsilon,\alpha}[a^{\varepsilon},v^{\varepsilon}](I).
\end{align}
This completes the proof.
\end{proof}

\begin{lemm}\label{lemm:lim}
Let us consider the same assumptions as in Lemma \ref{lemm:Y} with
\begin{gather}
    \frac{1}{r} < \frac{2d}{q} - 1.
\end{gather}
Then, there exists a positive constant $C_4=C_4(d,q,r,\mu,P)$ such that
\begin{align}
    &
    \| (a^{\varepsilon},\mathbb{Q}u^{\varepsilon}) \|_{{L^r}(I;\dB_{q,1}^{\frac{d}{q}-1+\frac{1}{r}})}
    +
    \| \mathbb{P}u^{\varepsilon}-w \|_{{L^{\infty}}(I;\dB_{q,1}^{\frac{d}{q}-1-\frac{1}{r}}) \cap L^1(I;\dB_{q,1}^{\frac{d}{q}+1-\frac{1}{r}})}\\
    &\quad
    \leqslant
    C_4 
    \varepsilon^{\frac{1}{r}} 
    \| (a^{\varepsilon},\mathbb{Q}u^{\varepsilon})(t_0) \|_{D^{\varepsilon}}
    +
    C_4 \| (\mathbb{P}u^{\varepsilon} - w)(t_0) \|_{\dB_{q,1}^{\frac{d}{q}-1-\frac{1}{r}}}\\
    &\qquad
    +
    C_4\varepsilon^{\frac{1}{r}}
    A_{q,r}^{\varepsilon,\alpha}[a^{\varepsilon},u^{\varepsilon}](I)
    \| (a^{\varepsilon},u^{\varepsilon}) \|_{X^{\varepsilon}(I)}\\
    &\qquad
    +
    C_4
    \widetilde{A_{q,r}^{\varepsilon,\alpha}}[a^{\varepsilon},u^{\varepsilon};w](I)
    \| (a^{\varepsilon},u^{\varepsilon}-w) \|_{L^r(I;\dB_{q,1}^{\frac{d}{q}-1+\frac{1}{r}})}.
\end{align}
Here, $\widetilde{A_{q,r}^{\varepsilon,\alpha}}[a^{\varepsilon},u^{\varepsilon};w](I)$ is defined in \eqref{At}.
\end{lemm}

\begin{proof}
We first consider the compressible part.
We see that
\begin{align}\label{lim-1}
    \begin{split}
    \| (a^{\varepsilon},\mathbb{Q}u^{\varepsilon}) \|_{{L^r}(I;\dB_{q,1}^{\frac{d}{q}-1+\frac{1}{r}})}
    \leqslant{}
    &
    \| (a^{\varepsilon},\mathbb{Q}u^{\varepsilon}) \|_{{L^r}(I;\dB_{q,1}^{\frac{d}{q}-1+\frac{1}{r}})}^{\ell;\frac{\beta_0}{\varepsilon}}
    +
    C
    \varepsilon^{\frac{1}{r}}
    \| (a^{\varepsilon},\mathbb{Q}u^{\varepsilon}) \|_{{L^r}(I;\dB_{q,1}^{\frac{d}{q}-1+\frac{2}{r}})}^{h;\frac{\beta_0}{\varepsilon}}\\
    \leqslant{}
    &
    \| (a^{\varepsilon},\mathbb{Q}u^{\varepsilon}) \|_{{L^r}(I;\dB_{q,1}^{\frac{d}{q}-1+\frac{1}{r}})}^{\ell;\frac{\beta_0}{\varepsilon}}
    +
    C
    \varepsilon^{\frac{1}{r}}
    \| (a^{\varepsilon},\mathbb{Q}u^{\varepsilon}) \|_{Y_{q,r}^{\varepsilon,\alpha}(I)}^{h;\frac{\beta_0}{\varepsilon}}.
    \end{split}
\end{align}
For the first term of the right hand side of \eqref{lim-1}, it holds by Lemma \ref{lemm:lin-Q-3} and Lemmas \ref{lemm:nonlin-1}, \ref{lemm:nonlin-3}, \ref{lemm:nonlin-4}, \ref{lemm:nonlin-5} and \ref{lemm:A} that 
\begin{align}\label{lim-2}
    \begin{split}
    &
    \| (a^{\varepsilon},\mathbb{Q}u^{\varepsilon}) \|_{{L^r}(I;\dB_{q,1}^{\frac{d}{q}-1+\frac{1}{r}})}^{\ell;\frac{\beta_0}{\varepsilon}}\\ 
    &\quad 
    \leqslant{}
    C
    \varepsilon^{\frac{1}{r}}
    \| (a^{\varepsilon},\mathbb{Q}u^{\varepsilon})(t_0) \|_{\dB_{2,1}^{\frac{d}{2}-1}}^{\ell;\frac{\beta_0}{\varepsilon}}
    +
    C
    \varepsilon^{\frac{1}{r}}
    \| ( \div(a^{\varepsilon}u^{\varepsilon}),\mathcal{N}^{\varepsilon}[a^{\varepsilon},u^{\varepsilon}] ) \|_{L^1( I ; \dB_{2,1}^{\frac{d}{2}-1})}^{\ell;\frac{\beta_0}{\varepsilon}}\\
    &\quad 
    \leqslant 
    C\varepsilon^{\frac{1}{r}}\| (a^{\varepsilon},\mathbb{Q}u^{\varepsilon})(t_0) \|_{D^{\varepsilon}}^{\ell;\frac{\beta_0}{\varepsilon}}
    +
    C\varepsilon^{\frac{1}{r}}
    A_{q,r}^{\varepsilon,\alpha}[a^{\varepsilon},u^{\varepsilon}](I)
    \| (a^{\varepsilon},u^{\varepsilon}) \|_{X^{\varepsilon}(I)}.
    \end{split}
\end{align}
For the second term of the right hand side of \eqref{lim-1}, it follows from Lemma \ref{lemm:lin-Q-1} and Lemmas \ref{lemm:nonlin-2}, \ref{lemm:nonlin-1}, \ref{lemm:nonlin-3}, \ref{lemm:nonlin-4}, \ref{lemm:nonlin-5} and \ref{lemm:A} that
\begin{align}\label{lim-3}
    \begin{split}
    &
    \| (a^{\varepsilon},\mathbb{Q}u^{\varepsilon}) \|_{Y_{q,r}^{\varepsilon,\alpha}(I)}^{h;\frac{\beta_0}{\varepsilon}}\\
    &
    \quad 
    \leqslant{}
    C
    \bigg(\varepsilon\|a^{\varepsilon}(t_0)\|_{\dB_{q,1}^{\frac{d}{q}}}^{h;\frac{\beta_0}{\varepsilon}}
    +
    \| \mathbb{Q}u^{\varepsilon}(t_0) \|_{\dB_{q,1}^{\frac{d}{q}-1}}^{h;\frac{\beta_0}{\varepsilon}}
    \bigg)\\
    &\qquad
    + 
    C\|a^{\varepsilon}u^{\varepsilon}\|_{L^1(I;\dB_{q,1}^{\frac{d}{q}})}^{h;\frac{\beta_0}{\varepsilon}}
    + 
    C\|\mathcal{N}^{\varepsilon}[a^{\varepsilon},u^{\varepsilon}]\|_{L^1(I;\dB_{q,1}^{\frac{d}{q}-1})}^{h;\frac{\beta_0}{\varepsilon}}
    +
    C\varepsilon 
    \| a^{\varepsilon} \div \mathbb{Q}u^{\varepsilon} \|_{L^1(I;\dB_{q,1}^{\frac{d}{q}})}^{h;\frac{\beta_0}{\varepsilon}}\\
    &\qquad
    +
    C\varepsilon\sum_{2^j \geqslant \frac{\beta_0}{\varepsilon}}2^{\frac{d}{q}j}
    \| [u^{\varepsilon}\cdot \nabla, \dot{\Delta}_j] a^{\varepsilon}\|_{L^1(I;L^q)}
    +
    C\varepsilon
    \left\|
        \| \div u^{\varepsilon} \|_{L^{\infty}}
        \| a^{\varepsilon} \|_{\dB_{q,1}^{\frac{d}{q}}}
    \right\|_{L^1(0,t)}\\
    & 
    \quad
    \leqslant{}
    C\| (a^{\varepsilon},\mathbb{Q}u^{\varepsilon})(t_0) \|_{D^{\varepsilon}}^{h;\frac{\beta_0}{\varepsilon}} 
    +
    A_{q,r}^{\varepsilon,\alpha}[a^{\varepsilon},u^{\varepsilon}](I)
    \| (a^{\varepsilon},u^{\varepsilon}) \|_{X^{\varepsilon}(I)}.
    \end{split}
\end{align}
Combining \eqref{lim-1}, \eqref{lim-2} and \eqref{lim-3},
we complete the estimate for the compressible part.

Next, we consider the incompressible part. 
The key of the proof of the estimate for the incompressible part is to use the linear estimate in Lemma \ref{lemm:lin-P-2} and the commutator estimate in Lemma \ref{lemm:nonlin-3}.
Let $v^{\varepsilon}:=u^{\varepsilon}-w$.
Then, applying Lemma \ref{lemm:lin-P-2} to the equation \eqref{eq:v}, we have
\begin{align}
    &
    \| \mathbb{P}v^{\varepsilon} \|_{{L^{\infty}}(I;\dB_{q,1}^{\frac{d}{q}-1-\frac{1}{r}}) \cap L^1(I;\dB_{q,1}^{\frac{d}{q}+1-\frac{1}{r}})}
    \\
    &\quad
    \leqslant
    C
    \| \mathbb{P}v^{\varepsilon}(t_0) \|_{\dB_{q,1}^{\frac{d}{q}-1-\frac{1}{r}}}
    +
    C
    \sum_{j \in \mathbb{Z}}
    2^{(\frac{d}{q}-1-\frac{1}{r})j}
    \| [w \cdot \nabla, \dot{\Delta}_j \mathbb{P}]v^{\varepsilon} \|_{L^1(I;L^q)}\\
    &\qquad
    +
    C
    \| (v^{\varepsilon} \cdot \nabla) u^{\varepsilon} \|_{L^1(I;\dB_{q,1}^{\frac{d}{q}-1-\frac{1}{r}})}
    +
    C
    \| \mathcal{J}(\varepsilon a^{\varepsilon}) \mathcal{L}u^{\varepsilon} \|_{L^1(I;\dB_{q,1}^{\frac{d}{q}-1-\frac{1}{r}})}.
\end{align}
We see by Lemmas \ref{lemm:nonlin-3}, \ref{lemm:nonlin-4} and \ref{lemm:A} that
\begin{align}
    &
    \| \mathbb{P}v^{\varepsilon} \|_{{L^{\infty}}(I;\dB_{q,1}^{\frac{d}{q}-1-\frac{1}{r}}) \cap L^1(I;\dB_{q,1}^{\frac{d}{q}+1-\frac{1}{r}})}\\
    &\quad
    \leqslant
    C
    \| \mathbb{P}v^{\varepsilon}(t_0) \|_{\dB_{q,1}^{\frac{d}{q}-1-\frac{1}{r}}}
    +
    C
    \| (w,u^{\varepsilon}) \|_{L^{r'}(I;\dB_{q,1}^{\frac{d}{q}-1+\frac{2}{r'}})}
    \| (a^{\varepsilon},v^{\varepsilon}) \|_{L^r(I;\dB_{q,1}^{\frac{d}{q}-1+\frac{1}{r}})}\\
    &\qquad
    +
    C
    \varepsilon^{\frac{1}{r}+1}
    \| a^{\varepsilon} \|_{L^{\infty}(I;\dB_{q,1}^{\frac{d}{q}})}
    \| u^{\varepsilon} \|_{L^1(I;\dB_{q,1}^{\frac{d}{q}+1})}^{h;\frac{\beta_0}{\varepsilon}}
    +
    C
    \varepsilon^{\frac{1}{r}}
    \| (a^{\varepsilon},u^{\varepsilon}) \|_{L^2(I;\dB_{q,1}^{\frac{d}{q}})}^2
    \\
    &\quad 
    \leqslant
    C\| \mathbb{P}v^{\varepsilon}(t_0) \|_{\dB_{q,1}^{\frac{d}{q}-1-\frac{1}{r}}}
    +
    C
    \varepsilon^{\frac{1}{r}}
    A_{q,r}^{\varepsilon,\alpha}[a^{\varepsilon},u^{\varepsilon}](I)
    \| (a^{\varepsilon},u^{\varepsilon}) \|_{X^{\varepsilon}(I)}
    \\
    &\qquad 
    +
    C
    \bigg(
    \| w \|_{L^{r'}(I;\dB_{q,1}^{\frac{d}{q}-1+\frac{2}{r'}})}
    +
    A_{q,r}^{\varepsilon,\alpha}[a^{\varepsilon},u^{\varepsilon}](I)
    \bigg)
    \| (a^{\varepsilon},v^{\varepsilon}) \|_{L^r(I;\dB_{q,1}^{\frac{d}{q}-1+\frac{1}{r}})}.
\end{align}
Thus, we complete the proof.
\end{proof}

\section{Proof of the main result}\label{sec:pf}
Now, we are in a position to present the proof of our main result.

\begin{proof}[Proof of Theorem \ref{thm:2}]
Let $q$ and $r$ satisfy \eqref{main:q} and \eqref{main:r}.
We note that $q$ and $r$ satisfy all assumptions of Lemmas \ref{lemm:X}, \ref{lemm:Y} and \ref{lemm:lim}.
Let $C_0:= \max\{1,C_1,C_2,C_3,C_4,C_5\}$,
where $C_1,...,C_4$ are the positive constants appearing in the previous section and the positive constant $C_5=C_5(d,q)$ satisfies $\| f \|_{L^{\infty}}\leqslant C_5 \| f \|_{\dB_{q,1}^{\frac{d}{q}}}$ for all $f \in \dB_{q,1}^{\frac{d}{q}}(\mathbb{R}^d)$.

Let $w$ be the solution to \eqref{main:eq-incomp} with the initial data $\mathbb{P}u_0$ in the class \eqref{main:incomp-reg}.
Let 
\begin{align}
        \delta_0
        :=
        \frac{1}{2}
        \min \left\{ \frac{1}{24C_0}, \frac{1}{C_0^2}, \| w \|_{L^{r}(0,\infty;\dB_{q,1}^{\frac{d}{q}-1+\frac{2}{r}}) \cap L^1(0,\infty;\dB_{q,1}^{\frac{d}{q}+1})} \right\}.
\end{align}
Then, by Lemma \ref{lemm:time}
there exists a positive integer $N_0=N_0(\mathbb{P}u_0,q,r)$ and a time sequence $\{ T_n=T_n(\mathbb{P}u_0,q,r) \}_{n=0}^{N_0}$ such that 
\begin{gather}
    0 = T_0 < T_1 < ... < T_{N_0-1} < T_{N_0} = \infty,\\
    \| w \|_{L^{r}(T_{n-1},T_n;\dB_{q,1}^{\frac{d}{q}-1+\frac{2}{r}}) \cap L^1(T_{n-1},T_n;\dB_{q,1}^{\frac{d}{q}+1})} \leqslant \delta_0
    \quad (n=1,...,N_0).\label{w-small-1}
\end{gather}
Here, we see by Lemma \ref{lemm:nonlin-3} that
\begin{align}\label{w}
    \| (w \cdot \nabla) w \|_{L^1(0,\infty;\dB_{2,1}^{\frac{d}{2}-1})}
    \leqslant
    C
    \| w \|_{L^r(0,\infty;\dB_{q,1}^{\frac{d}{q}-1+\frac{2}{r}})}
    \| w \|_{L^{r'}(0,\infty;\dB_{q,1}^{\frac{d}{q}-1+\frac{2}{r'}})}
    <
    \infty.
\end{align}
Since it holds by $(a_0,u_0) \in (\dB_{2,1}^{\frac{d}{2}-1}(\mathbb{R}^d) \cap \dB_{2,1}^{\frac{d}{2}}(\mathbb{R}^d)) \times \dB_{2,1}^{\frac{d}{2}-1}(\mathbb{R}^d)^d$ and \eqref{w} that
\begin{align}
    &
    \lim_{\alpha \to \infty}
    \left(
    \| a_0 \|_{\dB_{2,1}^{\frac{d}{2}}}^{h;\alpha}
    +
    \| (a_0, \mathbb{Q}u_0) \|_{\dB_{2,1}^{\frac{d}{2}-1}}^{h;\alpha}
    \right)
    =0,\\
    &
    \lim_{\alpha \to \infty}
    \| (w \cdot \nabla) w \|_{L^1(0,\infty;\dB_{2,1}^{\frac{d}{2}-1})}^{h;\alpha}
    =
    0,
\end{align}
we may choose a constant $\alpha_{0}=\alpha_{0}(d,q,r,a_0,u_0)>0$ satisfying 
\begin{align}
    \| a_0 \|_{\dB_{2,1}^{\frac{d}{2}}}^{h;\alpha_{0}}
    +
    \| (a_0, \mathbb{Q}u_0) \|_{\dB_{2,1}^{\frac{d}{2}-1}}^{h;\alpha_{0}}
    \leqslant
    \frac{\delta_0}{(4C_0)^{N_0}},\quad 
    \| (w \cdot \nabla) w \|_{L^1(0,\infty;\dB_{2,1}^{\frac{d}{2}-1})}^{h;\alpha_{0}}
    \leqslant
    \frac{\delta_0}{(4C_0)^{N_0}}.
\end{align}
Then, for such $\alpha_{0}$,
we take a constant
$0 < \varepsilon_0=\varepsilon_0(d,q,r,a_0,u_0) <  \min \{ 1/\alpha_{0}, \beta_0/\alpha_{0} \}$ 
so that
\begin{align}
    (\alpha_{0} \varepsilon)^{\frac{1}{r(r-1)}}
    \| (a_0,\mathbb{Q}u_0) \|_{D^{\varepsilon}}
    \leqslant
    \frac{1}{24}
    \cdot
    \frac{\delta_0}{(4C_0)^{N_0}}
\end{align}
for all $ 0 < \varepsilon \leqslant \varepsilon_0$.

Fix a Mach number $0 < \varepsilon  \leqslant \varepsilon_0$ and let $(a^{\varepsilon},u^{\varepsilon})$ be the local solution to \eqref{eq:re_comp-1} on $[0,T_{\rm max}^{\varepsilon})$ constructed in Lemma \ref{lemm:LWP},
where $T_{\rm max}^{\varepsilon}$ denotes the maximal existence time.
Our aim is to prove $T_{\rm max}^{\varepsilon}=\infty$.
To begin with, we prove $T_{\rm max}^{\varepsilon} \geqslant T_1$ by the continuation argument.
We introduce a time 
\begin{gather}
    T_1^*
    :=
    \sup
    \bigg\{
    T \in (0,T_{\rm max}^{\varepsilon})\ ;\ 
    {\rm
    \eqref{T-1},\ 
    \eqref{T-2},\ 
    \eqref{T-3}\ and\ 
    \eqref{T-4}\ hold\ at\ 
    }
    T.
    \bigg\},
\end{gather}
where
\begin{align}
    &
    \begin{aligned}\label{T-1}
    \| (a^{\varepsilon},u^{\varepsilon}) \|_{X^{\varepsilon}(0,T)} 
    \leqslant{} 
    4C_0 \| (a_0,u_0) \|_{D^{\varepsilon}},\quad
    \end{aligned}\\ 
    &\begin{aligned}\label{T-2}
    A_{q,r}^{\varepsilon,\alpha_{0}}[a^{\varepsilon},u^{\varepsilon}-w](0,T)
    \leqslant{}
    \frac{\delta_0}{(4C_0)^{N_0-1}},
    \end{aligned}\\
    &
    \begin{aligned}\label{T-3}
    \| \mathbb{P}u^{\varepsilon} - w \|_{
    {L^{\infty}}(0,T;\dB_{q,1}^{\frac{d}{q}-1})
    \cap
    L^1(0,T;\dB_{q,1}^{\frac{d}{q}+1})}
    \leqslant
    \frac{\delta_0}{(4C_0)^{N_0-1}},
    \end{aligned}\\
    &
    \begin{aligned}\label{T-4}
    &
    \|(a^{\varepsilon},\mathbb{Q}u^{\varepsilon})\|_{{L^r}(0,T;\dB_{q,1}^{\frac{d}{q}-1+\frac{1}{r}})}
    +
    \|\mathbb{P}u^{\varepsilon} - w \|_{{L^{\infty}}(0,T;\dB_{q,1}^{\frac{d}{q}-1-\frac{1}{r}}) \cap L^1(0,T;\dB_{q,1}^{\frac{d}{q}+1-\frac{1}{r}})}\\ 
    &\quad
    \leqslant 
    4C_0 \varepsilon^{\frac{1}{r}} \| (a_0, u_0) \|_{D^{\varepsilon}}.
    \end{aligned}
\end{align}
It is easy to see that $T_1^*>0$.
Suppose by contradiction that $T_1^* < T_1$.
Let $0 < T <T_1^*$.
Here, we note that \eqref{w-small-1} and \eqref{T-2} yield
\begin{align}
    A_{q,r}^{\varepsilon,\alpha_{0}}[a^{\varepsilon},u^{\varepsilon}](0,T)
    \leqslant{}&
    \| w \|_{L^{r}(0,T;\dB_{q,1}^{\frac{d}{q}-1+\frac{2}{r}}) \cap L^1(0,T;\dB_{q,1}^{\frac{d}{q}+1})}
    +
    A_{q,r}^{\varepsilon,\alpha_{0}}[a^{\varepsilon},u^{\varepsilon}-w](0,T)\\
    \leqslant{}&
    \delta_0 + \frac{\delta_0}{(4C_0)^{N_0-1}}\\
    \leqslant{}&
    2\delta_0,\\
    \widetilde{A_{q,r}^{\varepsilon,\alpha_{0}}}[a^{\varepsilon},u^{\varepsilon};w](0,T)
    \leqslant{}&
    \| w \|_{L^{r}(0,T;\dB_{q,1}^{\frac{d}{q}-1+\frac{2}{r}}) \cap L^1(0,T;\dB_{q,1}^{\frac{d}{q}+1})}
    +
    A_{q,r}^{\varepsilon,\alpha_{0}}[a^{\varepsilon},u^{\varepsilon}](0,T)\\
    \leqslant{}&
    3\delta_0.
\end{align}
By Lemma \ref{lemm:A}, we see that  
\begin{align}\label{T_1:ptw}
    &\varepsilon\|a^{\varepsilon}\|_{L^{\infty}(0,T;\dB_{q,1}^{\frac{d}{q}})} 
    \leqslant
    C_0A_{q,r}^{\varepsilon,\alpha_{0}}[a^{\varepsilon},\mathbb{Q}u^{\varepsilon}](0,T)
    \leqslant
    C_0
    \cdot
    \frac{\delta_0}{(4C_0)^{N_0-1}}
    \leqslant
    \frac{1}{2C_0},\\
    &
    \varepsilon\|a^{\varepsilon}\|_{L^{\infty}(0,T;L^{\infty})}
    \leqslant
    C_0 
    \varepsilon
    \|a^{\varepsilon}\|_{L^{\infty}(0,T;\dB_{q,1}^{\frac{d}{q}})} 
    \leqslant
    \frac{1}{2}.
\end{align}
Thus, we see that $\rho^{\varepsilon}(t,x) = 1 +\varepsilon a^{\varepsilon}(t,x) >0$ on $[0,T] \times \mathbb{R}^d$.
Lemma \ref{lemm:X} yields 
\begin{align}\label{cont-1}
    \begin{split}
    \| (a^{\varepsilon},u^{\varepsilon}) \|_{X^{\varepsilon}(0,T)}
    \leqslant{}&
    C_0
    \| (a_0,u_0) \|_{D^{\varepsilon}}
    +
    C_0
    \cdot 
    2\delta_0
    \cdot 
    4C_0
    \| (a_0, u_0)\|_{D^{\varepsilon}}\\
    \leqslant{}&
    2C_0
    \| (a_0,u_0) \|_{D^{\varepsilon}}.
    \end{split}
\end{align}
From Lemma \ref{lemm:Y}, it follows that 
\begin{align}
    &
    \| (a^{\varepsilon},\mathbb{Q}u^{\varepsilon}) \|_{Y_{q,r}^{\varepsilon,\alpha_{0}}(0,T)}^{h;\alpha_{0}}
    \leqslant{}
    C_0
    \frac{2\delta_0}{(4C_0)^{N_0}}
    +
    C_0
    \cdot 3\delta_0 \cdot 
    \frac{\delta_0}{(4C_0)^{N_0-1}}
    \leqslant{}
    \frac{5}{8}
    \cdot
    \frac{\delta_0}{(4C_0)^{N_0-1}},\\
    &
    \begin{aligned}
    \| (a^{\varepsilon},\mathbb{Q}u^{\varepsilon}) \|_{Y_{q,r}^{\varepsilon,\alpha_{0}}(0,T)}^{\ell;\alpha_{0}}
    \leqslant{}
    &
    C_0
    (\alpha_{0} \varepsilon)^{\frac{1}{r}}
    \| (a_0,u_0) \|_{D^{\varepsilon}}
    +
    C_0
    (\alpha_{0} \varepsilon)^{\frac{1}{r}}
    \cdot 
    3\delta_0
    \cdot 
    4C_0\| (a_0,u_0) \|_{D^{\varepsilon}}
    \\
    \leqslant{}
    &
    2C_0(\alpha_{0} \varepsilon)^{\frac{1}{r}}\| (a_0,u_0) \|_{D^{\varepsilon}},
    \end{aligned}
    \\
    &
    \begin{aligned}\label{cont-4}
    \| \mathbb{P}u^{\varepsilon} - w \|_{
    {L^{\infty}}(0,T;\dB_{q,1}^{\frac{d}{q}-1})
    \cap
    L^1(0,T;\dB_{q,1}^{\frac{d}{q}+1})}
    \leqslant{}
    C_0
    \cdot 
    3\delta_0
    \cdot
    \frac{\delta_0}{(4C_0)^{N_0-1}}
    \leqslant
    \frac{1}{8}
    \cdot
    \frac{\delta_0}{(4C_0)^{N_0-1}},
    \end{aligned}
\end{align}
which implies that
\begin{align}
    \| (a^{\varepsilon},u^{\varepsilon}-w) \|_{Y_{q,r}^{\varepsilon,\alpha_{0}}(0,T)}
    &
    \leqslant{}
    \| (a^{\varepsilon},\mathbb{Q}u^{\varepsilon}) \|_{Y_{q,r}^{\varepsilon,\alpha_{0}}(0,T)}^{h;\alpha_{0}}
    +
    \| (a^{\varepsilon},\mathbb{Q}u^{\varepsilon}) \|_{Y_{q,r}^{\varepsilon,\alpha_{0}}(0,T)}^{\ell;\alpha_{0}}\\
    &\quad
    +
    \| \mathbb{P}u^{\varepsilon}-w \|_{
    {L^{\infty}}(0,T;\dB_{q,1}^{\frac{d}{q}-1})
    \cap
    L^1(0,T;\dB_{q,1}^{\frac{d}{q}+1})}\\
    &\leqslant{}
    \frac{3}{4}\cdot
    \frac{\delta_0}{(4C_0)^{N_0-1}}
    +
    2C_0(\alpha_{0} \varepsilon)^{\frac{1}{r}}\| (a_0,u_0) \|_{D^{\varepsilon}}
\end{align}
Therefore, we have
\begin{align}\label{cont-2}
    \begin{split}
    A_{q,r}^{\varepsilon,\alpha_{0}}[a^{\varepsilon},u^{\varepsilon}-w](0,T)
    \leqslant{}
    &
    \alpha_{0} \varepsilon
    \| (a^{\varepsilon},u^{\varepsilon})\|_{X^{\varepsilon}(I)}
    +
    \| (a^{\varepsilon},u^{\varepsilon}-w)\|_{Y_{q,r}^{\varepsilon,\alpha}(I)}\\
    &
    +
    \left(
    \| (a^{\varepsilon},\mathbb{Q}u^{\varepsilon})\|_{Y_{q,r}^{\varepsilon,\alpha}(I)}^{\ell;\alpha}
    \right)^{\frac{1}{r-1}}
    \| (a^{\varepsilon},u^{\varepsilon})\|_{X^{\varepsilon}(I)}^{\frac{r-2}{r-1}}\\
    \leqslant{}
    &
    2C_0\alpha_{0} \varepsilon
    \| (a_0,u_0) \|_{D^{\varepsilon}}\\
    &
    +
    \frac{3}{4}\cdot
    \frac{\delta_0}{(4C_0)^{N_0-1}}
    +
    2C_0(\alpha_{0} \varepsilon)^{\frac{1}{r}}\| (a_0,u_0) \|_{D^{\varepsilon}}\\
    &
    +
    2
    C_0
    (\alpha_{0} \varepsilon)^{\frac{1}{r(r-1)}}
    \| (a_0,u_0) \|_{D^{\varepsilon}}\\
    \leqslant{}&
    \frac{3}{4}\cdot
    \frac{\delta_0}{(4C_0)^{N_0-1}}
    +
    6C_0(\alpha_{0} \varepsilon)^{\frac{1}{r(r-1)}}\| (a_0,u_0) \|_{D^{\varepsilon}}\\
    \leqslant{}&
    \frac{7}{8}\cdot
    \frac{\delta_0}{(4C_0)^{N_0-1}}
    \end{split}
\end{align}
By Lemma \ref{lemm:lim}, we have 
\begin{align}
    \begin{split}
    &
    \| (a^{\varepsilon},\mathbb{Q}u^{\varepsilon}) \|_{{L^r}(0,T;\dB_{q,1}^{\frac{d}{q}-1+\frac{1}{r}})}
    +
    \| \mathbb{P}u^{\varepsilon}-w \|_{{L^{\infty}}(0,T;\dB_{q,1}^{\frac{d}{q}-1-\frac{1}{r}}) \cap L^1(
    0,T;\dB_{q,1}^{\frac{d}{q}+1-\frac{1}{r}})}\\
    &\quad
    \leqslant
    C_0 \varepsilon^{\frac{1}{r}}
    \| (a_0, \mathbb{Q}u_0) \|_{D^{\varepsilon}}
    +
    C_0
    \varepsilon^{\frac{1}{r}}
    \cdot 
    2\delta_0
    \cdot 
    4C_0
    \| (a_0,u_0) \|_{D^{\varepsilon}}\\
    &\qquad
    +
    C_0
    \cdot 
    3\delta_0
    \cdot 
    \| (a^{\varepsilon},u^{\varepsilon}-w) \|_{L^r(0,T;\dB_{q,1}^{\frac{d}{q}-1+\frac{1}{r}})}\\
    &\quad
    \leqslant
    \frac{4}{3}C_0 \varepsilon^{\frac{1}{r}}
    \| (a_0, u_0) \|_{D^{\varepsilon}}\\
    &\qquad
    +
    \frac{1}{4}
    \left(
    \| (a^{\varepsilon},\mathbb{Q}u^{\varepsilon}) \|_{{L^r}(0,T;\dB_{q,1}^{\frac{d}{q}-1+\frac{1}{r}})}
    +
    \| \mathbb{P}u^{\varepsilon}-w \|_{{L^{\infty}}(0,T;\dB_{q,1}^{\frac{d}{q}-1-\frac{1}{r}}) \cap L^1(
    0,T;\dB_{q,1}^{\frac{d}{q}+1-\frac{1}{r}})}
    \right),
    \end{split}
\end{align}
which gives 
\begin{align}\label{cont-3}
    \begin{split}
    &
    \| (a^{\varepsilon},\mathbb{Q}u^{\varepsilon}) \|_{{L^r}(0,T;\dB_{q,1}^{\frac{d}{q}-1+\frac{1}{r}})}
    +
    \| \mathbb{P}u^{\varepsilon}-w \|_{{L^{\infty}}(0,T;\dB_{q,1}^{\frac{d}{q}-1-\frac{1}{r}}) \cap L^1(
    0,T;\dB_{q,1}^{\frac{d}{q}+1-\frac{1}{r}})}\\
    &\quad
    \leqslant
    \frac{16}{9}
    C_0 \varepsilon^{\frac{1}{r}}
    \| (a_0, u_0) \|_{D^{\varepsilon}}.
    \end{split}
\end{align}
Hence, \eqref{cont-1}, \eqref{cont-4}, \eqref{cont-2} and \eqref{cont-3} contradict the definition of $T_1^*$ and we have $T_{\rm \max}^{\varepsilon} \geqslant T_1^*\geqslant T_1$.
Next, we show $T_{\rm max}^{\varepsilon} \geqslant T_2$ by the continuation argument.
Let
\begin{gather}
    T_2^*
    :=
    \sup
    \bigg\{
    T \in [T_1,T_{\rm max}^{\varepsilon})\ ;\ 
    {\rm
    \eqref{T-1-2},\ 
    \eqref{T-2-2},\ 
    \eqref{T-3-2}\ and\ 
    \eqref{T-4-2}\ hold\ at\ 
    }
    T.
    \bigg\},
\end{gather}
where
\begin{align}
    &
    \begin{aligned}\label{T-1-2}
    \| (a^{\varepsilon},u^{\varepsilon}) \|_{X^{\varepsilon}(T_1,T)} 
    \leqslant{} 
    (4C_0)^2 \| (a_0,u_0) \|_{D^{\varepsilon}},\quad
    \end{aligned}\\ 
    &\begin{aligned}\label{T-2-2}
    A_{q,r}^{\varepsilon,\alpha_{0}}[a^{\varepsilon},u^{\varepsilon}-w](T_1,T)
    \leqslant{}
    \frac{\delta_0}{(4C_0)^{N_0-2}},
    \end{aligned}\\
    &
    \begin{aligned}\label{T-3-2}
    \| \mathbb{P}u^{\varepsilon} - w \|_{
    {L^{\infty}}(T_1,T;\dB_{q,1}^{\frac{d}{q}-1})
    \cap
    L^1(T_1,T;\dB_{q,1}^{\frac{d}{q}+1})}
    \leqslant
    \frac{\delta_0}{(4C_0)^{N_0-2}},
    \end{aligned}\\
    &
    \begin{aligned}\label{T-4-2}
    &
    \|(a^{\varepsilon},\mathbb{Q}u^{\varepsilon})\|_{{L^r}(T_1,T;\dB_{q,1}^{\frac{d}{q}-1+\frac{1}{r}})}
    +
    \|\mathbb{P}u^{\varepsilon} - w \|_{{L^{\infty}}(T_1,T;\dB_{q,1}^{\frac{d}{q}-1-\frac{1}{r}}) \cap L^1(T_1,T;\dB_{q,1}^{\frac{d}{q}+1-\frac{1}{r}})}\\ 
    &\quad
    \leqslant (4C_0)^2 \varepsilon^{\frac{1}{r}} \| (a_0,u_0) \|_{D^{\varepsilon}}.
    \end{aligned}
\end{align}
Suppose by contradiction that $T_2^* < T_2$.
Let $T_1 < T <T_2^*$.
Here, we note that \eqref{w-small-1} and \eqref{T-2-2} yield
\begin{align}
    A_{q,r}^{\varepsilon,\alpha_{0}}[a^{\varepsilon},u^{\varepsilon}](T_1,T)
    \leqslant{}
    2\delta_0,\qquad
    \widetilde{A_{q,r}^{\varepsilon,\alpha_{0}}}[a^{\varepsilon},u^{\varepsilon};w](T_1,T)
    \leqslant{}
    3\delta_0
\end{align}
and we see that
\begin{align}
    \varepsilon\|a^{\varepsilon}\|_{L^{\infty}(T_1,T;\dB_{q,1}^{\frac{d}{q}})} 
    \leqslant{}
    \frac{1}{2C_0},\qquad
    \varepsilon\|a^{\varepsilon}\|_{L^{\infty}(T_1,T;L^{\infty})}
    \leqslant{}
    \frac{1}{2},
\end{align}
which yields $\rho^{\varepsilon}(t,x) = 1 +\varepsilon a^{\varepsilon}(t,x) >0$ on $[T_1,T] \times \mathbb{R}^d$.
Since $T_1 \leqslant T_1^*$, we see that 
\begin{align}
    &\| (a^{\varepsilon},u^{\varepsilon})(T_1) \|_{D^{\varepsilon}}
    \leqslant
    \| (a^{\varepsilon},u^{\varepsilon}) \|_{X^{\varepsilon}(0,T_1)}
    \leqslant
    4C_0
    \| (a_0,u_0) \|_{D^{\varepsilon}},\\
    &
    \| (a^{\varepsilon},\mathbb{Q}u^{\varepsilon})(T_1) \|_{D_q^{\varepsilon}}^{h;\alpha_{0}}
    \leqslant
    \| (a^{\varepsilon},\mathbb{Q}u^{\varepsilon}) \|_{Y_{q,r}^{\varepsilon,\alpha_{0}}(0,T_1)}^{h;\alpha_{0}}
    \leqslant
    \frac{\delta_0}{(4C_0)^{N_0-1}},\\
    &\| (\mathbb{P}u^{\varepsilon} - w)(T_1) \|_{\dB_{q,1}^{\frac{d}{q}-1}}
    \leqslant
    \| \mathbb{P}u^{\varepsilon} - w \|_{
    {L^{\infty}}(0,T_1;\dB_{q,1}^{\frac{d}{q}-1})}
    \leqslant
    \frac{\delta_0}{(4C_0)^{N_0-1}},\\
    &\| (\mathbb{P}u^{\varepsilon} - w)(T_1) \|_{\dB_{q,1}^{\frac{d}{q}-1-\frac{1}{r}}}
    \leqslant
    \| \mathbb{P}u^{\varepsilon} - w \|_{
    {L^{\infty}}(0,T_1;\dB_{q,1}^{\frac{d}{q}-1-\frac{1}{r}})
    }
    \leqslant
    4C_0 \varepsilon^{\frac{1}{r}} \| (a_0, u_0) \|_{D^{\varepsilon}}.
\end{align}
From Lemma \ref{lemm:X} and \eqref{T-1-2}, we have 
\begin{align}\label{cont-1-2}
    \begin{split}
    \| (a^{\varepsilon},u^{\varepsilon}) \|_{X^{\varepsilon}(T_1,T)}
    \leqslant{}&
    4C_0^2
    \| (a_0,u_0) \|_{D^{\varepsilon}}
    +
    C_0
    \cdot 
    2\delta_0
    \cdot 
    (4C_0)^2
    \| (a_0,u_0) \|_{D^{\varepsilon}}\\
    \leqslant{}&
    \frac{1}{2}(4C_0)^2
    \| (a_0,u_0) \|_{D^{\varepsilon}}.
    \end{split}
\end{align}
By Lemma \ref{lemm:Y}, it holds
\begin{align}
    &
    \| (a^{\varepsilon},\mathbb{Q}u^{\varepsilon}) \|_{Y_{q,r}^{\varepsilon,\alpha_{0}}(T_1,T)}^{h;\alpha_{0}}
    \leqslant{}
    C_0
    \frac{2\delta_0}{(4C_0)^{N_0-1}}
    +
    C_0
    \cdot 3\delta_0 \cdot 
    \frac{\delta_0}{(4C_0)^{N_0-2}}
    \leqslant{}
    \frac{9}{16}
    \cdot
    \frac{\delta_0}{(4C_0)^{N_0-2}},\\
    &
    \begin{aligned}
    \| (a^{\varepsilon},\mathbb{Q}u^{\varepsilon}) \|_{Y_{q,r}^{\varepsilon,\alpha_{0}}(T_1,T)}^{\ell;\alpha_{0}}
    \leqslant{}
    &
    4C_0^2
    (\alpha_{0} \varepsilon)^{\frac{1}{r}}
    \| (a_0,u_0) \|_{D^{\varepsilon}}\\
    &
    +
    C_0
    (\alpha_{0} \varepsilon)^{\frac{1}{r}}
    \cdot 
    3\delta_0
    \cdot 
    (4C_0)^2\| (a_0,u_0) \|_{D^{\varepsilon}}
    \\
    \leqslant{}
    &
    \frac{1}{2}(4C_0)^2(\alpha_{0} \varepsilon)^{\frac{1}{r}}\| (a_0,u_0) \|_{D^{\varepsilon}},
    \end{aligned}
    \\
    &
    \begin{aligned}\label{cont-4-2}
    \| \mathbb{P}u^{\varepsilon}-w \|_{
    {L^{\infty}}(T_1,T;\dB_{q,1}^{\frac{d}{q}-1})
    \cap
    L^1(T_1,T;\dB_{q,1}^{\frac{d}{q}+1})}
    \leqslant{}&
    C_0
    \frac{\delta_0}{(4C_0)^{N_0-1}}
    +
    C_0
    \cdot 
    3\delta_0
    \cdot
    \frac{\delta_0}{(4C_0)^{N_0-2}}\\
    \leqslant{}&
    \frac{5}{16}
    \cdot
    \frac{\delta_0}{(4C_0)^{N_0-2}},
    \end{aligned}
\end{align}
which implies that
\begin{align}
    \| (a^{\varepsilon},u^{\varepsilon}-w) \|_{Y_{q,r}^{\varepsilon,\alpha_{0}}(T_1,T)}
    \leqslant{}
    \frac{7}{8}\cdot
    \frac{\delta_0}{(4C_0)^{N_0-2}}
    +
    \frac{1}{2}(4C_0)^2(\alpha_{0} \varepsilon)^{\frac{1}{r}}\| (a_0,u_0) \|_{D^{\varepsilon}}
\end{align}
and thus
\begin{align}\label{cont-2-2}
    \begin{split}
    A_{q,r}^{\varepsilon,\alpha_{0}}[a^{\varepsilon},u^{\varepsilon}-w](T_1,T)
    \leqslant{}&
    \frac{7}{8}\cdot
    \frac{\delta_0}{(4C_0)^{N_0-2}}
    +
    \frac{3}{2}(4C_0)^2(\alpha_{0} \varepsilon)^{\frac{1}{r(r-1)}}\| (a_0,u_0) \|_{D^{\varepsilon}}\\
    \leqslant{}&
    \frac{15}{16}\cdot
    \frac{\delta_0}{(4C_0)^{N_0-2}}.
    \end{split}
\end{align}
By Lemma \ref{lemm:lim}, we have 
\begin{align}
    \begin{split}
    &
    \| (a^{\varepsilon},\mathbb{Q}u^{\varepsilon}) \|_{{L^r}(T_1,T;\dB_{q,1}^{\frac{d}{q}-1+\frac{1}{r}})}
    +
    \| \mathbb{P}u^{\varepsilon}-w \|_{{L^{\infty}}(T_1,T;\dB_{q,1}^{\frac{d}{q}-1-\frac{1}{r}}) \cap L^1(
    T_1,T;\dB_{q,1}^{\frac{d}{q}+1-\frac{1}{r}})}\\
    &\quad
    \leqslant
    C_0\cdot4C_0 \varepsilon^{\frac{1}{r}}
    \| (a_0, u_0) \|_{D^{\varepsilon}}
    +
    C_0\cdot4C_0 \varepsilon^{\frac{1}{r}}
    \| (a_0, u_0) \|_{D^{\varepsilon}}\\
    &\qquad
    +
    C_0
    \varepsilon^{\frac{1}{r}}
    \cdot 
    2\delta_0
    \cdot 
    (4C_0)^2
    \| (a_0,u_0) \|_{D^{\varepsilon}}
    +
    C_0
    \cdot 
    3\delta_0
    \cdot 
    \| (a^{\varepsilon},u^{\varepsilon}-w) \|_{L^r(0,T;\dB_{q,1}^{\frac{d}{q}-1+\frac{1}{r}})}\\
    &\quad
    \leqslant
    \frac{7}{12}(4C_0)^2 \varepsilon^{\frac{1}{r}}
    \| (a_0, u_0) \|_{D^{\varepsilon}}\\
    &\qquad
    +
    \frac{1}{4}
    \left(
    \| (a^{\varepsilon},\mathbb{Q}u^{\varepsilon}) \|_{{L^r}(T_1,T;\dB_{q,1}^{\frac{d}{q}-1+\frac{1}{r}})}
    +
    \| \mathbb{P}u^{\varepsilon}-w \|_{{L^{\infty}}(T_1,T;\dB_{q,1}^{\frac{d}{q}-1-\frac{1}{r}}) \cap L^1(
    T_1,T;\dB_{q,1}^{\frac{d}{q}+1-\frac{1}{r}})}
    \right),
    \end{split}
\end{align}
which gives 
\begin{align}\label{cont-3-2}
    \begin{split}
    &
    \| (a^{\varepsilon},\mathbb{Q}u^{\varepsilon}) \|_{{L^r}(T_1,T;\dB_{q,1}^{\frac{d}{q}-1+\frac{1}{r}})}
    +
    \| \mathbb{P}u^{\varepsilon}-w \|_{{L^{\infty}}(T_1,T;\dB_{q,1}^{\frac{d}{q}-1-\frac{1}{r}}) \cap L^1(
    T_1,T;\dB_{q,1}^{\frac{d}{q}+1-\frac{1}{r}})}\\
    &\quad
    \leqslant
    \frac{7}{9}(4C_0)^2 \varepsilon^{\frac{1}{r}}
    \| (a_0, u_0) \|_{D^{\varepsilon}}.
    \end{split}
\end{align}
Hence, \eqref{cont-1-2}, \eqref{cont-4-2}, \eqref{cont-2-2} and \eqref{cont-3-2} contradict the definition of $T_2^*$ and we have $T_{\rm \max}^{\varepsilon} \geqslant T_2^*\geqslant T_2$.
Repeating the same procedure many times, we obtain that $T_{\rm max}^{\varepsilon} \geqslant T_{N_0}=\infty$, which implies that a unique global solution to \eqref{eq:re_comp-1} $(a^{\varepsilon},u^{\varepsilon})$ exists in the class \eqref{main:class}, and we also see that the global solution satisfies
\begin{align}
    &
    \begin{aligned}\label{T-1-n}
    \| (a^{\varepsilon},u^{\varepsilon}) \|_{X^{\varepsilon}(T_{n-1},T_{n})} 
    \leqslant{} 
    (4C_0)^n \| (a_0,u_0) \|_{D^{\varepsilon}},\quad
    \end{aligned}\\ 
    &\begin{aligned}\label{T-2-n}
    A_{q,r}^{\varepsilon,\alpha_{0}}[a^{\varepsilon},u^{\varepsilon}-w](T_{n-1},T_{n})
    \leqslant{}
    \frac{\delta_0}{(4C_0)^{N_0-n}},
    \end{aligned}\\
    &
    \begin{aligned}\label{T-3-n}
    \| \mathbb{P}u^{\varepsilon} - w \|_{
    {L^{\infty}}(T_{n-1},T_{n};\dB_{q,1}^{\frac{d}{q}-1})
    \cap
    L^1(T_{n-1},T_{n};\dB_{q,1}^{\frac{d}{q}+1})}
    \leqslant
    \frac{\delta_0}{(4C_0)^{N_0-n}},
    \end{aligned}\\
    &
    \begin{aligned}\label{T-4-n}
    &
    \|(a^{\varepsilon},\mathbb{Q}u^{\varepsilon})\|_{{L^r}(T_{n-1},T_{n};\dB_{q,1}^{\frac{d}{q}-1+\frac{1}{r}})}
    +
    \|\mathbb{P}u^{\varepsilon} - w \|_{{L^{\infty}}(T_{n-1},T_{n};\dB_{q,1}^{\frac{d}{q}-1-\frac{1}{r}}) \cap L^1(T_{n-1},T_{n};\dB_{q,1}^{\frac{d}{q}+1-\frac{1}{r}})}\\ 
    &\quad
    \leqslant (4C_0)^n \varepsilon^{\frac{1}{r}} \| (a_0,u_0) \|_{D^{\varepsilon}}.
    \end{aligned}
\end{align}
for all $n=1,2,...,N_0$ 
and there also holds
\begin{align}
    \varepsilon\|a^{\varepsilon}\|_{L^{\infty}(0,\infty;\dB_{q,1}^{\frac{d}{q}})} 
    \leqslant{}
    \frac{1}{2C_0},\qquad
    \varepsilon\|a^{\varepsilon}\|_{L^{\infty}(0,\infty;L^{\infty})}
    \leqslant{}
    \frac{1}{2},
\end{align}
which gives $\rho^{\varepsilon}(t,x) = 1 +\varepsilon a^{\varepsilon}(t,x) >0$ on $[0,\infty) \times \mathbb{R}^d$.

For the incompressible limit \eqref{main:incomp},
it follows from \eqref{T-4-n} that 
\begin{align}
    &
    \|(a^{\varepsilon},\mathbb{Q}u^{\varepsilon})\|_{{L^r}(0,\infty;\dB_{q,1}^{\frac{d}{q}-1+\frac{1}{r}})}
    +
    \|\mathbb{P}u^{\varepsilon} - w \|_{{L^{\infty}}(0,\infty;\dB_{q,1}^{\frac{d}{q}-1-\frac{1}{r}}) \cap L^1(0,\infty;\dB_{q,1}^{\frac{d}{q}+1-\frac{1}{r}})}\\ 
    &\quad
    \leqslant 
    \sum_{n=1}^{N_0}
    \left(
    \|(a^{\varepsilon},\mathbb{Q}u^{\varepsilon})\|_{{L^r}(T_{n-1},T_n;\dB_{q,1}^{\frac{d}{q}-1+\frac{1}{r}})}\right.\\
    &\quad \qquad \qquad\left.
    +
    \|\mathbb{P}u^{\varepsilon} - w \|_{{L^{\infty}}(T_{n-1},T_n;\dB_{q,1}^{\frac{d}{q}-1-\frac{1}{r}}) \cap L^1(T_{n-1},T_n;\dB_{q,1}^{\frac{d}{q}+1-\frac{1}{r}})}\right)\\ 
    &\quad 
    \leqslant
    \sum_{n=1}^{N_0}(4C_0)^n  \| (a_0,u_0) \|_{D^{\varepsilon}}\varepsilon^{\frac{1}{r}}\\ 
    &\quad 
    \leqslant
    (4C_0)^{N_0+1}  \| (a_0,u_0) \|_{D^{\varepsilon}}
    \varepsilon^{\frac{1}{r}}.
\end{align}

Finally, we prove \eqref{main:incomp-critical}.
Let $0 < \delta \leqslant \delta_0$ be arbitrary.
By the similar argument as above, 
there exists 
a number $N_{\delta} \in \mathbb{N}$, 
a time sequence $0=T_0^{\delta}<T_1^{\delta}<...<T_{N_{\delta}-1}^{\delta}<T_{N_{\delta}}^{\delta}=\infty$
and constants ${\alpha_{\delta}}>0$, $0<\varepsilon_{\delta}\leqslant \min \{ \varepsilon_0, 1/\alpha_{\delta},\beta_0/\alpha_{\delta} \}$
such that
\begin{gather}
    \| w \|_{L^r(T_{n-1}^{\delta},T_n^{\delta} ; \dB_{q,1}^{\frac{d}{q}-1+\frac{2}{r}})
    \cap 
    L^1(T_{n-1}^{\delta},T_n^{\delta} ; \dB_{q,1}^{\frac{d}{q}+1})}\leqslant \delta \qquad (n=1,2,...,N_{\delta}),\\
    \| a_0 \|_{\dB_{2,1}^{\frac{d}{2}}}^{h;\alpha_{0}}
    +
    \| (a_0, \mathbb{Q}u_0) \|_{\dB_{2,1}^{\frac{d}{2}-1}}^{h;\alpha_\delta}
    \leqslant
    \frac{\delta}{(4C_0)^{N_\delta}},\quad 
    \| (w \cdot \nabla) w \|_{L^1(0,\infty;\dB_{2,1}^{\frac{d}{2}-1})}^{h;\alpha_{\delta}}
    \leqslant
    \frac{\delta}{(4C_0)^{N_\delta}},\\
    (\alpha_{\delta} \varepsilon)^{\frac{1}{r(r-1)}}
    \| (a_0,\mathbb{Q}u_0) \|_{D^{\varepsilon}}
    \leqslant
    \frac{1}{24}
    \cdot
    \frac{\delta}{(4C_0)^{N_{\delta}}}
\end{gather}
for all $0 < \varepsilon \leqslant \varepsilon_{\delta}$.
Then, for $0 < \varepsilon \leqslant \varepsilon_{\delta}$, 
the similar continuation argument as above implies that the global solution $(a^{\varepsilon},u^{\varepsilon})$ to \eqref{eq:re_comp-1} constructed above satisfies
\begin{align}
    &
    \begin{aligned}\label{T-1-n-d}
    \| (a^{\varepsilon},u^{\varepsilon}) \|_{X^{\varepsilon}(T_{n-1}^{\delta},T_{n}^{\delta})} 
    \leqslant{} 
    (4C_0)^n \| (a_0,u_0) \|_{D^{\varepsilon}},\quad
    \end{aligned}\\ 
    &\begin{aligned}\label{T-2-n-d}
    A_{q,r}^{\varepsilon,\alpha_{\delta}}[a^{\varepsilon},u^{\varepsilon}-w](T_{n-1}^{\delta},T_{n}^{\delta})
    \leqslant{}
    \frac{\delta}{(4C_0)^{N_\delta-n}},
    \end{aligned}\\
    &
    \begin{aligned}\label{T-3-n-d}
    \| \mathbb{P}u^{\varepsilon} - w \|_{
    {L^{\infty}}(T_{n-1}^{\delta},T_{n}^{\delta};\dB_{q,1}^{\frac{d}{q}-1})
    \cap
    L^1(T_{n-1}^{\delta},T_{n}^{\delta};\dB_{q,1}^{\frac{d}{q}+1})}
    \leqslant
    \frac{\delta}{(4C_0)^{N_\delta-n}},
    \end{aligned}
\end{align}
for all $n=1,2,...,N_\delta$, which implies 
\begin{align}\label{lim-delta}
    \begin{split}
    &
    \|(a^{\varepsilon},\mathbb{Q}u^{\varepsilon})\|_{{L^r}(0,\infty;\dB_{q,1}^{\frac{d}{q}-1+\frac{2}{r}})}
    +
    \|\mathbb{P}u^{\varepsilon} - w \|_{{L^{\infty}}(0,\infty;\dB_{q,1}^{\frac{d}{q}-1}) \cap L^1(0,\infty;\dB_{q,1}^{\frac{d}{q}+1})}\\ 
    &\quad
    \leqslant 
    C\sum_{n=1}^{N_\delta}
    \left(
    \|(a^{\varepsilon},\mathbb{Q}u^{\varepsilon})\|_{{L^r}(T_{n-1}^{\delta},T_n^{\delta};\dB_{q,1}^{\frac{d}{q}-1+\frac{2}{r}})}\right.\\
    &\qquad \qquad \qquad\left.
    +
    \|\mathbb{P}u^{\varepsilon} - w \|_{{L^{\infty}}(T_{n-1}^{\delta},T_n^{\delta};\dB_{q,1}^{\frac{d}{q}-1}) \cap L^1(T_{n-1}^{\delta},T_n^{\delta};\dB_{q,1}^{\frac{d}{q}+1})}\right)\\ 
    &\quad
    \leqslant 
    C\sum_{n=1}^{N_\delta}
    \left(
    A_{q,r}^{\varepsilon,\alpha_{\delta}}[a^{\varepsilon},u^{\varepsilon}-w](T_{n-1}^{\delta},T_{n}^{\delta})\right.\\
    &\qquad \qquad \qquad\left.
    +
    \|\mathbb{P}u^{\varepsilon} - w \|_{{L^{\infty}}(T_{n-1}^{\delta},T_n^{\delta};\dB_{q,1}^{\frac{d}{q}-1}) \cap L^1(T_{n-1}^{\delta},T_n^{\delta};\dB_{q,1}^{\frac{d}{q}+1})}\right)\\ 
    &\quad 
    \leqslant
    C\sum_{n=1}^{N_\delta}\frac{\delta}{(4C_0)^{N_\delta-n}}\\ 
    &\quad 
    \leqslant
    C\delta
    \end{split}
\end{align}
for all $0 < \varepsilon \leqslant \varepsilon_\delta$.
Since $\delta$ may be chosen arbitrarily small, \eqref{lim-delta} implies \eqref{main:incomp-critical}.
Thus, we complete the proof.
\end{proof}


\noindent
{\bf Conflict of interest statement.}\\
The author has declared no conflicts of interest.

\noindent
{\bf Acknowledgements.} \\
This work was supported by Grant-in-Aid for JSPS Research Fellow, Grant Number JP20J20941.

\end{document}